\documentclass[10 pt]{article}
\usepackage[utf8]{inputenc}
\usepackage{amsmath}
\usepackage{amsthm}
\usepackage{amsfonts}
\usepackage{amssymb}
\usepackage{tabu}
\usepackage{geometry}
\usepackage{enumerate}
\usepackage[colorlinks=true,pagebackref]{hyperref}
\usepackage{mathrsfs}
\usepackage{titlesec} 
\usepackage{graphicx}
\usepackage{caption}
\usepackage{placeins}
\usepackage{tikz}
 \usetikzlibrary{arrows.meta, positioning, shapes, calc}

\titleformat{\section}[block]{\large\scshape\centering}{\thesection.}{1em}{} 
\titleformat{\subsection}[block]{\scshape\centering}{\thesubsection.}{1em}{} 

\newtheorem{theorem}{Theorem}[section]
\newtheorem*{theorem*}{Theorem}
\newtheorem{lemma}[theorem]{Lemma}
\newtheorem*{lemma*}{Lemma}
\newtheorem{proposition}[theorem]{Proposition}
\newtheorem{corollary}[theorem]{Corollary}
\newtheorem*{corollary*}{Corollary}

\theoremstyle{definition}
\newtheorem{definition}[theorem]{Definition}

\newtheorem{remark}[theorem]{Remark}

\usetikzlibrary{matrix,arrows,decorations.pathmorphing}
\usepackage{tikz-cd}
\usepackage[all]{xy}

\newcommand{\cC}{{\mathcal{C}}}

\newcommand{\cS}{{\mathcal{S}}}
\newcommand{\cF}{{\mathcal{F}}}

\newcommand{\IT}{{\mathbb{T}}}
\newcommand{\IC}{{\mathbb{C}}}
\newcommand{\ID}{{\mathbb{D}}}
\newcommand{\IH}{{\mathbb{H}}}
\newcommand{\IN}{{\mathbb{N}}}
\newcommand{\IR}{{\mathbb{R}}}

\newcommand{\IX}{{\mathbb{X}}}

\newcommand{\IZ}{{\mathbb{Z}}}
\newcommand{\IP}{{\mathbb{P}}}
\newcommand{\IE}{{\mathbb{E}}}

\newcommand{\be}{\begin{equation}}
\newcommand{\ee}{\end{equation}}

\newcommand{\diam}{\mathop{\rm{diam}}}

\DeclareMathOperator{\dist}{dist}

\renewcommand\Re{\operatorname{Re}}
\renewcommand\Im{\operatorname{Im}}

\newcommand\Length{\operatorname{Length}}

\newcommand{\Dop}{\mathcal{D}}

\DeclareMathOperator{\Crad}{Crad}
\DeclareMathOperator{\hcap}{hcap}
\DeclareSymbolFont{extraup}{U}{zavm}{m}{n}
\DeclareMathSymbol{\vardiamond}{\mathalpha}{extraup}{87}

\usepackage{setspace}

\numberwithin{equation}{section}
\title{\scshape{Power rate of convergence of discrete curves: framework and applications}}
\author{\scshape{Ilia Binder \footnote{Department of Mathematics, University of Toronto, Toronto, ON M5S 2E4 Canada;\newline e-mail: ilia@math.toronto.edu} and Larissa Richards} \footnote{School of Mathematics, University of Leeds, Leeds,  LS2 9JT United Kingdom; e-mail: l.richards@leeds.ac.uk}}

\begin{document}
\maketitle

\begin{abstract}
We develop a framework for power-law rates of convergence of random discrete curves to Schramm Loewner Evolution (SLE). We show that a power-law rate of convergence of a suitable martingale observable gives a power-law rate of convergence of the interface to SLE, provided the discrete curves satisfy the Kemppainen-Smirnov crossing condition; the proof uses known estimates on the derivative of the SLE map. We apply the framework to the exploration process of critical site percolation on the triangular lattice, obtaining a power-law rate of convergence to SLE$_6$.
\end{abstract}

 \singlespacing {\footnotesize\tableofcontents}




\section{Introduction and Background}

\subsection{Introduction}

Introduced by Oded Schramm \cite{Schramm}, SLE is a one-parameter family of conformally invariant random curves in simply-connected planar domains.
It is conjectured that these curves are the scaling limits of the various interfaces in critical planar lattice models of Statistical Physics.
These two-dimensional lattice models describe a variety of physical phenomena, including percolation, the Ising model, loop-erased random walk, and the Potts model.

For several two-dimensional lattice models at criticality, it has been shown that the discrete interfaces converge in the scaling limit to SLE curves \cite{LSW,lawler2016convergence,SS,MR2486487,MR3151886,MR2227824,SLEI,SLEII}.
These proofs begin by identifying the scaling limit of an interface observable.
The limit is constructed from the interface itself through conformal invariance.
In this paper, we study the rate of convergence mentioned above. We obtain a power-law convergence rate to an SLE curve from a power-law convergence rate for a martingale observable under suitable conditions on the discrete curves.

The chordal Loewner equation describes a family of conformal maps onto decreasing simply-connected domains, governed by a continuous real-valued driving term. When the evolution is generated by a curve, this term encodes its growth; we recall the precise definitions in Section~\ref{LEandSLE}.

\emph{Schramm Loewner evolutions (SLE)} are the one-parameter family of random fractal curves in a reference domain (either unit disk $\ID$ or the upper half-plane $\IH$) whose  Loewner evolution is driven by a scaled standard one-dimensional Brownian motion. It is proven in \cite{MR2883396,LSW} that the scaled Brownian motion generates a Loewner curve a.s.

Existing proofs typically proceed either by proving convergence of the
driving processes and then of the curves, or by first proving
precompactness of the measures on curves.


 Given a discrete random curve that is expected to have a scaling limit described by a variant of SLE, Kemppainen and Smirnov in \cite{KS} provide a framework for both approaches, building upon the earlier work of Aizenman and Burchard \cite{AB}. They show that, under conditions similar to those in \cite{AB}, an interface admits subsequential scaling limits that are described almost surely by Loewner evolutions.

In \cite{V}, Viklund develops a framework for deriving a power-law convergence rate to an SLE curve from the convergence rate of the driving function, given geometric data on crossing events and a bound on the growth of the derivative of the SLE map. To obtain this geometric data, he introduces the \emph{tip structure modulus}, which measures the maximal distance a Loewner curve travels into a fjord with an opening less than $\epsilon$, viewed from the curve's growth direction. The framework is demonstrated for the loop-erased random walk, but the required estimate on the tip structure modulus can be hard to establish for other models.

We show that the KS condition implies the tip-structure estimate required in Viklund's framework \cite{V}. Together with polynomial convergence of a suitable martingale observable, this gives a power-law path rate to SLE$_\kappa$, $\kappa\in(0,8)$. Earlier quantitative LERW results include driving and Hausdorff estimates \cite{BJVK}, path estimates \cite{V}, and capacity-parameterized convergence \cite{LV2022}. Rates for Cardy's formula were obtained in \cite{BCL} and, for piecewise analytic Jordan domains, in \cite{MNW2014}.

We apply the above framework to establish the power-law rate of convergence of the percolation exploration process, specifically critical site percolation on the triangular lattice, represented by hexagonal faces, to SLE$_6$. In a subsequent paper with D. Chelkak \cite{HarmExplIsing}, we apply the framework to three additional models: the Harmonic Explorer, the Ising model, and the FK-Ising model. These applications involve a novel estimate similar to \cite{chelkak2020ising}. In the radial setting, Theorem~\ref{mainthmradial}, together with \cite[\S4.5]{KS} and \cite{BJVK}, gives a power-law rate for loop-erased random walk up to $\sigma_\varepsilon$. The exponent produced by our argument is explicit in the admissibility exponent, $\kappa$, and the KS crossing exponent, although these model-dependent inputs are not explicit in the percolation application.
\subsection*{Acknowledgments}
	We thank Dmitry Beliaev, Dmitry Chelkak, and Fredrik Viklund for helpful discussions, and the anonymous referees for comments that improved the paper.
\subsection{The Setup and Assumptions}
Our paper is concerned with sequences of random planar curves arising as an interface (a domain wall boundary) in some lattice model of statistical physics or a self-avoiding random trajectory on a lattice, and conditions sufficient to establish a power-law rate of convergence to a continuous curve as the lattice mesh is taken to zero.

Let $\cS(\IC)$ be the space of planar curves $\gamma:[0,1]\to\IC$ modulo reparameterization, endowed with the supremum metric $d_{\cS}([f],[g]) = \inf\{\|f_0-g_0\|_\infty\}$ over representatives $f_0\in[f]$, $g_0\in[g]$; $(\cS,d_\cS)$ is a metric space (\cite[Lemma 2.1]{AB}). We start with a probability measure $\IP \in \mathrm{Prob}(\cS(\IC))$, equipped with the Borel $\sigma$-algebra and the weak topology induced by bounded continuous functions; since these spaces are metrizable, convergence in $\mathrm{Prob}(\cS)$ reduces to sequences $(\IP_n)$. In applications we typically take $\IP_n$ supported on curves of the same lattice model at mesh $\delta_n \searrow 0$; see \cite{KS} for this probability structure. For a domain $D \subset \IC$, write $\cS_{\mathrm{simple}}(D)$ for the simple curves in $D$, and $\cS_{\mathrm{simple}}(D,a,b)$ for those with endpoints $a, b$.

\begin{remark}\label{rem:DMP}
Let $\mathfrak S_t^n$ be the discrete exploration state: the remaining graph, its induced boundary data, and the current and target vertices. The exploration filtration $({\mathcal F}_t^n)$ includes all revealed data. The \emph{domain Markov property} means that, at every stopping time $\tau$ at a vertex,
$$
\mathrm{Law}\bigl(\gamma^n|_{[\tau,1]}\mid{\mathcal F}_\tau^n\bigr)
=\IP_n^{\mathfrak S_\tau^n}\qquad\text{a.s.},
$$
where the right-hand side is the exploration law started in that state, with the inherited parameterization. These laws are defined on all reachable states. Conformal maps refer to the geometric slit, which need not coincide with the unrevealed lattice region.
\end{remark}

The case $b\in\partial D$ corresponds to Chordal Loewner Evolution (see Section \ref{LEandSLE}). The curves we consider are simple by construction, since the exploration rules of the models at hand forbid self-crossings; they satisfy the Loewner equation.

We uniformize by the reference domains $\ID=\{|z|<1\}$ and $\IH = \{\mathrm{Im}\,z>0\}$ (the choice is immaterial, since our conditions are conformally invariant), via a conformal map ${\mathfrak{p}}:D\rightarrow\ID$ with ${\mathfrak{p}}(a)=-1, \; {\mathfrak{p}}(b) = 1$ (see \cite[\S 2.4, \S 2.5]{pomm} for prime ends). Following \cite{KS}, our main object of study is a family $(\IP_{(D;a,b)})$ of probability measures, each supported on a closed subset of the curves $\gamma \in \cS_{\mathrm{simple}}(D;a,b)$ whose image $\mathfrak p(\gamma)$ joins $-1$ to $1$; we assume this set of curves is nonempty, which holds quite generally (\cite[Cor.\ 2.17]{MR1217706}).

\subsubsection{Carath\'eodory Distance and Discrete Approximation }\label{subsec:graph_approx}

We use the following definition of the Carath\'eodory distance.
\begin{definition}
\label{def:crosscut}
A \emph{crosscut} of a bounded domain $D$ is an open Jordan arc $\mathcal{C}$ such that $\mathcal{C}\subset D$ and $\overline{\mathcal{C}}=\mathcal{C}\cup\{p_1,p_2\}$ where $p_j\in\partial D, \; j=1,2$.
\end{definition}

\begin{definition}\label{def:carametric}
Let $D$ be a bounded simply-connected domain and $c\in D$. The \emph{Carath\'eodory distance} between two points $z,w\in D\setminus\{c\}$ is defined by
$$\dist_C(z,w):=\inf\left\{\diam(\gamma):\ \gamma\text{ separates }z,w\text{ from }c\right\},$$
where $\gamma$ ranges over crosscuts of $D$ and Jordan curves in $D$. For a Jordan curve, $z,w$ must lie in its bounded complementary component and $c$ in its unbounded complementary component. For $z\ne c$, set
$$\dist_C(z,c)=\dist_C(c,z):=\lim_{w\to c}\dist_C(z,w),\qquad \dist_C(c,c):=0.$$
\end{definition}
This is a metric and extends continuously to the prime-end compactification of $D$ \cite[Theorem~3.5]{BRY2014}.

Let us now fix two accessible prime ends of $D$, $a$ and $b$. Let us also fix $\mathfrak{p}:D\rightarrow\ID$ with $\mathfrak{p}(a)=-1, \; \mathfrak{p}(b) = 1$, and $c:=\mathfrak{p}^{-1}(0)$.

We use very general discrete approximations of the domain: a finite embedded planar graph $\Gamma^n=(V^n, E^n)$ with vertices $V^n\subset\overline D$ and straight edges $E^n$, whose faces' interior union $D^n$ is a simply-connected subdomain of $D$.

\begin{definition}\label{def:approx}
   The domain $D^n\subset D$, created from a graph $\Gamma^n$ with two fixed vertices $a^n,\ b^n\in\partial D^n\cap V^n$, each with a specified prime-end side, is called a \emph{$1/n$-discrete approximation of  $(D,a,b)$} if
   \begin{enumerate}
    \item $c\in D^n$
   \item For any edge $e\in E^n$, $e$ is a straight segment with $\diam(e)\leq1/n$.
   \item For any face $f$ of $\Gamma^n$, $\diam(f)\leq1/n$.
   \item For any $z\in \partial D^n$ one can find $w\in\partial D$ and a simple curve $\chi:[0,1]\to \overline{D}$ joining $z$ and $w$ in $D\setminus D^n$ ($\chi(0)=z$, $\chi(1)=w$, $\chi((0,1))\subset D\setminus D^n$) with $\diam(\chi)\leq1/n$; if $z\in\partial D^n\cap\partial D$, the constant connection $w=z$, $\chi\equiv z$ is also allowed.
   \item $a^n$ and $b^n$, approached from their specified prime-end sides, are $1/n$ close to $a$ and $b$ respectively in the Carath\'eodory distance, i.e. one can find two crosscuts of $D$, $\chi_a$, $\chi_b$ of diameter at most $1/n$ such that $\chi_a$ separates $a$ and $a^n$ from $c$ and $\chi_b$ separates $b$ and $b^n$ from $c$.
   \end{enumerate}
\end{definition}

Conditions 1, 2 and 4 already provide marked vertices satisfying condition 5, with $4/n$ in place of $1/n$, once $4/n<\dist(c,\partial D)$. Indeed, fix a connection as in condition 4 at each boundary vertex, and let $e_1,\dots,e_k$ be the edges of the positively oriented boundary walk of $D^n$, moved slightly inward if it meets $\partial D$, with the connections extended along short segments. Let $L_j$ join the connections at the endpoints of $e_j$ through $e_j$. Suppose no connection lands at $a$, and close each $\mathfrak p(L_j)$ along the arc of $\partial\ID\setminus\{-1\}$ between its endpoints. The winding numbers of these curves about $0$ sum to one: the connections cancel, the image of the walk winds once about $0=\mathfrak p(c)$ as $c\in D^n$, and the closing arcs have total winding zero. So one curve winds about $0$ but not about points near $-1$, and its $L_j$ separates $a$ from $c$. For $r\in(3/(2n),2/n)$ and the midpoint $m_j$ of $e_j$, conditions 2 and 4 give $L_j\subset B(m_j,r)$, while $c\notin\overline B(m_j,r)$, as $\dist(m_j,\partial D)\le 3/(2n)$. By simple connectivity, a component of $\partial B(m_j,r)\cap D$ is a crosscut separating the connected set $L_j\cap D$ from $c$. It has diameter less than $4/n$ and separates from $c$ both $a$ and the endpoints of $e_j$, either of which can serve as $a^n$. If a connection lands at $a$, use an edge at its vertex: the connection and its landing point lie in the open disk, so the crosscut again separates $a$ from $c$. The argument for $b$ is identical.

Let $\mathfrak{p}^n: D^n\to\ID$ be the conformal map with $\mathfrak{p}^n(a^n)=-1$ and $\mathfrak{p}^n(b^n)=1$ normalized so that $|\mathfrak{p}^n(c)|$ is minimal among such maps; then $\mathfrak{p}^n \to \mathfrak{p}$ uniformly on compacts, with $\mathfrak{p}^n(c)\to \mathfrak{p}(c) = 0$ (see \cite{BRY}).

A \emph{discrete curve} in $(D^n, a^n, b^n)$ will be a simple curve that consists of edges from $E^n$, starts at $a^n$, ends at $b^n$, and contains no other boundary vertices, i.e., no vertices $v\in V^n\cap \partial D^n$. Every discrete curve is a polygonal path of step at most $1/n$.

\subsubsection{The KS Condition}

Let $D^n$ be a $1/n$-discrete approximation of $D$. Since annuli separating $a^n$ from $b^n$ must be crossed, we impose a bound only on crossings avoidable by deforming the curve.

\begin{figure}[!ht]
  \centering
  \resizebox{0.55\textwidth}{!}{
  \begin{minipage}{0.38\textwidth}
    \centering
    \begin{tikzpicture}
  \draw[thick] (0,0) rectangle (3,2);

  \coordinate (a) at (0,0.75); 
  \fill (a) circle (2pt);
  \node[left] at (a) {$a^n$};

  \coordinate (b) at (3,1.25); 
  \fill (b) circle (2pt);
  \node[right] at (b) {$b^n$};

  \fill[gray!50] (2,0) arc[start angle=0, end angle=180, radius=0.75cm] -- cycle;
  \draw[thick, densely dotted] (2,0) arc[start angle=0, end angle=180, radius=0.75cm] -- cycle;
   \fill[white!50] (1.5,0) arc[start angle=0, end angle=180, radius=0.25cm] -- cycle;
   \draw[thick, densely dotted] (1.5,0) arc[start angle=0, end angle=180, radius=0.25cm] -- cycle;

  \draw[dashed, ->, thick] (a) .. controls (0.75,1) .. (1.25,0.25); 

\end{tikzpicture}

    \caption*{(a)}
  \end{minipage}
  \hfill
  \begin{minipage}{0.38\textwidth}
    \centering
   \begin{tikzpicture}
  \draw[thick] (0,0) rectangle (3,2);

  \coordinate (a) at (0,0.75); 
  \fill (a) circle (2pt);
  \node[left] at (a) {$a^n$};

  \fill[gray!50] (3,2) arc[start angle=90, end angle=270, radius=0.75cm] -- cycle;
  \draw[thick, densely dotted] (3,2) arc[start angle=90, end angle=270, radius=0.75cm] -- cycle;
   \fill[white!50] (3,1.5) arc[start angle=90, end angle=270, radius=0.25cm] -- cycle;
  \draw[thick, densely dotted] (3,1.5) arc[start angle=90, end angle=270, radius=0.25cm] -- cycle;

  \coordinate (b) at (3,1.25); 
  \fill (b) circle (2pt);
  \node[right] at (b) {$b^n$};

  \draw[dashed, ->, thick] (a) .. controls (0.75,1.25) .. (2.75,1.25); 

\end{tikzpicture}

    \caption*{(b)}
  \end{minipage}
}
\captionsetup{font=footnotesize}
  \caption{Annulus crossings: (a) unforced (avoidable by deforming the curve; $A^{D^n}$ is the full half-annulus); (b) forced ($A^{D^n}=\emptyset$).}\label{fig:annulus-crossings}

\end{figure}

\enlargethispage{\baselineskip}
\begin{definition}\label{def:unforcros} A curve $\gamma^n$ \emph{crosses} an annulus $A(z_0,r,R) = B(z_0,R)\backslash \overline{B(z_0,r)}$ if it intersects both its inner and outer boundaries $\partial B(z_0,r)$ and $\partial B(z_0,R)$. An \emph{unforced crossing} is a crossing that can be avoided by deforming the curve inside $D^n$. That is, the crossing occurs along a sub-arc of $\gamma^n$ contained in a connected component of $A(z_0,r,R)\cap D^n$ that does not disconnect $a^n$ and $b^n$. We can define this set $A^{D^n}$ as follows. If $\partial B(z_0,r)\cap\partial D^n=\emptyset$ define $A^{D^n}=\emptyset$, otherwise
$$
A^{D^n} =
\left\{
\begin{aligned}
 z \in D^n \cap A : & \text{ the connected component of } z \in D^n \cap A \\
&  \text{does not disconnect $a^n$ from $b^n$ in } D^n
\end{aligned}
\right\}.$$
See Figure~\ref{fig:annulus-crossings}.
\end{definition}

\FloatBarrier
 For a simple curve $\gamma^n$ in $D^n$ starting from $a^n$, write $D^n_\tau = D^n\backslash\gamma^n[0,\tau]$. Define the dynamic version $A^{D^n}_\tau$ of $A^{D^n}$: it is $\emptyset$ if $\partial B(z_0,r)\cap\partial D^n_\tau = \emptyset$, and otherwise
 \[
A^{D^n}_\tau =
\left\{
\begin{aligned}
 z \in D^n_\tau \cap A : & \text{ the connected component of } z \in D^n_\tau \cap A \\
&  \text{does not disconnect $\gamma^n(\tau)$ from $b^n$ in } D^n_\tau
\end{aligned}
\right\}
\]
(a connected set $C$ disconnects $\gamma^n(\tau)$ from $b^n$ if it separates a neighbourhood of $\gamma^n(\tau)$ from a neighbourhood of $b^n$ in $D^n_\tau$). If $\gamma^n[\tau,1]$ contains a crossing of $A$ contained in $\overline{A^{D^n}_\tau}$, $\gamma^n$ makes an \emph{unforced crossing} of $A$ in $D^n_\tau$.

\begin{definition}
\label{cond:KS}
A family of random curves $\left(\gamma^n\right)$ satisfies a \emph{geometric bound on unforced crossings} or \emph{KS Condition}, if there exists $C>1$ such that for any $n\in\IN$, for any stopping time $0\leq \tau\leq 1$ and for any annulus $A(z_0,r,R)$ where $0<Cr<R$
the event $E^n_\tau(A)$ that $\gamma^n[\tau,1]$ makes a crossing of $A$ contained in $\overline{A^{D^n}_\tau}$ satisfies
$$\IP^n\bigl(E^n_\tau(A)\mid\mathcal F^n_\tau\bigr)<1/2\quad\text{a.s.}$$
 We refer to the constant $C$ as a \emph{KS modulus bound}.
\end{definition}

We use the same notation for curves in other domains, with the corresponding domain and endpoints.

As we show in Proposition \ref{tip}, this bound controls how far a discrete curve can penetrate into a narrow fjord. There is an equivalent, conformally invariant formulation in terms of topological quadrilaterals, stable under the domain Markov property and preserved by uniformization to $\ID$ (\cite[Prop.\ 2.6, Condition C2]{KS}); the (equivalent) condition has been verified for the FK (random-cluster) models with $1 \le q \le 4$ \cite{DST}, and for Ising, percolation, harmonic explorer, and chordal loop-erased random walk, but fails for the uniform spanning tree \cite{KS}.

\subsection{Sketch of Proof of Main Theorem}
In this section, we state the main theorem and outline its proof. The admissibility conditions are defined in Section \ref{sec:mainthmsetup} and can be summarized as follows: a family of discrete curves is \emph{weakly $\kappa$-admissible} if it admits a multivariable discrete observable $H^n$ such that
\begin{itemize}
\item[(i)] $H^n$ is an \emph{almost martingale} along the curve, up to error $n^{-s}$ (Definition \ref{def:family}, condition \ref{cond:1});
\item[(ii)] $H^n$ converges, at the same rate $n^{-s}$, to a continuum observable $\mathfrak{h}$ solving the BPZ equation \eqref{eq:BPZ} for the target parameter $\kappa$ (condition \ref{cond:2b});
\item[(iii)] $\mathfrak{h}$ is \emph{non-degenerate} (condition \ref{cond:2a}, Definition \ref{def:nondeg});
\end{itemize}
and it is (fully) \emph{$\kappa$-admissible} if, in addition, the discrete curves satisfy a geometric bound on unforced annulus crossings, the \emph{KS Condition} of Definition \ref{cond:KS}.

Write $\Phi(z)=i(1+z)/(1-z)$, $\varphi=\Phi\circ\mathfrak p$, $\varphi^n=\Phi\circ\mathfrak p^n$, and $\Crad(\varphi):=|(\varphi^{-1})'(i)|$. In the following theorem, $\tilde\gamma^n:=\varphi^n(\gamma^n)$, and $\gamma^n$ is parameterized by the capacity of its image.

\begin{theorem}[Main Theorem]\label{mainthm}
Let $\{\gamma_{(D,a,b)}^n\}$ be a family of $\kappa$-admissible probability measures for some $\kappa\in(0,8)$. Let $D$ be a simply-connected domain with two boundary prime ends $a$ and $b$ and the corresponding conformal map $\varphi: (D,a,b) \to (\IH, 0,\infty)$, $\Crad(\varphi) > R$ and $\diam(D) \le R_1 < \infty$. Let $T \ge 1$ and $\rho \in (0,1)$, and set
\[
T_\rho := T \wedge \inf\bigl\{ t \ge 0 : d_*\bigl(\tilde\gamma(t), \infty\bigr) \le \rho \bigr\},
\]
where $\tilde\gamma$ denotes the chordal SLE$_\kappa$ path in $\IH$ driven by $\sqrt{\kappa}B(t)$, parameterized by capacity, and $d_*$ is the pullback to $\overline{\IH}\cup\{\infty\}$ of the Euclidean metric on $\overline{\ID}$ under the Cayley map (defined precisely in Section \ref{sec:mainthmsetup}). Then there exist $u > 0$ and $N < \infty$, and a coupling of $\tilde\gamma^n$ with the Brownian motion $\sqrt{\kappa}B(t)$, such that for all $n \ge N$,
\[
\IP \left(\sup_{t\in[0,T_{\rho}]} d_* \bigl(\tilde\gamma^{n}(t),\tilde\gamma(t)\bigr) > n^{-u} \right) < n^{-u},
\]
with $\tilde\gamma^n$ also parameterized by capacity. The exponent $u$ depends only on the admissibility data (in particular $s$), the KS modulus bound, and --- through the derivative estimate of Proposition \ref{derest1} --- on $\kappa$; it does not depend on $D$, $T$, or $\rho$. The threshold $N$ depends in addition on $T$, $\rho$, $R$, and $R_1$.

The Radial analogue is established in Appendix \ref{app:radial} (Theorem \ref{mainthmradial}): $d_*$ is replaced by the usual distance in $\ID$, no target-proximity horizon is needed and $N$ does not depend on $\rho$, but the supremum is taken up to the stopping time $\sigma_\varepsilon$ of the radial SLE derivative estimate. Its H\"older-domain conclusion has the same form, with $[0,\sigma_\varepsilon]$ in place of $[0,T_\rho]$.

For the Chordal theorem, suppose $D$ is an $\alpha$-H\"older domain, meaning that $\psi=\mathfrak p^{-1}$ extends to an $\alpha$-H\"older map on $\overline\ID$, $0<\alpha\le1$. Then under the same coupling
\[
\IP \left( \sup_{t\in[0,T_{\rho}]} \bigl| \gamma^{n}(t) - \psi\bigl(\Phi^{-1}(\tilde\gamma(t))\bigr) \bigr| > n^{-\lambda} \right) < n^{-\lambda}, \qquad n \ge N,
\]
where
\[
\lambda := \tfrac12 \min\bigl\{ \beta_{\mathrm{map}},\ \alpha\, u \bigr\},
\]
with $\beta_{\mathrm{map}} = \alpha^2/100$ the domain-map convergence exponent of Lemma \ref{lem:discrete}. For this last estimate the threshold $N$ depends in addition on $\alpha$, on the $\alpha$-H\"older norm of $\psi$, and on $\diam(D)$ and $|\psi'(0)|$, through the threshold of Lemma \ref{lem:discrete}.
\end{theorem}
\begin{figure}[!ht]
    \centering
\resizebox{\linewidth}{!}{%
\begin{tikzpicture}[
  >=Stealth,
  every node/.style={font=\footnotesize},
  obj/.style={draw, rounded corners, fill=blue!6, align=center,
              inner sep=4pt, minimum width=44mm},
  strong/.style={-Stealth, thick},
  comb/.style={-Stealth, thick, bend left=10},
  note/.style={font=\footnotesize\itshape, midway, anchor=west, xshift=4pt}
]

\node[obj] (A) at (0,0) {1. Aizenman--Burchard condition (Hypothesis H1)};
\node[obj] (B) [below=8mm of A] {2. KS Condition (Definition \ref{cond:KS})};
\node[obj] (B2) [below=8mm of B] {2$'$. Quantitative unforced-crossing bound};
\node[obj] (C) [below=8mm of B2] {3. Bound on tip structure modulus};
\node[obj] (D) [below=14mm of C] {4. Weak $\kappa$-admissibility\\(almost martingale observable) Definition \ref{def:family}};
\node[obj] (E) [below=8mm of D] {5. Power-law convergence\\of the driving function};
\node[obj] (F) [below=8mm of E] {6. Power-law convergence\\of the interface (model domain)};
\node[obj] (G) [below=8mm of F] {7. H\"older condition for the domain};
\node[obj] (H) [below=8mm of G] {8. Power-law convergence\\of the interface (domain)};
\node[obj] (D2) [right=13mm of E] {4$'$. SLE derivative bound};
\draw[strong] (B) to node[note] {Proposition \ref{KStoAB}} (A);
\draw[strong] (B) to node[note] {Lemma \ref{lem:G3}} (B2);
\draw[strong] (B2) to node[note] {Proposition \ref{tip}} (C);
\draw[-Stealth, thick, bend right=25] (A.west) to node[note, anchor=east, xshift=-4mm] {Proposition \ref{tip}} (C.west);
\draw[strong] (D) to node[note] {Theorem \ref{drivingcvg}} (E);
\draw[-Stealth, thick, bend left=45]   (C.east) to node[note] {Section \ref{subsec:mainthm}} (F.east);
\draw[comb]   (E.east) to node[note] {} (F.east);
\draw[comb] (D2.south) to node[note] {Proposition \ref{derest1}} (F.east);
\draw[comb] (F.east) to node[note] {Section \ref{subsec:mainthm}} (H.east);
\draw[comb] (G.east) to node[note] {} (H.east);
\end{tikzpicture}
}
    \caption{The structure of the paper: the main implications, following the chain \emph{observable rate $\Rightarrow$ driving rate}; \emph{KS $\Rightarrow$ quantitative unforced-crossing bound, and H1 $\Rightarrow$ tip bound}; \emph{driving rate + tip bound + SLE derivative bound $\Rightarrow$ curve rate}.}\label{fig:proof-structure}
\end{figure}
\noindent\emph{Sketch of Proof}. Figure~\ref{fig:proof-structure} shows the main implications.
The proof separates into two parts. First, the martingale-observable hypotheses yield a power-law coupling of the driving function with
Brownian motion; this extends the scheme of \cite{BJVK} to the abstract setting of Definition~\ref{def:family}. Second, the KS
Condition yields a quantitative tip-structure-modulus bound, which combined with the SLE derivative estimate gives a power-law bound for
the curve. The final transfer to physical domains uses the polynomial convergence of the uniformizing maps for $\alpha$-H\"older domains.

\subsection{Loewner Evolution and SLE}
\label{LEandSLE}
We fix notation for Loewner evolution and SLE curves, referring to \cite{Beliaev2019ConformalMA} and \cite{MR2129588} for background.

We call a bounded set $K\subset\IH$ a \emph{hull} if $K$ is relatively closed in $\IH$ and $\IH\setminus K$ is connected and simply-connected. (Thus $\overline{K}$ is compact in $\overline{\IH}$ and may meet $\IR$; $K$ itself need not be compact in the open half-plane.) For each hull $K$ there exists a unique conformal map $f_K\,:\,\IH\to\IH\setminus K$ satisfying \emph{hydrodynamic normalization} $\lim_{z\to\infty}(f_K(z)-z)=0$; write $g_K:=f_K^{-1}$, and set the \emph{half-plane capacity}
$$\hcap(K):=\lim_{z\to\infty}{z(z-f_K(z))}=\lim_{z\to\infty}{z(g_K(z)-z)}. $$

A \emph{Chordal Loewner pair} or \emph{$\IH$-Loewner pair} $(f,W)$ consists of a function $f_t(z)=f(t,z),\ t>0,\ z\in\IH$ and a continuous function $W(t)\,:\,\IR_+\to\IR$ satisfying the following equation
\begin{align}
 \partial_t f_t = -f'_t \frac{2}{z - W_t} \;, \; f_0(z)\equiv z \label{rLE}.
\end{align}
The functions $f_t$ are conformal maps of $\IH$ onto some simply-connected domains $D_t:=\IH\setminus K_t$ for a growing family of compact hulls $(K_t)$ with $\hcap(K_t)=2t$, satisfying the hydrodynamic normalization. The inverse maps $g_t$ satisfy the simpler ODE
\begin{equation}
     \partial_t g_t(z) = \frac{2}{g_t(z) - W_t} \;, \; t\in[0,T]. \label{LDE}
\end{equation}
 We call $W_t$ the \emph{driving function} of the Chordal Loewner Chain.

A \emph{Loewner curve} is a continuous curve $\gamma$ in $\overline{\IH}$ generating a Loewner chain, so that $\IH\setminus K_t$ is the unbounded component of $\IH\setminus\gamma[0,t]$. We use the capacity parameterization $\hcap(K_t)=2t$. We recover its trace from
\[
\gamma(t)=\lim_{d\to0^+}f(t,W_t+id).
\]

We will use the following observation, which is Lemma 2.1 of \cite{LSW}.
\begin{lemma}[Diameter bounds on $K_t$]\label{DiamKt}
 There is a constant $C>0$ such that the following always holds. Let $W: [0,\infty)\rightarrow\IR$ be continuous and let $(K_t, \; t\geq 0)$ be the corresponding hull for Loewner's Chordal equation with driving function $W$. Set $k(t):= \sqrt{t} + \max \left\{ |W(s) - W(0)| \; : \; s\in[0,t] \right\}$, then
 \[
  \forall t\geq 0 \;\; C^{-1}k(t) \leq \diam(K_t) \leq C k(t).
 \]
\end{lemma}

Let $D$ be a simply-connected domain with two boundary prime ends $a,\ b$. $\gamma$ is a \emph{Loewner curve in $(D,a,b)$} if it is an image of a Loewner curve in $(\IH,0,\infty)$ under a conformal map $\psi:(\IH,0,\infty)\to(D,a,b)$. A \emph{chordal SLE$_\kappa$} is the Loewner chain driven by $W_t=\sqrt{\kappa}B_t$. In a domain $(D,a,b)$ it is defined by conformal covariance.

For $\kappa\ne8$ the SLE trace is generated by a continuous curve \cite{MR2883396}; the case $\kappa=8$ follows from \cite{LSW}.
\begin{theorem}[Rohde-Schramm, Lawler-Schramm-Werner] \label{Rohde-Schramm}
Almost surely the SLE$_\kappa$ hull is generated by a continuous trace $\gamma$; moreover, $\gamma$ is simple for $0<\kappa\le4$, non-simple for $4<\kappa<8$, and space-filling for $\kappa\ge8$.
\end{theorem}

We use the derivative estimates of \cite{LJV2011} in the uniform-in-time form of \cite[Proposition~A.4]{V}.
Suppose $\kappa>0$, let
\begin{align*}
 \lambda_C & = 1 + \frac{2}{\kappa} + \frac{3\kappa}{32}, \\
 q(\beta) & = \min\left\lbrace \lambda_C \beta, \beta + \frac{2(1+\beta)}{\kappa} + \frac{\beta^2\kappa}{8(1+\beta)} - 2 \right\rbrace, \\
 \mathrm{ and } \; \beta_+ & = \max \left\lbrace 0 , \frac{4(\kappa\sqrt{8+\kappa} - (4 - \kappa) )}{(4+\kappa)^2} \right\rbrace.
\end{align*}

\begin{proposition}\label{derest1}
 Let $T<\infty$ be fixed, $\kappa\in (0,8)$, and let $(f_t)$ be the Chordal SLE$_\kappa$ Loewner chain, that is, the solution of \eqref{rLE} with $W=\sqrt{\kappa}B$, so that $f_t=g_t^{-1}$. Let $\beta\in(\beta_+, 1)$ and $q<q(\beta)$. There exists a constant $0<c<\infty$ depending only on $T,\kappa,\beta,q$ such that for every $y_* < 1$
 \[
  \IP\left\lbrace \forall y\leq y_*,~ \sup_{t\in[0,T]} y~\left|f'(t,W(t)+iy)\right| \leq cy^{1-\beta} \right\rbrace \geq 1-cy_*^q.
 \]

\end{proposition}

\section{Sufficient Conditions on Martingale Observable}\label{sec:mgobs}

We next formulate the martingale-observable assumptions. Suppose that for every simply-connected domain $D$ with marked points $a_1,\cdots,a_n\in\overline D$, we have a random Loewner curve $\gamma$ starting at $a_1$.

\begin{definition}\label{def:cov} A function $F(D;a_1,\cdots,a_n)(z): D\rightarrow \IC$, indexed by simply-connected domains with marked points $a_1,\cdots, a_n\in\overline{D}$ and evaluated at $z\in D$, is a \emph{conformally covariant martingale observable} for a random Loewner curve $\gamma$ if
\begin{itemize}
    \item $F$ is \emph{conformally covariant}: there exist weights
 $\alpha, \alpha', \beta_1, \beta_1', \cdots, \beta_n, \beta_n' \geq 0$ such that for any domain $D$ and any conformal map $\varphi:D\rightarrow\IC$, we have the following:
 \begin{multline*}
  F(\varphi(D);\varphi(a_1),\cdots,\varphi(a_n))(\varphi(z)) =\\ \varphi'(z)^\alpha \overline{\varphi'(z)}^{\,\alpha'}\varphi'(a_1)^{\beta_1}\overline{\varphi'(a_1)}^{\,\beta_1'}\cdots \varphi'(a_n)^{\beta_n}\overline{\varphi'(a_n)}^{\,\beta_n'}\cdot F(D;a_1,\cdots, a_n)(z),
 \end{multline*}
 where, at a boundary prime end $a_j$, the derivative $\varphi'(a_j)$ is understood as the (real, positive) boundary derivative in the local coordinate.
 If $\alpha=\alpha'=\beta_1=\beta_1'=\cdots=\beta_n=\beta_n'=0$, then the family is \emph{conformally invariant}.
 \item and $F(D\backslash\gamma[0,t];\gamma(t),a_2,\cdots,a_n)$ is a martingale with respect to $\gamma$ drawn from $a_1$ (with Loewner parameterization).
\end{itemize}

\end{definition}

If $(a,b,c)$ maps to $(0,\infty,x)$ in $\IH$, these properties give the martingale
$$F(\IH;W_t,\infty,g_t(x),\cdots)\cdot g'_t(x)^\xi\overline{g'_t}(x)^\nu\cdots.$$
The factor at $\infty$ is $1$. The martingale methods of \cite{LSW,MR2227824} motivate the following principle \cite{S}.

\begin{remark}[Martingale principle, informal]\label{rem:mgprinciple}
The following heuristic, due to Schramm and Smirnov, motivates the approach: a nondegenerate smooth martingale observable, together with
the domain Markov property, determines the driving drift and quadratic variation. We use this principle quantitatively through the BPZ
equation and the non-degeneracy condition below.
\end{remark}


For conformally invariant observables, the next lemma gives a necessary condition for a smooth function on $\IH$ to define an SLE observable: a version of the BPZ equations \cite{BPZeq}. We first introduce the first-order differential operator
\begin{equation}\label{eq:Dop}
\Dop \;:=\; \sum_{j=1}^m \frac{\partial}{\partial x^j} \;+\; \sum_{j=1}^l \frac{\partial}{\partial c^j},
\end{equation}
the derivative in the direction of a simultaneous horizontal shift of all marked points. Its square contains all mixed second-order derivatives:
\begin{equation}\label{eq:Dopsq}
\Dop^2 \;=\; \sum_{i,j=1}^m \frac{\partial^2}{\partial x^i\,\partial x^j}
\;+\; 2\sum_{i=1}^m\sum_{j=1}^l \frac{\partial^2}{\partial x^i\,\partial c^j}
\;+\; \sum_{i,j=1}^l \frac{\partial^2}{\partial c^i\,\partial c^j}.
\end{equation}

\begin{lemma}\label{lem:BPZ}
 Let $$F=F(\IH;0,\infty, z^1,\cdots, z^m; c^1,\cdots, c^l)$$ where $z^1,\cdots,z^m$ are interior points of $\IH$ and $c^1,\cdots, c^l$ are boundary points, be a $C^2$-smooth conformally invariant
 martingale observable for SLE$_\kappa$. Then $F$ satisfies the following PDE (for $z^j = x^j+iy^j$):
\begin{equation}\label{eq:BPZ}
     \sum_{j=1}^m \frac{2}{(x^j)^2+(y^j)^2} \left(x^j\frac{\partial F}{\partial x^j} - y^j \frac{\partial F}{\partial y^j} \right)
     + \sum_{j=1}^l \frac{2}{c^j}\frac{\partial F}{\partial c^j} + \frac{\kappa}{2}\,\Dop^2 F
     = 0,
\end{equation}
where $\Dop^2$ is given by \eqref{eq:Dopsq}.
\end{lemma}

\begin{proof} Let $g_t$ be the Chordal Loewner map of Section \ref{LEandSLE}, driven by $W_t=\sqrt{\kappa}B_t$, and set $\xi_t(z) := g_t(z) - W_t$. Let $\tau_1$ be the first time at which a marked point is swallowed by the hull or two boundary coordinates collide, with each other or with $W$. For $t<\tau_1$ the trajectories of the points $z\in\overline{\IH}$ satisfy
\begin{equation}\label{eq:evolve1}
d\xi^z_t  = \frac{2}{\xi_t(z)}\, dt - \sqrt{\kappa}\,dB_t.
\end{equation}
Writing $\xi^{z^j}_t = x^j_t + i y^j_t$ for the interior points and taking real and imaginary parts of \eqref{eq:evolve1}, we obtain, for $j=1,\dots,m$,
\begin{equation}\label{eq:evolve2}
dx^j_t = \frac{2\,x^j_t}{(x^j_t)^2+(y^j_t)^2}\,dt - \sqrt{\kappa}\,dB_t,
\qquad
dy^j_t = \frac{-2\,y^j_t}{(x^j_t)^2+(y^j_t)^2}\,dt,
\end{equation}
while for the boundary points $c^j_t := \xi^{c^j}_t \in \IR$, $j = 1,\dots,l$,
\begin{equation}\label{eq:evolve3}
dc^j_t = \frac{2}{c^j_t}\,dt - \sqrt{\kappa}\,dB_t.
\end{equation}

The coordinates $y^j_t$ have bounded variation, so contribute no second-order terms, while every noise-carrying coordinate --- each $x^j_t$ and each $c^j_t$ --- has martingale part $-\sqrt{\kappa}\,dB_t$ with the \emph{same} $B_t$. Hence all quadratic covariations, diagonal and mixed alike, are equal:
\begin{equation}\label{eq:qcov}
d\langle x^i, x^j\rangle_t = d\langle x^i, c^j\rangle_t = d\langle c^i, c^j\rangle_t = \kappa\, dt
\qquad\text{for all admissible } i,j.
\end{equation}

Define $\hat\xi_t := (\xi^{z^1}_t,\cdots, \xi^{z^m}_t,\xi^{c^1}_t,\cdots, \xi^{c^l}_t)^T$. Applying It\^o's formula to $F(\hat\xi_t)$, viewed as a function of the real coordinates $(x^1,y^1,\dots,x^m,y^m,c^1,\dots,c^l)$, and using \eqref{eq:evolve2}, \eqref{eq:evolve3}, and \eqref{eq:qcov}, we get
\begin{multline*}
   dF(\hat\xi_t)  = \left(\sum_{j=1}^m \frac{2}{(x^j)^2+(y^j)^2}\left( x^j\frac{\partial F}{\partial x^j} - y^j \frac{\partial F}{\partial y^j}\right)  +  \sum_{j=1}^l \frac{2}{c^j}\frac{\partial F}{\partial c^j} + \frac{\kappa}{2}\, \Dop^2 F \right) dt \\
   - \sqrt{\kappa}\,\bigl( \Dop F \bigr)\, d B_t,
\end{multline*}
all derivatives being evaluated at $\hat\xi_t$.

Since $F(\hat\xi_t)$ is a martingale for $t<\tau_1$, its drift vanishes. By continuity at $t=0$, and since every configuration can be the initial one, it vanishes identically on the configuration space, which is \eqref{eq:BPZ}.
\end{proof}

In the following, we will use the following notation. For a $C^2$ function $F$ on an open set ${\mathcal V}\subset\IH^m\times\left(\IR\setminus\{0\}\right)^l$, define
\begin{equation}\label{eq:BCdef}
B(F) := \Dop F, \qquad C(F) := \tfrac12\, \Dop^2 F .
\end{equation}
With this notation, equation \eqref{eq:BPZ} takes the form $A(F) + \kappa\, C(F) = 0$, where $A(F)$ denotes the first-order (drift) part of \eqref{eq:BPZ}. For a configuration $\mathfrak{v} \in {\mathcal V}$, we form the row vector
\begin{equation}\label{eq:rowdef}
\mathcal{R}(F)(\mathfrak{v}) := \bigl( C(F)(\mathfrak{v}),\; B(F)(\mathfrak{v}) \bigr) \in \IC^2,
\end{equation}
and for a pair of (not necessarily distinct) configurations $\mathfrak{v}^{(1)}, \mathfrak{v}^{(2)} \in {\mathcal V}$, the real $4\times 2$ matrix
\begin{equation}\label{eq:Mdef}
M(F)(\mathfrak{v}^{(1)}, \mathfrak{v}^{(2)}) :=
\begin{pmatrix}
\Re\, \mathcal{R}(F)(\mathfrak{v}^{(1)}) \\
\Im\, \mathcal{R}(F)(\mathfrak{v}^{(1)}) \\
\Re\, \mathcal{R}(F)(\mathfrak{v}^{(2)}) \\
\Im\, \mathcal{R}(F)(\mathfrak{v}^{(2)})
\end{pmatrix},
\end{equation}
where $\Re$ and $\Im$ are applied entrywise. For real-valued $F$, the second and fourth rows vanish.

\begin{definition}\label{def:nondeg}
Let $F$ be a $C^2$ solution of the equation \eqref{eq:BPZ} on an open set ${\mathcal V}\subset\IH^m\times\left(\IR\setminus\{0\}\right)^l$. We say that $F$ is \emph{non-degenerate} on ${\mathcal V}$ if there exist configurations $\mathfrak{v}^{(1)}, \mathfrak{v}^{(2)} \in {\mathcal V}$, not necessarily distinct, such that the real $4\times 2$ matrix $M(F)(\mathfrak{v}^{(1)}, \mathfrak{v}^{(2)})$ defined in \eqref{eq:Mdef} has rank $2$, and \emph{degenerate} otherwise. We call $\mathfrak{v}^{(1)}, \mathfrak{v}^{(2)}$ \emph{witnessing configurations} for $F$.
\end{definition}

\begin{remark}\label{rem:nondeg}
We will use the following two sufficient criteria for non-degeneracy.
\begin{enumerate}
\item[(a)] \emph{Two configurations with different ratios:} there exist $\mathfrak{v}^{(1)}, \mathfrak{v}^{(2)} \in {\mathcal V}$ at which $(B(F), C(F)) \ne (0,0)$ and the ratio $B(F)/C(F)$, viewed as a map to the Riemann sphere $\widehat{\IC}$, takes two distinct values. Indeed, then the complex determinant $B(F)(\mathfrak{v}^{(1)})\, C(F)(\mathfrak{v}^{(2)}) - B(F)(\mathfrak{v}^{(2)})\, C(F)(\mathfrak{v}^{(1)})$ is nonzero, and hence $M$ has rank $2$ over $\IR$. This is the only mechanism available for real-valued $F$.
\item[(b)] \emph{One configuration, complex-valued $F$:} there exists a single $\mathfrak{v}^{(0)} \in {\mathcal V}$ at which $B(F)$ and $C(F)$ are linearly independent over $\IR$; equivalently, $C(F)(\mathfrak{v}^{(0)}) \neq 0$ and $B(F)/C(F)(\mathfrak{v}^{(0)}) \notin \IR$; equivalently, $\Im\bigl( \overline{B(F)}\, C(F) \bigr)(\mathfrak{v}^{(0)}) \neq 0$. Then $M(F)(\mathfrak{v}^{(0)}, \mathfrak{v}^{(0)})$ already has rank $2$.
\end{enumerate}
\end{remark}

\begin{lemma}\label{lem:nondegquant}
Let $F$ be non-degenerate on ${\mathcal V}$ with witnessing configurations $\mathfrak{v}^{(1)}, \mathfrak{v}^{(2)}$. Then there exist constants $\varepsilon_0 > 0$, $c > 0$, and $\Lambda < \infty$, depending only on $F$ and on $\mathfrak{v}^{(1)}, \mathfrak{v}^{(2)}$, with the following properties. The closed balls $\overline{B}(\mathfrak{v}^{(i)}, \varepsilon_0)$, $i = 1,2$, are contained in ${\mathcal V}$,
\[
\sup_{\overline{B}(\mathfrak{v}^{(1)}, \varepsilon_0) \cup \overline{B}(\mathfrak{v}^{(2)}, \varepsilon_0)} \bigl( |B(F)| + |C(F)| \bigr) \le \Lambda,
\]
and for any $\tilde{\mathfrak{u}}^{(i)} \in \overline{B}(\mathfrak{v}^{(i)}, \varepsilon_0)$, $i = 1, 2$, any $\varepsilon \ge 0$, and any real numbers $X$, $Y$ satisfying
\[
\bigl| C(F)(\tilde{\mathfrak{u}}^{(i)})\, X + B(F)(\tilde{\mathfrak{u}}^{(i)})\, Y \bigr| \le \varepsilon, \qquad i = 1, 2,
\]
one has
\[
|X| + |Y| \;\le\; \frac{4\Lambda}{c}\, \varepsilon .
\]
\end{lemma}

\begin{proof}
Choose a $2\times 2$ submatrix $M'$ of $M(F)(\mathfrak v^{(1)},\mathfrak v^{(2)})$ with nonzero determinant and set $c:=\tfrac12|\det M'|$. By continuity, choose $\varepsilon_0>0$ so that the closed balls in the statement lie in ${\mathcal V}$ and the corresponding minor has determinant of absolute value at least $c$ throughout their product. Let $\Lambda$ be the supremum in the statement, finite by compactness. Taking the corresponding two real components of the assumed inequalities and applying Cramer's rule gives $|X|,|Y|\le 2\Lambda\varepsilon/c$.
\end{proof}

\section{Main Theorem}\label{sec:mainthmsetup}

As outlined in the introduction, the proof of the main theorem involves two key steps. The first step is to derive a rate of convergence for the driving terms. Then we extend the result based on the work of Viklund in \cite{V} to a rate of convergence for the curves. All the a priori estimates required to extend from a convergence of driving terms to a convergence of paths follow from the discrete curve satisfying the Kemppainen-Smirnov condition. The method developed here is a general framework for obtaining the rate of convergence of the scaling limit of interfaces for various models in statistical physics known to converge to SLE curves, provided the martingale (or almost martingale) observable converges polynomially to its continuous counterpart.

Let us describe the set-up for the main theorem. Fix conformal transformation $\Phi:\ID\rightarrow\IH$ such that $z\mapsto i \frac{z+1}{1-z}$. Let $d_*(\cdot,\cdot)$ be the metric on $\overline{\IH}\cup \{ \infty\}$ defined by $d_*(z,w) = | \Phi^{-1}(z) - \Phi^{-1}(w)|$. If $z\in\overline{\IH}$ then $d_*(z_n,z)\rightarrow 0$ is equivalent to $|z_n-z|\rightarrow 0$ and $d_*(z_n,\infty)\rightarrow 0$ is equivalent to $|z_n|\rightarrow \infty$. This metric is convenient because it makes $\overline{\IH}\cup\{\infty\}$ compact.
Let $D$ be a simply-connected bounded domain with distinct accessible prime ends $a,b\in\partial D$, let $D^{n} \subset D$ denote a $1/n$-discrete approximation of $D$, and let $\mathfrak{p}:(D,a,b)\rightarrow(\ID, -1,1)$ and $\mathfrak{p}^n:(D^n,a^n,b^n)\rightarrow (\ID, -1, 1)$ be the maps of Section \ref{subsec:graph_approx}; recall that $\mathfrak{p}^n\rightarrow\mathfrak{p}$ uniformly on compact subsets of $D$. Set $\varphi:=\Phi\circ\mathfrak{p}: D\rightarrow\IH$ and $\varphi^{n}:=\Phi\circ\mathfrak{p}^{n}: D^{n}\rightarrow \IH$, so that $\varphi(a)=\varphi^n(a^n)=0$ and $\varphi(b)=\varphi^n(b^n)=\infty$.
Then let $\tilde\gamma^{n} = \varphi^n(\gamma^{n})$ where $\gamma^{n}$ is a random discrete curve between $a^n$ and $b^n$, as defined in  Section \ref{subsec:graph_approx}. Thus, $\tilde\gamma^{n}$ is a random curve in $(\IH,0,\infty)$ parameterized by capacity. Let $D^n_t=D^n\backslash \gamma^n[0,t]$, and $g_t^{n}: \IH_t^{n}= \varphi^{n}(D_t^{n})\rightarrow \IH$ be the corresponding Loewner evolution with driving terms $W_t^{n}$, so that $(f_t^{n} = (g_t^{n})^{-1}, W_t^{n})$ is a $\IH$-Loewner pair generated by $\tilde\gamma^{n}$.

To deal with multivariable observables, let us introduce some notation. Fix integers $m, l \ge 0$ with $m + l \ge 1$. Let $\hat{v}^n=(v_1^n,\dots, v_m^n; c_1^n,\dots, c_l^n)$ be a collection of $m$ interior points of the domain $D^n$ and $l$ boundary points of $D^n$, and let $V^{m,l}(D^n)$ be the set of all such collections.

Define
$$\hat{\varphi^n}(\hat{v}^n):=(\varphi^n(v_1^n),\dots, \varphi^n(v_m^n); \varphi^n(c_1^n),\dots, \varphi^n(c_l^n)).$$
Since $\partial D^n$ is locally connected (it is a finite union of lattice edges), the inverse maps $(\mathfrak{p}^n)^{-1}$ and $(\varphi^n)^{-1}$ extend continuously to $\overline{\ID}$ and $\overline{\IH}$ by the Carath\'eodory--Torhorst theorem, so $\mathfrak{p}^n(c_j^n)$ and $\varphi^n(c_j^n)$ are well-defined boundary values at boundary prime ends $c_j^n$.

The natural configuration space for the transformed collections is
\begin{equation}\label{eq:confspace}
\begin{aligned}
{\mathcal X}^{m,l} := \bigl\{ (z^1,\dots,z^m; c^1,\dots,c^l) \in \IH^m\times(\IR\setminus\{0\})^l \,:\, {} \\
\text{the } z^j \text{ are distinct, the } c^j \text{ are distinct} \bigr\},
\end{aligned}
\end{equation}
an open subset of $\IH^m \times \IR^l$; it excludes collisions between marked points and collisions of boundary points with the driving value (which sits at the origin in the centered coordinates introduced below).

Let $({\mathcal F}^n_t)$ be a filtration to which $\gamma^n$ is adapted and for which the domain Markov property holds; the curve filtration $\sigma\bigl(\gamma^n[0,t]\bigr)$ suffices when the trace determines the revealed state, and otherwise we use the full exploration filtration. All conditional expectations below are with respect to this filtration. For $t \ge 0$, recall the slit domain $D^n_t = D^n \setminus \gamma^n[0,t]$, the map $\varphi^n_t := g^n_t \circ \varphi^n : D^n_t \to \IH$, and the driving term $W^n_t$; all three are ${\mathcal F}^n_t$-measurable. Write
$$H^n_t(\hat v):=H^n(\mathfrak S_t^n;\hat v)$$
for the state observable at a vertex time, with $\hat v$ marked in the geometric slit domain, and define the \emph{centered configuration map}
\begin{equation}\label{eq:centered}
\Xi^n_t(\hat{v}) := \hat{\varphi}^n_t(\hat{v}) - W^n_t\,(1,\dots,1) = \bigl(\varphi^n_t(v_1) - W^n_t, \dots;\; \varphi^n_t(c_1) - W^n_t, \dots\bigr),
\end{equation}
so that $\Xi^n_t(\hat{v}) \in {\mathcal X}^{m,l}$ as long as no point of $\hat{v}$ has been swallowed and no collision has occurred. At $t = 0$ we have $\Xi^n_0 = \hat{\varphi}^n$. For a bounded ${\mathcal F}^n$-stopping time $\tau$ and $G\in{\mathcal F}^n_\tau$, an \emph{${\mathcal F}^n_\tau$-measurable multivertex on $G$} is a map $\hat v$ on $G$ with values in $V^{m,l}(D^n_\tau)$ whose interior components and image $\Xi^n_\tau(\hat v)$ are ${\mathcal F}^n_\tau$-measurable; $\Xi^n_\tau(\hat v)$ determines its boundary components, which are prime ends of $D^n_\tau$ and may lie on the slit. We also write
\begin{equation}\label{eq:CradSlit}
\Crad^n_t := \Bigl| \bigl( (\varphi^n_t - W^n_t)^{-1} \bigr)'(i) \Bigr|
\end{equation}
for the conformal radius of the slit configuration $(D^n_t, \gamma^n(t), b^n)$; note that $\Crad^n_0 = \Crad(\varphi^n)$. We use this normalization throughout; since the conformal radius of $\IH$ at $i$ is $2$, it is one half of the standard conformal radius.

\begin{definition}\label{def:family}
Let $\{\gamma_{(D,a,b)}^n\}$ be a family of probability measures on curves indexed by simply-connected domains $D$, boundary prime ends $a,\, b$ of $D$, and $n\in\IN$ such that for each $D$ and $n$ it is supported on discrete curves joining $a^n$ and $b^n$ in $D^n$. The threshold functions below take the conformal-radius bound as argument, and, when the family is indexed by domains of diameter at most a fixed $R_1 < \infty$, may take $R_1$ as a second argument; no other dependence on the domain is permitted, and exponents may depend on neither.

We say that the family $\{\gamma_{(D,a,b)}^n\}$ is \emph{weakly $\kappa$-admissible} for some $\kappa>0$ if it satisfies the domain Markov property (Remark \ref{rem:DMP}) and one can find \emph{admissibility data}: integers $m, l \ge 0$ with $m+l\ge 1$, an exponent $s\in(0,1)$, a bounded open set ${\mathcal V}\subset{\mathcal X}^{m,l}$, a function $\mathfrak h\in C^3({\mathcal V})$, and threshold functions $n_1$ and $n_0 \ge n_1$, the latter nonincreasing in its first argument --- none of which depend on the particular domain $D$ beyond the arguments allowed above, on the choice of $\varphi$, or on $n$ --- such that for any conformal isomorphism $\varphi: D\to\IH$ with $\varphi(a)=0$ and $\varphi(b)=\infty$, the following holds.

\begin{enumerate}
 \item\label{cond:1} \textbf{(Almost martingale at stopping times.)} There is an adapted family of discrete observables $H_t^n:V^{m,l}(D_t^n)\to\IC$ at vertex times, determined by the exploration state and the marked incidences, independently of the conformal normalization, such that for every $r>0$, all $n\geq n_1(r,R_1)$, all bounded ${\mathcal F}^n$-stopping times $\tau \le \tau'$, all events $G\in{\mathcal F}^n_\tau$ on which $\gamma^n(\tau),\gamma^n(\tau')\in V^n$, and every ${\mathcal F}^n_\tau$-measurable multivertex $\hat v$ on $G$ with $\Crad_t^n\ge r$ and $\Xi^n_t(\hat{v}) \in {\mathcal V}$ for all $t \in [\tau, \tau']$ almost surely on $G$, the variable $\mathbf 1_G H^n_{\tau'}(\hat v)$ is integrable and
 $$\Bigl| H^{n}_{\tau}(\hat{v}) -  \mathbb{E}\bigl[\mathbf 1_G H^{n}_{\tau'}(\hat{v}) \; \bigm| \; {\mathcal F}^n_\tau \bigr] \Bigr| \leq  n^{-s} \qquad \text{almost surely on } G.$$

 \item \label{cond:2a} \textbf{(Non-degenerate BPZ solution.)} The function $\mathfrak h$ in the admissibility data solves the BPZ equation \eqref{eq:BPZ} with parameter $\kappa$ and is non-degenerate in the sense of Definition \ref{def:nondeg}.

  \item\label{cond:2b} \textbf{(Dynamic polynomial closeness.)} $H^{n}$ is polynomially close to $\mathfrak{h}$, uniformly on ${\mathcal V}$ and uniformly in all slit domains produced by the curve: for every $R' > 0$, every $n \ge n_0(R',R_1)$, every bounded ${\mathcal F}^n$-stopping time $\tau$, every event $G\in{\mathcal F}^n_\tau$ on which $\gamma^n(\tau)\in V^n$ and $\Crad^n_\tau \ge R'$, and every ${\mathcal F}^n_\tau$-measurable multivertex $\hat v$ on $G$ with $\Xi^n_\tau(\hat{v})\in{\mathcal V}$ on $G$,
  \begin{equation}\label{assup}\Bigl|H^{n}_{\tau}(\hat{v}) - \mathfrak{h}\bigl(\Xi^n_\tau (\hat{v})\bigr) \Bigr| \leq  n^{-s} \qquad \text{almost surely on } G.  \end{equation}
\end{enumerate}

If, \emph{in addition}, the family $\{\gamma_{(D,a,b)}^n\}$ satisfies the KS-condition with a uniform modulus bound, it is called \emph{$\kappa$-admissible}.
\end{definition}

\begin{remark}\label{rem:admissibility} \begin{enumerate}
    \item We need lower bounds on the conformal radius to make sure that the lattice mesh is fine enough to capture the behavior of the domain, including along the comparison intervals in condition \ref{cond:1}.
\item The BPZ equation is automatic for a $C^2$ conformally invariant martingale observable by Lemma~\ref{lem:BPZ}; non-degeneracy is the
additional model-dependent hypothesis. Lemma~\ref{lem:nondegquant} then supplies the corresponding uniform constants.
    \item Under the stopping and measurability conditions above, optional sampling gives condition~\ref{cond:1} for bounded exact martingales. In Section~\ref{applications}, we compare the state observable with the martingale of a crossing event fixed at the start of the interval. For almost martingales, condition~\ref{cond:1} requires separate verification: deterministic-time errors may accumulate.
    \item The condition $\Xi^n_t(\hat v)\in{\mathcal V}$ includes the swallowing and collision cutoffs.
\end{enumerate}
\end{remark}

All notions in Theorem~\ref{mainthm} are now defined. Its proof occupies Section~\ref{subsec:mainthm}, after we have developed the driving-term convergence (Section~\ref{sec:convdrive}) and the path-level estimates (Section~\ref{sec:path}) it combines.

\begin{remark}
We believe that the $\alpha$-H\"{o}lder condition can be relaxed using the estimates of Garban, Rohde, and Schramm in \cite{GRS}. The exponents $u$ and $\lambda$ are specified in the proof in Section \ref{subsec:mainthm}.
\end{remark}

In Section \ref{sec:path} we transfer the curves to $\ID$ using $\mathfrak p^n$. Lemma \ref{lem:discrete} and Remark \ref{rem:discretemaps} give a polynomial rate of convergence of $(\mathfrak p^n)^{-1}$ to $\mathfrak p^{-1}$ when $D$ is H\"older.

\section{Convergence of Driving Term}\label{sec:convdrive}

 For the first step, the derivation of a rate of convergence of driving terms, we follow the scheme outlined in \cite{BJVK}. The main contributions to the convergence rate are: the rate $n^{-s}$ at which the discrete observable approaches its continuum counterpart (Definition \ref{def:family}); the transfer of this rate to the driving function over mesoscopic pieces of the curve (Proposition \ref{keyest}); and the coupling with Brownian motion by Skorokhod embedding (Lemma \ref{lem:coupling}), which yields the rate $n^{-s/16}$ of Theorem \ref{drivingcvg}. A mesoscopic piece has capacity at most $2n^{-\frac{2s}{3}}$; on it the driving function moves by at most $2n^{-\frac{s}{3}}$, and its conditional drift and variance match those of $\sqrt\kappa B$ up to $O(n^{-s})$.

Recall the set-up described before Theorem \ref{mainthm}. For $\rho \in (0,1)$ and the discrete curve $\tilde\gamma^n$ in $(\IH, 0, \infty)$, parameterized by capacity, define the stopped horizon
\begin{equation}\label{eq:Tnrho}
T^n_\rho := T \wedge \inf\bigl\{ t \ge 0 \,:\, d_*\bigl(\tilde\gamma^n(t), \infty\bigr) \le \rho \bigr\},
\end{equation}
an ${\mathcal F}^n$-stopping time: since $d_*(z,\infty) = |\Phi^{-1}(z)-1| = 2/|z+i|$, the infimum is the first time the curve, in disk coordinates, comes within distance $\rho$ of the target prime end, and it is attained before the endpoint of the curve.

\begin{theorem}\label{drivingcvg} Let $\{\gamma^n\}$ be a weakly $\kappa$-admissible family for some $\kappa>0$ (in the sense of Definition \ref{def:family}). Suppose $\Crad(\varphi)>R$ and $\diam D\le R_1$, and let $T \ge 1$, $\rho \in (0,1)$. Set $r := s/16$. There exists $n_\star = n_\star(R, R_1, T, \rho, \text{admissibility data}) < \infty$ such that for each $n \ge n_\star$ there is a coupling of $\gamma^n$ with a standard Brownian motion $B(t)$, $t \ge 0$, with the property that
 \[
  \IP \left(\sup_{t\in[0,T^n_\rho]} \bigl| W^n(t) - \sqrt\kappa\,  B(t) \bigr| > n^{-r} \right) < n^{-r}.
 \]
 The Radial analogue, with $T$ in place of $T^n_\rho$ and $n_\star$ independent of $\rho$, is discussed in Appendix~\ref{app:radial}.
\end{theorem}
We do not need the KS Condition (as in Definition \ref{cond:KS}) to establish the convergence of driving functions. The result also works for all $\kappa>0$.

\subsection{Key Estimate}

 As in \cite{LSW} and \cite{BJVK}, the idea is to use the polynomial convergence of the discrete observable to its continuous counterpart in order to transfer the fact that we have a discrete martingale observable to information about the Loewner driving function for a mesoscopic piece of the path.

Throughout this section, $\{\gamma^n_{(D,a,b)}\}$ is a weakly $\kappa$-admissible family with admissibility data as in Definition \ref{def:family}: the exponent $s \in (0,1)$, the open set ${\mathcal V}$, and the non-degenerate observable $\mathfrak{h}$. We fix, once and for all, witnessing configurations $\mathfrak{v}^{(1)}, \mathfrak{v}^{(2)}$ for $\mathfrak{h}$ (equal in the one-point case of Remark \ref{rem:nondeg}(b)), let $\varepsilon_0, c, \Lambda$ be the constants provided by Lemma \ref{lem:nondegquant}, and set
\[
{\mathcal K} := \overline{B}(\mathfrak{v}^{(1)}, \varepsilon_0) \cup \overline{B}(\mathfrak{v}^{(2)}, \varepsilon_0) \Subset {\mathcal V} .
\]
Conditions \ref{cond:1} and \ref{cond:2b} of Definition \ref{def:family} will be used at the conformal-radius lower bound of each comparison interval, with $n_0(r,R_1)\ge n_1(r,R_1)$. We also use the constants
\begin{equation}\label{eq:Khconstants}
\begin{aligned}
d_0 &:= \min_{\hat{z} \in {\mathcal K}}\ \min\Bigl(\{\Im z^r:1\le r\le m\}\cup\{|c^r|:1\le r\le l\}\Bigr) > 0,\\
\Lambda_3 &:= \sup_{{\mathcal K}} \bigl( |D\mathfrak{h}| + |D^2 \mathfrak{h}| + |D^3 \mathfrak{h}| \bigr) < \infty,
\end{aligned}
\end{equation}
finite and positive since ${\mathcal K}$ is compact and $\mathfrak{h} \in C^3({\mathcal V})$. All constants below depend only on $s, \varepsilon_0, c, \Lambda, d_0, \Lambda_3, \diam({\mathcal K})$, $m,l$, the separation of the witnessing coordinates, and $R$; never on $n$, on $D$, or on the step considered.

Fix the domain $D$ and retain the superscript $n$ throughout. Let $\gamma^n$ be the random curve on the graph $\Gamma^n$ between $a^n$ and $b^n$, passing through the vertices
$$v_0 = a^n,\ v_1, \dots, v_{l(\gamma^n)} = b^n,$$
and let ${\mathcal F}^n_j := {\mathcal F}^n_{t_j}$, $j = 0, 1, \dots$, be the filtration of Section \ref{sec:mainthmsetup} sampled at the vertex times $t_j$ of \eqref{eq:tjdef} below; each $t_j$ is an ${\mathcal F}^n$-stopping time, since $\gamma^n$ is adapted. For the curve filtration, ${\mathcal F}^n_j=\sigma(v_0,\dots,v_j)$; in general ${\mathcal F}^n_j$ also contains everything revealed by time $t_j$. In this section an integer subscript on ${\mathcal F}^n$ always refers to this sampled filtration, and stopping indices and conditional expectations are taken with respect to it. Writing $[v_0,\ldots,v_j]$ for the polygonal path through the first $j$ edges, set
\begin{equation}\label{eq:tjdef}
t_j := \tfrac12\, \hcap\bigl(\varphi^n([v_0,\ldots,v_j])\bigr),
\end{equation}
so that $\hcap(K_t)=2t$. Write $\varphi_t^n=g_t^n\circ\varphi^n$, with $\varphi_{t_j}^n(v_j)=W^n(t_j)$. Recall the centered configuration map $\Xi_t^n$, the radius $\Crad_t^n$ of \eqref{eq:CradSlit}, and the moving basepoint
\begin{equation}\label{eq:basepoint}
u^n_t := (\varphi^n_t - W^n(t))^{-1}(i) \in D^n_t,
\qquad
\tfrac12\, \Crad^n_t \le \dist(u^n_t, \partial D^n_t) \le 2\,\Crad^n_t
\end{equation}
by the Koebe quarter and distortion theorems.

We first show that, in every slit domain of sufficiently large conformal radius, each witnessing configuration can be realized, up to a small error, by a multivertex of the graph.

\begin{lemma}[Multivertex selection]\label{lem:select}
There are constants $C_1 = C_1({\mathcal K}, R) < \infty$ and $n_2 = n_2({\mathcal K}, R) < \infty$ such that the following holds for all $n \ge n_2$. Let $j$ be an index with $\Crad^n_{t_j} \ge R$, and let $\mathfrak{v} \in \{\mathfrak{v}^{(1)}, \mathfrak{v}^{(2)}\}$. Then there exists a multivertex $\hat{p} \in V^{m,l}(D^n_{t_j})$, measurable with respect to ${\mathcal F}^n_j$, whose components are distinct vertices of $\Gamma^n$ (interior vertices of $D^n_{t_j}$ for the first $m$ components, boundary prime ends supported at vertices of $\partial D^n_{t_j}$, including the two sides of the slit, for the last $l$), such that
\[
\bigl| \Xi^n_{t_j}(\hat{p}) - \mathfrak{v} \bigr| \;\le\; C_1\, n^{-1/2} .
\]
In particular, $\bigl| \Xi^n_{t_j}(\hat{p}) - \mathfrak{v} \bigr| \le \varepsilon_0/2$ for $n \ge n_2$.
\end{lemma}

\begin{proof}
Write $\psi := (\varphi^n_{t_j} - W^n(t_j))^{-1} : \IH \to D^n_{t_j}$, so $\psi(i) = u^n_{t_j}$ and $|\psi'(i)| = \Crad^n_{t_j} \ge R$. Since ${\mathcal V}$ is bounded, set $M := \sup_{\mathcal V} |\hat z| < \infty$; since $\mathfrak{v} \in {\mathcal K} \subset {\mathcal V}$, all interior components $z^r$ of $\mathfrak{v}$ satisfy $d_0 \le \Im z^r \le |z^r| \le M$, and all boundary components $\zeta^r$ satisfy $d_0 \le |\zeta^r| \le M$.

\emph{Interior components.} For each interior coordinate $z^r$, put
$q_r=\psi(z^r)$. Since $z^r$ stays in the fixed compact set
${\mathcal K}$, Koebe distortion gives
$\dist(q_r,\partial D^n_{t_j})\ge c({\mathcal K})R$.
Choose a lattice vertex $v_r^*$ within $1/n$ of $q_r$; Koebe distortion
on the resulting fixed-size ball gives
\[
|\varphi^n_{t_j}(v_r^*)-\varphi^n_{t_j}(q_r)|=O(n^{-1}).
\]

\emph{Boundary components.} Let $\beta_r$ be the boundary prime end of $D^n_{t_j}$ with $\varphi^n_{t_j}(\beta_r) = \zeta^r + W^n(t_j)$. Every boundary side of $D^n_{t_j}$ --- an arc of $\partial D^n$, or one of the two sides of the lattice path $\gamma^n[0,j]$, each side of a slit point carrying its own prime end --- is a polygonal arc with edges of diameter at most $1/n$. Hence $\beta_r$ and one of the two vertex-supported prime ends bounding its edge on that same side are separated from $u^n_{t_j}$ by a crosscut $\chi_r$ of $D^n_{t_j}$ of diameter at most $2/n$. Choose such a $\beta^*_r$. By the Beurling projection theorem and \eqref{eq:basepoint},
\[
\omega\bigl( u^n_{t_j},\, \chi_r,\, D^n_{t_j} \bigr) \le C \sqrt{ \frac{2/n}{\dist(u^n_{t_j}, \partial D^n_{t_j})} } \le C(R)\, n^{-1/2}.
\]
By conformal invariance, the interval $I_r\subset\IR$ cut off by $(\varphi^n_{t_j}-W^n(t_j))(\chi_r)$ contains $\zeta^r$ and has harmonic measure at most $C(R)n^{-1/2}$ from $i$. For large $n$, $I_r\subset[-M-1,M+1]$: otherwise it contains a unit interval adjoining $\zeta^r\in[-M,M]$, of harmonic measure bounded below in terms of $M$. On this bounded interval the Poisson kernel is bounded below, so
\[
\bigl| \bigl(\varphi^n_{t_j}(\beta^*_r) - W^n(t_j)\bigr) - \zeta^r \bigr| \le |I_r| \le C(M, R)\, n^{-1/2}.
\]

\emph{Distinctness and measurability.} Interior vertices are distinct by the separation of the witnessing coordinates. Choose boundary supports successively, following the boundary walk on the specified side from each candidate to an unused vertex. At most $2(l+1)^2$ edges are needed, since a simple graph on $l$ vertices has at most $l(l-1)/2$ edges and each edge occurs at most twice in the boundary walk. There are more than $l$ boundary vertices for large $n$, since $B(u^n_{t_j},R/2)\subset D^n_{t_j}$ and each edge has length at most $1/n$. The added walk has diameter $O(l^2/n)$; the small-crosscut and Beurling argument retains the $O(n^{-1/2})$ error. Boundary and interior supports are disjoint. Fix an enumeration of vertex--side incidences and a walk direction for measurability. At a stopping index $k$, the selection is ${\mathcal F}^n_k$-measurable on $\{\Crad^n_{t_k}\ge R\}$.

\end{proof}

The next lemma controls one mesoscopic step: the driving term and capacity increments, and the persistence of the conformal radius. For a bounded ${\mathcal F}^n$-stopping index $k$, define
\begin{equation}\label{eq:mdef}
k^+ := \min\Bigl\{ j \ge k \;:\; t_j - t_k \geq n^{-\frac{2s}{3}} \ \text{ or } \ |W^n(t_j) - W^n(t_k)| \geq n^{-\frac{s}{3}} \Bigr\} \wedge l(\gamma^n),
\end{equation}
an ${\mathcal F}^n$-stopping index.

We record one elementary estimate, used below to pass from a harmonic-measure bound to a diameter bound for the image of a lattice edge.

\begin{lemma}[Diameter versus harmonic measure for a boundary hull]\label{lem:hulldiam}
Let $K \subset \IH$ be compact. There is $C = C(K) < \infty$ with the following property. If $L \subset \overline{\IH}$ is compact and connected and $0 \in L$, then for every $z \in K\setminus L$,
\[
\min\{\diam(L),1\} \ \le\ C\, \omega\bigl( z,\, L,\, \IH \setminus L \bigr) .
\]
\end{lemma}

\begin{proof}
Set $r=\min\{\diam L,1\}$ and choose $w\in L$ with $|w|=r/2$. For $T_z(v):=(v-z)/(v-\overline z)$, the continuum $T_z(L)$ meets $\partial\ID$ at $T_z(0)$ and has diameter at least $|T_z(w)-T_z(0)|\ge c(K)r$. Conformal invariance and \cite[Theorem~2]{FRW1985} give
$$\omega(z,L,\IH\setminus L)\ge\frac1\pi\arcsin\frac{\diam T_z(L)}2\ge c(K)r.$$
\end{proof}

\begin{lemma}[One-step estimates]\label{lem:onestep}
There is $n_3 = n_3(s, R) < \infty$ such that for all $n \ge n_3$, almost surely on the event $\{\Crad^n_{t_k} \ge R\}$:
\begin{align}
&\sup_{t \in [t_k, t_{k^+}]} |W^n(t) - W^n(t_k)| \le 2\, n^{-\frac{s}{3}}, \qquad t_{k^+} - t_k \le 2\, n^{-\frac{2s}{3}}; \label{eq:onestepincr}\\
&\Crad^n_{t'} \ge \tfrac{1}{16}\, \Crad^n_{t_k} \quad \text{for all } t' \in [t_k, t_{k^+}]. \label{keyes:0}
\end{align}
In particular $k^+<l(\gamma^n)$.
\end{lemma}

\begin{proof}
Write $\varrho := \Crad^n_{t_k}$, $d_k := \dist(u^n_{t_k}, \partial D^n_{t_k}) \ge \varrho/2$ by \eqref{eq:basepoint}, and $B := B(u^n_{t_k}, d_k/2)$. All statements are on the event $\{\varrho \ge R\}$.

\emph{Step 1: the ball $B$ is protected.} The hyperbolic diameter of $B$ in $D^n_{t_k}$ is bounded by an absolute constant, so its image $(\varphi^n_{t_k} - W^n(t_k))(B)$ lies within a fixed hyperbolic ball around $i$ in $\IH$; in particular there is an absolute constant $c_0 \in (0,1)$ with
\begin{equation}\label{eq:imlow}
\Im\bigl( \varphi^n_{t_k}(z) - W^n(t_k) \bigr) \ge c_0, \qquad
\bigl| \varphi^n_{t_k}(z) - W^n(t_k) - i\bigr| \le 1/c_0, \qquad z \in B.
\end{equation}
Under the Loewner flow started at time $t_k$, $\frac{d}{dt} \Im g \ge -2/\Im g$, so $(\Im g)^2$ decreases at rate at most $4$; hence no point of $B$ is swallowed before time $t_k + c_0^2/4$.

\emph{Step 2: increments, by induction along the vertex steps.} Let $j_* \le k^+$ be the largest index such that for all $j \in [k, j_*]$,
\begin{equation}\label{eq:induct}
\sup_{t \in [t_k, t_j]} |W^n(t) - W^n(t_k)| \le 2 n^{-\frac{s}{3}}
\quad\text{and}\quad
t_j - t_k \le 2 n^{-\frac{2s}{3}} .
\end{equation}
We claim $j_* = k^+$ for $n \ge n_3(s, R)$. Suppose $j < k^+$ and \eqref{eq:induct} holds up to $j$; we extend it to $j+1$. Note first that $2n^{-2s/3} \le c_0^2/4$ for $n$ large, so by Step 1 no point of $B$ is swallowed by time $t_j$; moreover, the hull $\tilde K_j := (\varphi^n_{t_k} - W^n(t_k))\bigl( \gamma^n[t_k, t_j] \bigr)$ grown in the $\IH$-picture between times $t_k$ and $t_j$ contains the point $0$ and, by Lemma \ref{DiamKt} and \eqref{eq:induct}, has diameter at most $C\bigl( \sqrt{t_j - t_k} + \sup_{[t_k,t_j]}|\Delta W^n| \bigr) \le C n^{-\frac{s}{3}}$; hence $\tilde K_j \subset B(0, Cn^{-s/3})$, which is disjoint from the image of $B$ by \eqref{eq:imlow} once $Cn^{-s/3} < c_0$. By injectivity of $\varphi^n_{t_k}$,
\begin{equation}\label{eq:avoid}
\gamma^n[t_k, t_j] \cap B = \emptyset, \qquad\text{so}\qquad \dist(u^n_{t_k}, \partial D^n_{t_j}) \ge d_k/2 \ge \varrho/8 .
\end{equation}
Now consider the single edge $e_{j} := \gamma^n[t_{j}, t_{j+1}]$, of Euclidean diameter at most $1/n$. Flowing $\varphi^n_{t_k}(u^n_{t_k}) = i + W^n(t_k)$ from $t_k$ to $t_{j}$: by \eqref{eq:induct} and the Loewner equation the imaginary part stays in $[\tfrac12, 1]$ and the speed is bounded by $4$, so
\begin{equation}\label{eq:baseflow}
\varphi^n_{t_{j}}(u^n_{t_k}) - W^n(t_{j}) = i + O(n^{-\frac{s}{3}}) .
\end{equation}
Set
\[
z_j:=\varphi^n_{t_j}(u^n_{t_k})-W^n(t_j),\qquad
L_j:=\overline{(\varphi^n_{t_j}-W^n(t_j))
                  (e_j\setminus\{\gamma^n(t_j)\})};
\]
then $0\in L_j$, $z_j\notin L_j$, and $z_j$ lies in a fixed compact subset of $\IH$ by \eqref{eq:baseflow}. By conformal invariance, the Beurling projection theorem, and \eqref{eq:avoid},
\[
\omega(z_j,L_j,\IH\setminus L_j)
=\omega(u^n_{t_k},e_j,D^n_{t_j}\setminus e_j)
\le C\sqrt{\frac{1/n}{\varrho/8-1/n}}
\le C(R)n^{-1/2}.
\]
Consequently $e_j$ is not the terminal edge of $\gamma^n$: otherwise $L_j$ would join $0$ to $\infty$, leaving only one half-line of $\IR$ accessible from $z_j$. Its exit probability is at most $\tfrac12+O(|z_j-i|)$, so $\omega(z_j,L_j,\IH\setminus L_j)\ge\tfrac12-O(n^{-s/3})$. Lemma \ref{lem:hulldiam} now gives
\begin{equation}\label{eq:edgediam}
\diam(L_j) \le C(R)\, n^{-1/2} .
\end{equation}
For $t \in [t_j, t_{j+1}]$ the hull grown after $t_j$ is, in the coordinates of $\varphi^n_{t_j}-W^n(t_j)$, contained in $L_j$, so Lemma \ref{DiamKt} and \eqref{eq:edgediam} give $\sup_{t \in [t_j, t_{j+1}]} |W^n(t) - W^n(t_j)| \le C(R) n^{-1/2}$, and $t_{j+1} - t_j = \tfrac12 \hcap \le C \diam^2 \le C(R) n^{-1}$, both taken for the centered image hull of \eqref{eq:edgediam}. Because $j < k^+$, the defining thresholds \eqref{eq:mdef} were not yet reached at index $j$, i.e. $t_j - t_k < n^{-2s/3}$ and $|W^n(t_j) - W^n(t_k)| < n^{-s/3}$; since $s < 1$ we have $n^{-1/2} = o(n^{-s/3})$ and $n^{-1} = o(n^{-2s/3})$, so for $n \ge n_3(s,R)$ the bounds \eqref{eq:induct} hold up to $j+1$. This proves the claim, and \eqref{eq:onestepincr} follows.

\emph{Step 3: radius persistence.} Let $t' \in [t_k, t_{k^+}]$. By Step 2 and \eqref{eq:avoid} (now valid up to $t'$), $\dist(u^n_{t_k}, \partial D^n_{t'}) \ge d_k/2$. By \eqref{eq:baseflow}, the centered images of $u^n_{t_k}$ and of $u^n_{t'}$ (which is $i$ by definition) are within $O(n^{-s/3})$, hence within hyperbolic distance $o(1)$ in $\IH$; therefore $|u^n_{t'} - u^n_{t_k}| = o(1) \cdot \dist(u^n_{t_k}, \partial D^n_{t'}) $, and
\[
\Crad^n_{t'} \ge \tfrac12\dist(u^n_{t'}, \partial D^n_{t'}) \ge \tfrac{d_k}{4}(1 - o(1)) \ge \tfrac{d_k}{8} \ge \tfrac{\varrho}{16}
\]
for $n \ge n_3$, which is \eqref{keyes:0}.
\end{proof}


\begin{proposition}\label{keyest}
Let $\{\gamma^n_{(D,a,b)}\}$ be weakly $\kappa$-admissible. There exist $c_\star < \infty$ and $n_* < \infty$, depending only on the admissibility data, $R$, and the diameter bound $R_1$ of the class, such that for all $n \ge \max\bigl(n_*,\ n_2({\mathcal K}, R),\ n_3(s, R),\ n_0(R/16, R_1)\bigr)$ the following holds. Let $k$ be a bounded ${\mathcal F}^n$-stopping index, let $k^+$ be as in \eqref{eq:mdef}, and set $W^n(t_{k^+})-W^n(t_k)$ and $t_{k^+}-t_k$ to $0$ off the event $\{\Crad^n_{t_k} \ge R\}$. Then, almost surely on this event,
\begin{align}
  &\Crad^n_{t'} \ge \tfrac1{16}\Crad^n_{t_k} \quad \text{for all } t' \in [t_k, t_{k^+}], \tag{\ref{keyes:0}}\\[2pt]
  &\bigl|\IE\bigl[W^n(t_{k^+}) - W^n(t_k)\; \bigm| \; {\mathcal F}^n_k\bigr]\bigr| \leq c_\star\, n^{-s}, \label{keyes:1} \\[2pt]
  &\bigl|\IE\bigl[(W^n(t_{k^+}) - W^n(t_k))^2 - \kappa\, (t_{k^+} - t_k) \; \bigm| \; {\mathcal F}^n_k \bigr]\bigr| \leq c_\star\, n^{-s}, \label{keyes:2}
\end{align}
where $s\in (0,1)$ is the exponent from Definition \ref{def:family}.
\end{proposition}

\begin{proof}
Abbreviate $\Delta W := W^n(t_{k^+}) - W^n(t_k)$ and $\Delta t := t_{k^+} - t_k$; by Lemma \ref{lem:onestep}, on $\{\Crad^n_{t_k} \ge R\}$ we have $|\Delta W| \le 2n^{-s/3}$, $\Delta t \le 2n^{-2s/3}$, and \eqref{keyes:0}.

\emph{Step 1: selection and tracking.} Fix $i \in \{1, 2\}$ and apply Lemma \ref{lem:select} at the index $k$ (valid on our event) to obtain an ${\mathcal F}^n_k$-measurable multivertex $\hat{p}^{(i)}$ with
\[
\bigl| \Xi^n_{t_k}(\hat{p}^{(i)}) - \mathfrak{v}^{(i)} \bigr| \le \varepsilon_0/2 .
\]
In the one-point case $\mathfrak{v}^{(1)} = \mathfrak{v}^{(2)}$, a single selection suffices and $\hat{p}^{(1)} = \hat{p}^{(2)}$.

Write $\hat\xi_t := \Xi^n_t(\hat{p}^{(i)})$ for $t \in [t_k, t_{k^+}]$, with interior components $\xi^r_t = \varphi^n_t(v^*_r) - W^n(t)$ and boundary components $\zeta^r_t = \varphi^n_t(\beta^*_r) - W^n(t)$. As long as $\hat\xi_t \in \overline{B}(\mathfrak{v}^{(i)}, \varepsilon_0) \subset {\mathcal K}$, every component satisfies $\Im \xi^r_t \ge d_0$, $|\zeta^r_t| \ge d_0$, so by the Loewner equation each component moves with speed at most $2/d_0$, while the common shift by the driving term displaces all of them by at most $2n^{-s/3}$. Let $\sigma$ be the first exit time from $\overline{B}(\mathfrak{v}^{(i)}, \varepsilon_0)$ after $t_k$. For $t < \sigma \wedge t_{k^+}$ the total displacement is at most $\sqrt{m+l}\bigl((2/d_0)\, \Delta t + 2 n^{-s/3}\bigr) < \varepsilon_0/2$ for sufficiently large $n$, so the configuration, having started within $\varepsilon_0/2$ of $\mathfrak{v}^{(i)}$, cannot reach the boundary of the ball before $t_{k^+}$. Hence $\sigma > t_{k^+}$, and for $n \ge n_*$,
\begin{equation}\label{eq:tracking}
\hat\xi_t \in \overline{B}(\mathfrak{v}^{(i)}, \varepsilon_0) \subset {\mathcal K} \subseteq {\mathcal V} \qquad \text{for all } t \in [t_k, t_{k^+}].
\end{equation}

\emph{Step 2: the almost-martingale identity for $\mathfrak{h}$.}
Set $G:=\{\Crad^n_{t_k}\ge R\}\in{\mathcal F}^n_k$. By \eqref{eq:tracking} and Step 1, the configuration of $\hat p^{(i)}$ stays in ${\mathcal K}\subseteq{\mathcal V}$ on $[t_k,t_{k^+}]$, with components at distance at least $d_0$ from $0$; so no point of $\hat p^{(i)}$ is hit or swallowed, and $\hat p^{(i)}$ is also an ${\mathcal F}^n_{t_{k^+}}$-measurable multivertex on $G$. As $\Gamma^n$ is finite and $\gamma^n(t_{k^+})\neq b^n$ on $G$, there $t_{k^+}\le C$ for a deterministic $C$. By Lemma~\ref{lem:onestep}, $\Crad_t^n\ge R/16$ throughout $[t_k,t_{k^+}]$ on $G$, so condition~\ref{cond:1} applies with $r=R/16$. Conditions \ref{cond:1} and \ref{cond:2b} thus apply on $G$ at the bounded stopping times $t_k\wedge C\le t_{k^+}\wedge C$, which agree with $t_k,t_{k^+}$ there, as do, by locality, the conditional expectations given ${\mathcal F}^n_{t_k\wedge C}$ and ${\mathcal F}^n_{t_k}$; condition \ref{cond:2b} is used with $\Crad^n_{t_k}\ge R$ and $\Crad^n_{t_{k^+}}\ge R/16$. Hence, almost surely on $G$,
\begin{equation}\label{eq:MarkovCont}
\IE\bigl[\mathbf 1_G\,\mathfrak{h}(\hat\xi_{t_{k^+}})\mid{\mathcal F}^n_k\bigr]
=\mathfrak{h}(\hat\xi_{t_k})+O(n^{-s}).
\end{equation}

\emph{Step 3: increment formulas.} For an interior component, the Loewner equation gives $\partial_t \varphi^n_t(v^*_r) = 2/\xi^r_t$, and by \eqref{eq:tracking}, $|\xi^r_t - \xi^r_{t_k}| \le (2/d_0)\Delta t + 2n^{-s/3} = O(n^{-s/3})$ with $|\xi^r_t|, |\xi^r_{t_k}| \ge d_0$; hence
\begin{equation}\label{eq:increment}
\Delta z^r := \varphi^n_{t_{k^+}}(v^*_r) - \varphi^n_{t_k}(v^*_r)
= \int_{t_k}^{t_{k^+}} \frac{2\, dt}{\xi^r_t}
= \frac{2\, \Delta t}{\xi^r_{t_k}} + O\bigl(n^{-\frac{2s}{3}} \cdot n^{-\frac{s}{3}}\bigr)
= \frac{2\, \Delta t}{\xi^r_{t_k}} + O(n^{-s}).
\end{equation}
The same formula holds for the boundary components: $\Delta c^r := \varphi^n_{t_{k^+}}(\beta^*_r) - \varphi^n_{t_k}(\beta^*_r) = 2\Delta t/\zeta^r_{t_k} + O(n^{-s})$, with real terms throughout.

\emph{Step 4: Taylor expansion in real coordinates.} The increment of the centered configuration is
\[
\hat\xi_{t_{k^+}} - \hat\xi_{t_k} = \eta, \qquad
\eta := -\Delta W \, e + \mathfrak{r},
\]
where $e$ is the common-shift direction (the vector with entry $1$ in every $x^r$- and $c^r$-coordinate and $0$ in every $y^r$-coordinate; the directional derivative along $e$ is $\Dop$) and $\mathfrak{r}$ collects the flow contributions $(\Re \Delta z^r, \Im \Delta z^r, \Delta c^r)$, so that $|\mathfrak{r}| = O(n^{-2s/3})$ and $|\eta| = O(n^{-s/3})$ by \eqref{eq:onestepincr}. Since the segment from $\hat\xi_{t_k}$ to $\hat\xi_{t_{k^+}}$ lies in the convex set $\overline{B}(\mathfrak{v}^{(i)}, \varepsilon_0) \subset {\mathcal K}$ and $\mathfrak{h} \in C^3$, Taylor's theorem with the bounds \eqref{eq:Khconstants} gives
\[
\mathfrak{h}(\hat\xi_{t_{k^+}}) - \mathfrak{h}(\hat\xi_{t_k})
= D\mathfrak{h} \cdot \eta + \tfrac12\, D^2\mathfrak{h}[\eta, \eta] + O(\Lambda_3 |\eta|^3),
\]
with all derivatives evaluated at the ${\mathcal F}^n_k$-measurable point $\hat\xi_{t_k}$, and $|\eta|^3 = O(n^{-s})$. We expand the two terms.
\begin{itemize}
\item \emph{First order.} $D\mathfrak{h} \cdot \eta = -\Dop \mathfrak{h} \cdot \Delta W + D\mathfrak{h} \cdot \mathfrak{r}$. By \eqref{eq:increment} and $\Re(2/\xi) = \frac{2x}{x^2+y^2}$, $\Im(2/\xi) = \frac{-2y}{x^2+y^2}$,
\[
D\mathfrak{h} \cdot \mathfrak{r} = \underbrace{\left( \sum_{r=1}^m \frac{2}{(x^r)^2 + (y^r)^2} \Bigl( x^r \frac{\partial \mathfrak{h}}{\partial x^r} - y^r \frac{\partial \mathfrak{h}}{\partial y^r} \Bigr) + \sum_{r=1}^l \frac{2}{\zeta^r} \frac{\partial \mathfrak{h}}{\partial c^r} \right)}_{=:\ A(\mathfrak{h})(\hat\xi_{t_k})} \Delta t \; + \; O(n^{-s}),
\]
which is exactly the first-order part of the BPZ equation \eqref{eq:BPZ}.
\item \emph{Second order.} $D^2\mathfrak{h}[\eta,\eta] = (\Delta W)^2\, D^2\mathfrak{h}[e,e] - 2\Delta W\, D^2\mathfrak{h}[e, \mathfrak{r}] + D^2\mathfrak{h}[\mathfrak{r},\mathfrak{r}]$. The last two terms are $O(n^{-s/3} \cdot n^{-2s/3}) + O(n^{-4s/3}) = O(n^{-s})$, and $D^2\mathfrak{h}[e,e] = \Dop^2 \mathfrak{h}$, the full square of the common-shift operator, including all mixed derivatives.
\end{itemize}
Altogether, with $B(\mathfrak{h}) = \Dop \mathfrak{h}$ and $C(\mathfrak{h}) = \tfrac12 \Dop^2 \mathfrak{h}$ as in \eqref{eq:BCdef},
\begin{equation}\label{eq:onestepexp}
\mathfrak{h}(\hat\xi_{t_{k^+}}) - \mathfrak{h}(\hat\xi_{t_k})
= A(\mathfrak{h})(\hat\xi_{t_k})\, \Delta t \; - \; B(\mathfrak{h})(\hat\xi_{t_k})\, \Delta W \; + \; C(\mathfrak{h})(\hat\xi_{t_k})\, (\Delta W)^2 \; + \; O(n^{-s}).
\end{equation}

\emph{Step 5: conclusion.} Multiplying \eqref{eq:onestepexp} by $\mathbf 1_G$ and taking $\IE[\,\cdot \mid {\mathcal F}^n_k]$, using \eqref{eq:MarkovCont} for the left-hand side and the BPZ identity $A(\mathfrak{h}) = -\kappa\, C(\mathfrak{h})$ (condition \ref{cond:2a}, in the form following \eqref{eq:BCdef}) on the right, we obtain, for each $i = 1, 2$,
\begin{equation}\label{eq:linsys}
\Bigl| C(\mathfrak{h})\bigl(\hat\xi^{(i)}_{t_k}\bigr) \cdot X \; + \; B(\mathfrak{h})\bigl(\hat\xi^{(i)}_{t_k}\bigr) \cdot Y \Bigr| \le c_2\, n^{-s},
\end{equation}
where
\[
X := \IE\bigl[ (\Delta W)^2 - \kappa \Delta t \bigm| {\mathcal F}^n_k \bigr],
\qquad
Y := -\IE\bigl[ \Delta W \bigm| {\mathcal F}^n_k \bigr]
\]
are real ${\mathcal F}^n_k$-measurable quantities, $\hat\xi^{(i)}_{t_k} := \Xi^n_{t_k}(\hat{p}^{(i)}) \in \overline{B}(\mathfrak{v}^{(i)}, \varepsilon_0)$, and $c_2$ depends only on the admissibility data and $R$. The hypotheses of Lemma \ref{lem:nondegquant} are met with $\tilde{\mathfrak{u}}^{(i)} = \hat\xi^{(i)}_{t_k}$ and $\varepsilon = c_2 n^{-s}$ --- in the two-point case, \eqref{eq:linsys} for $i = 1, 2$ supplies the two rows; in the one-point complex case, the single inequality \eqref{eq:linsys} supplies, through its real and imaginary parts, both rows of the rank-$2$ matrix. Therefore
\[
|X| + |Y| \le \frac{4\Lambda}{c}\, c_2\, n^{-s},
\]
which is \eqref{keyes:1} and \eqref{keyes:2} with $c_\star = \max\{1,(4\Lambda/c)\, c_2\}$.
\end{proof}

\begin{remark}\label{rem:keyestscheme}
When $\mathfrak v^{(1)}\ne\mathfrak v^{(2)}$, the two multivertex families are tracked separately, with constants controlled by the single compact set ${\mathcal K}$.
\end{remark}

\subsection{Proof of Driving Convergence Theorem \ref{drivingcvg}}

The proof has three ingredients: a deterministic lower bound on the slit-domain conformal radius along the stopped horizon, which makes Proposition \ref{keyest} applicable at every mesoscopic step (Lemma \ref{lem:basepoint} and Corollary \ref{cor:radhorizon}); the Skorokhod embedding theorem; and an abstract coupling lemma (Lemma \ref{lem:coupling}) that converts the conclusions of Proposition \ref{keyest} into a quantitative coupling of the driving process with Brownian motion.

\subsubsection{The conformal radius along the stopped horizon}

\begin{lemma}[Base-point estimate]\label{lem:basepoint}
There is an absolute constant $c_0 > 0$ with the following property. Let $K$ be a hull of a chordal Loewner chain with $K \cup \IR$ connected and $K \subset B(x_0, d)$ for some $x_0 \in \IR$ and $d \ge 1$, let $f = f_K : \IH \to \Omega := \IH \setminus K$ be the hydrodynamically normalized map, and let $x \in \IR$. Then
\[
\dist\bigl( f(x + i),\ \partial\Omega \bigr) \ \ge\ \frac{c_0}{(1+d)^2}
\qquad\text{and}\qquad
\bigl| f'(x+i) \bigr| \ \ge\ \frac{c_0}{(1+d)^2}.
\]
\end{lemma}

\begin{proof}
Let $z_* := f(x+i)$ and $\delta := \dist(z_*, \partial\Omega)$. Since $\Im g_K(z) \le \Im z$ on $\Omega$, we have $\Im z_* \ge \Im(x+i) = 1$; and since $K \subset B(x_0, d)$ with $x_0$ real, the standard estimate $|f(w) - w| \le 3d$ gives $\Im z_* \le 1 + 3d$.

If $\delta \ge 1/4$, both claims below hold with room to spare, so assume $\delta < 1/4$. Then the nearest boundary point $p$ to $z_*$ lies on $K$ (because $\dist(z_*, \IR) = \Im z_* \ge 1$). The set $K \cup \IR$ is closed and connected, contains $p$ with $|p - z_*| = \delta$, and is not contained in $B(z_*, 1/2)$ (again because $\dist(z_*, \IR) \ge 1$); hence the connected component of $(K \cup \IR) \cap \overline{B}(z_*, 1/2)$ containing $p$ joins $p$ to the sphere $\partial B(z_*, 1/2)$. By the Beurling projection theorem, a Brownian motion started at $z_*$ reaches $\partial B(z_*, 1/2)$ before hitting $\partial\Omega$ with probability at most $C\sqrt{\delta / (1/2)} = C' \sqrt{\delta}$.

Now apply the optional stopping theorem to the bounded harmonic function $u(z) := \Im g_K(z)$ on $\Omega \cap B(z_*, 1/2)$, at the exit time of this domain. On the $\partial\Omega$-part of the boundary $u$ vanishes; on the sphere part, $u(z) \le \Im z \le \Im z_* + \tfrac12 \le \tfrac32 + 3d$. Therefore
\[
1 = \Im g_K(z_*) = u(z_*) \ \le\ \Bigl( \tfrac32 + 3d \Bigr) \cdot C' \sqrt{\delta},
\]
which gives $\delta \ge c_0 (1+d)^{-2}$. Finally, Koebe's theorem gives $|f'(x+i)|\ge\delta/2$, proving the second claim after decreasing $c_0$.
\end{proof}

\begin{corollary}[Radius along the horizon]\label{cor:radhorizon}
There is a constant $c_\rho > 0$, depending only on $\rho$, such that almost surely
\[
\Crad^n_t \ \ge\ c_\rho\, R \qquad \text{for all } t \in [0,\, T^n_\rho],
\]
whenever $\Crad(\varphi^n) \ge R$.
\end{corollary}

\begin{proof}
Fix $t\le T^n_\rho$ and set $M:=2/\rho+1$. By \eqref{eq:Tnrho}, the trace and its hull $\tilde K_t$ lie in $B(0,M)$, and Lemma \ref{DiamKt} gives $|W^n(t)|\le CM$. Put $f^n_t:=(g^n_t)^{-1}$ and $z_t:=f^n_t(i+W^n(t))$. The chain rule gives
\begin{equation}\label{eq:radfactor}
\Crad^n_t=\bigl|((\varphi^n)^{-1})'(z_t)\bigr|\,\bigl|(f^n_t)'(i+W^n(t))\bigr|.
\end{equation}
Since $\tilde K_t\cup\IR$ is connected, Lemma \ref{lem:basepoint} bounds the second factor below by $c_0(1+M)^{-2}$. Also $1\le\Im z_t\le1+3M$ and $|z_t|\le CM$, by the bounds on $f^n_t$ used in that lemma. Thus $z_t$ lies in a fixed compact subset of $\IH$ depending only on $\rho$. Koebe distortion applied to $F:=(\varphi^n)^{-1}$ gives $|F'(z_t)|\ge c(\rho)|F'(i)|=c(\rho)\Crad(\varphi^n)$. Substitution in \eqref{eq:radfactor} proves the claim.
\end{proof}

\subsubsection{Probabilistic tools}

The coupling uses two standard results, in the form in which they appear in \cite[\S 6]{BJVK}. The first is the Skorokhod embedding for martingales: if $(M_k)_{k \le K}$ is an $(\mathcal{F}_k)$-martingale with $\|M_{k+1} - M_k\|_\infty \le \delta$ and $M_0 = 0$, there are stopping times $0 = \tau_0 \le \dots \le \tau_K$ for a standard Brownian motion $B$ with $(M_k)_{k \le K} \overset{d}{=} (B(\tau_k))_{k \le K}$ and, for $k < K$,
\begin{align}
&\IE[\tau_{k+1} - \tau_k \mid B[0,\tau_k]] = \IE[(B(\tau_{k+1}) - B(\tau_k))^2 \mid B[0,\tau_k]], \label{Ske1}\\
&\IE[(\tau_{k+1} - \tau_k)^p \mid B[0,\tau_k]] \le C_p\, \IE[(B(\tau_{k+1}) - B(\tau_k))^{2p} \mid B[0,\tau_k]], \label{Ske2}
\end{align}
together with $\tau_{k+1} \le \inf\{ t \ge \tau_k : |B(t) - B(\tau_k)| \ge \delta \}$ (\cite[Lemma 6.1]{BJVK}; (\ref{Ske1})--(\ref{Ske2}) are Theorem~A.1 of \cite{HallHeyde}, the exit bound follows from its proof). The second is the Brownian modulus-of-continuity estimate \cite[Lemma 6.2]{BJVK} (Lemma 1.2.1 of \cite{revesz81}): for each $\epsilon > 0$ there is $C = C(\epsilon)$ with
\begin{equation}\label{eq:bmodulus}
\IP\Bigl( \sup_{t \in [0, T-h]} \sup_{s \in (0,h]} |B(t+s) - B(t)| \le v\sqrt{h} \Bigr) \ge 1 - \frac{CT}{h}\, e^{-v^2/(2+\epsilon)}
\end{equation}
for all positive $v, T$ and $0 < h < T$.

We also use the following maximal inequality for martingale differences with conditionally sub-exponential tails (cf.\ \cite[Lemma 6.3]{BJVK}).

\begin{lemma}[Bernstein inequality, maximal form]\label{lem:bernstein}
There are absolute constants $c_1, C_1$ with the following property. Let $(D_j)_{j < N}$ be martingale differences with respect to a filtration $({\mathcal H}_j)$ such that $\IE\bigl[ e^{|D_j| / \beta} \bigm| {\mathcal H}_j \bigr] \le 2$ for all $j$ and some $\beta > 0$. Then for every $x > 0$,
\[
\IP\Bigl( \max_{i \le N} \Bigl| \sum_{j < i} D_j \Bigr| \ge x \Bigr) \ \le\ 2 \exp\Bigl( - c_1 \min\Bigl( \frac{x^2}{N \beta^2},\ \frac{x}{\beta} \Bigr) \Bigr).
\]
\end{lemma}

\begin{proof}
The moment bound gives $\IE[e^{\lambda D_j} \mid {\mathcal H}_j] \le e^{C_1 \lambda^2 \beta^2}$ for $|\lambda| \le 1/(2\beta)$ (expand and use $\IE[D_j \mid {\mathcal H}_j] = 0$ together with $\IE[|D_j|^p \mid {\mathcal H}_j] \le 2\, p!\, \beta^p$). Hence $\exp\bigl( \lambda \sum_{j<i} D_j - C_1 \lambda^2 \beta^2 i \bigr)$ is a supermartingale; Doob's maximal inequality and optimization over $|\lambda| \le 1/(2\beta)$ give the stated bound, first for the one-sided maxima and then, by symmetry, for the absolute value.
\end{proof}

\subsubsection{The abstract coupling lemma}

\begin{lemma}[Martingale-to-Brownian coupling with rates]\label{lem:coupling}
Let $\kappa > 0$, $c_\star \ge 1$, $T \ge 1$, and $\varepsilon \in (0, \varepsilon_1)$, where $\varepsilon_1 = \varepsilon_1(\kappa, c_\star, T) > 0$ is specified in the proof. Suppose that on a filtered probability space $({\mathcal G}_j)_{j \ge 0}$ we are given ${\mathcal G}_{j+1}$-measurable increments $\Delta W_j \in \IR$, $\Delta \tau_j \ge 0$, and a $({\mathcal G}_j)$-stopping index $J$ with the following properties, almost surely for every $j < J$:
\begin{enumerate}
\item[(H1)] \emph{(step size)} $|\Delta W_j| \le 2\varepsilon$ and $\Delta\tau_j \le 2 \varepsilon^2$;
\item[(H2)] \emph{(mesoscopic non-degeneracy)} $\Delta\tau_j \ge \varepsilon^2$ or $|\Delta W_j| \ge \varepsilon$;
\item[(H3)] \emph{(drift)} $\bigl| \IE[\Delta W_j \mid {\mathcal G}_j] \bigr| \le c_\star\, \varepsilon^3$;
\item[(H4)] \emph{(conditional variance)} $\bigl| \IE[(\Delta W_j)^2 - \kappa\, \Delta\tau_j \mid {\mathcal G}_j] \bigr| \le c_\star\, \varepsilon^3$;
\end{enumerate}
and such that $\tau_i < \widetilde{T} \le \tau_J$ for all $i < J$, for a random variable $\widetilde{T} \le T$; that is, $J$ is the first index at which $\tau_i$ reaches $\widetilde T$. Here $\tau_i := \sum_{j < i} \Delta\tau_j$ and $W_i := \sum_{j < i} \Delta W_j$. Set $N := \lceil 8 (1+\kappa) c_\star T\, \varepsilon^{-2} \rceil$. Then, on an enlargement of the probability space, there exists a standard Brownian motion $\widehat{B}$ such that, outside an event of probability at most $\exp\bigl( -\tfrac15 \log^2 (1/\varepsilon) \bigr)$, the following hold simultaneously:
\begin{enumerate}
\item[(C1)] $J \le N$;
\item[(C2)] $\displaystyle \sup_{i \le J} \bigl| W_i - \widehat{B}(\kappa \tau_i) \bigr| \ \le\ C_2 (1+T)\, \varepsilon^{1/5}$;
\item[(C3)] $\displaystyle \sup_{0 \le t \le \kappa T} \ \sup_{0 < u \le 2\kappa\varepsilon^2} \bigl| \widehat{B}(t + u) - \widehat{B}(t) \bigr| \ \le\ C_2\, \varepsilon^{1/5},$
\end{enumerate}
with $C_2 = C_2(\kappa, c_\star) < \infty$.
\end{lemma}

The proof follows \cite[\S 6]{BJVK} for loop-erased random walk at $\kappa = 2$, with $\kappa$ and the increment scale $\varepsilon$ as free parameters. Lemma~\ref{lem:onestep} supplies (H1), the stopping rule \eqref{eq:mdef} gives (H2), and Proposition~\ref{keyest} gives (H3)--(H4).

\begin{proof}
Throughout, $C$ denotes a constant depending only on $\kappa$ and $c_\star$, possibly changing from line to line. Set $\Delta W_j=\Delta\tau_j=0$ for $j\ge J$.

\emph{Step 1: number of mesoscopic increments.} For $j < J$, (H2) gives $\Delta\tau_j + (\Delta W_j)^2 \ge \varepsilon^2$ pointwise, while (H4) gives $\IE[(\Delta W_j)^2 \mid {\mathcal G}_j] \le \kappa\, \IE[\Delta\tau_j \mid {\mathcal G}_j] + c_\star \varepsilon^3$. Hence $\IE[\Delta\tau_j \mid {\mathcal G}_j] \ge (\varepsilon^2 - c_\star\varepsilon^3)/(1+\kappa) \ge \varepsilon^2/(2(1+\kappa)) =: \mu$ on $\{j < J\}$, for $\varepsilon \le 1/(2 c_\star)$. The differences $D_j := \bigl( \Delta\tau_j - \IE[\Delta\tau_j \mid {\mathcal G}_j] \bigr)\mathbf{1}_{\{j<J\}}$ are martingale differences bounded by $2\varepsilon^2$ (note $\{j < J\} \in {\mathcal G}_j$). On $\{J > N\}$ every $j < N$ satisfies $j < J$, so $\tau_N \ge \sum_{j<N} D_j + N\mu \ge \sum_{j<N} D_j + 4T$ by the choice of $N$, while $\tau_N < \widetilde{T} \le T$ by the choice of $J$; hence $\{J > N\} \subseteq \{ \sum_{j<N} D_j \le -3T \}$, and by Azuma's inequality
\begin{equation}\label{eq:stepcount}
\IP(J > N) \ \le\ \exp\Bigl( - \frac{9T^2}{2 N (2\varepsilon^2)^2} \Bigr) \ \le\ \exp\Bigl( - \frac{c\, T}{\varepsilon^{2}} \Bigr).
\end{equation}

\emph{Step 2: martingale correction and accumulated drift.} Let $\delta_j := \IE[\Delta W_j \mid {\mathcal G}_j] \mathbf{1}_{\{j < J\}}$ and $\Delta M_j := (\Delta W_j - \delta_j) \mathbf{1}_{\{j < J\}}$, so that $M_i := \sum_{j<i} \Delta M_j$ is a $({\mathcal G}_j)$-martingale with $|\Delta M_j| \le 2\varepsilon + c_\star \varepsilon^3 \le 3\varepsilon$, and on $\{J \le N\}$ the accumulated drift satisfies
\begin{equation}\label{eq:driftacc}
\sup_{i \le J} \bigl| W_i - M_i \bigr| \ =\ \sup_{i \le J} \Bigl| \sum_{j < i} \delta_j \Bigr| \ \le\ N c_\star \varepsilon^3 \ \le\ C\, T\, \varepsilon .
\end{equation}

\emph{Step 3: Skorokhod embedding.} Conditionally on ${\mathcal G}_i$, apply the Skorokhod embedding \eqref{Ske1}--\eqref{Ske2} to the law of $\Delta M_i$, with $\delta=3\varepsilon$, using a fresh Brownian motion. Disintegrate the stopped segment with respect to its terminal value and sample it, on an extension, from this kernel at the original increment $\Delta M_i$. This preserves the discrete data. Earlier segment kernels depend only on ${\mathcal G}_i$, so the additional Brownian past does not change the conditional law of the next discrete increment. The stopped segments, followed by an independent Brownian continuation, form a Brownian motion $\widehat B$ in a continuous-time filtration with stopping times $0=\sigma_0\le\cdots\le\sigma_N$ satisfying
$$\sigma_{i+1}\le\inf\{t\ge\sigma_i:|\widehat B(t)-\widehat B(\sigma_i)|\ge3\varepsilon\}.$$
Writing ${\mathcal H}_i:=\sigma({\mathcal G}_i,\sigma_0,\ldots,\sigma_i,\widehat B[0,\sigma_i])$, we have $M_i=\widehat B(\sigma_i)$ almost surely, \eqref{Ske1}--\eqref{Ske2} hold conditionally on ${\mathcal H}_i$, and the ${\mathcal H}_i$-conditional laws of the discrete increments agree with their ${\mathcal G}_i$-conditional laws.

\emph{Step 4: Skorokhod clock versus capacity clock.} Decompose, for $i \le N$,
\begin{align*}
\sigma_i - \kappa \tau_i \ &=\ \underbrace{\sum_{j<i} \Bigl( (\sigma_{j+1} - \sigma_j) - \IE[\sigma_{j+1} - \sigma_j \mid {\mathcal H}_j] \Bigr)}_{S^{(1)}_i} \\
&\ +\ \underbrace{\sum_{j<i} e_j}_{S^{(2)}_i}
\ +\ \underbrace{\kappa \sum_{j<i} \Bigl( \IE[\Delta\tau_j \mid {\mathcal G}_j] - \Delta\tau_j \Bigr)}_{S^{(3)}_i},
\end{align*}
where, by \eqref{Ske1} and (H3)--(H4),
\begin{align*}
e_j &:= \IE[\sigma_{j+1} - \sigma_j \mid {\mathcal H}_j] - \kappa\, \IE[\Delta\tau_j \mid {\mathcal G}_j] \\
&= \IE[(\Delta M_j)^2 \mid {\mathcal G}_j] - \kappa\, \IE[\Delta\tau_j \mid {\mathcal G}_j] \\
&= \IE[(\Delta W_j)^2 \mid {\mathcal G}_j] - \delta_j^2 - \kappa\, \IE[\Delta\tau_j \mid {\mathcal G}_j],
\end{align*}
so that $|e_j| \le c_\star \varepsilon^3 + c_\star^2 \varepsilon^6 \le 2 c_\star \varepsilon^3$ on $\{j < J\}$ and $e_j = 0$ otherwise; hence
\begin{equation}\label{eq:clocksys}
\sup_{i \le N} |S^{(2)}_i| \ \le\ 2 N c_\star \varepsilon^3 \ \le\ C\, T\, \varepsilon .
\end{equation}
For $S^{(3)}$: the increments are $({\mathcal G}_j)$-martingale differences bounded by $4\kappa\varepsilon^2$; Azuma's inequality in maximal form gives
\begin{equation}\label{eq:clockS3}
\IP\Bigl( \sup_{i \le N} |S^{(3)}_i| \ge (1+T)\, \varepsilon^{1/2} \Bigr) \ \le\ 2 \exp\Bigl( - \frac{(1+T)^2 \varepsilon}{2 N (4\kappa\varepsilon^2)^2} \Bigr) \ \le\ 2 \exp\Bigl( - \frac{c}{\varepsilon} \Bigr).
\end{equation}
For $S^{(1)}$: by the exit bound of the Skorokhod embedding, $\sigma_{j+1} - \sigma_j$ is dominated by the exit time of $\widehat{B}$ from an interval of half-width $3\varepsilon$, whose exponential moment at the scale $\beta := C\varepsilon^2$ is at most $2$; the same then holds, with an adjusted constant, for the centered increments. Lemma \ref{lem:bernstein} with $x = (1+T)\varepsilon^{1/2}$, $\beta = C\varepsilon^2$ gives
\begin{equation}\label{eq:clockS1}
\begin{aligned}
\IP\Bigl( \sup_{i \le N} |S^{(1)}_i| \ge (1+T)\, \varepsilon^{1/2} \Bigr)
&\ \le\ 2 \exp\Bigl( - c \min\Bigl( \frac{(1+T)^2\varepsilon}{N \varepsilon^4},\ \frac{(1+T)\varepsilon^{1/2}}{\varepsilon^2} \Bigr) \Bigr) \\
&\ \le\ 2 \exp\Bigl( - \frac{c}{\varepsilon} \Bigr).
\end{aligned}
\end{equation}
Combining \eqref{eq:clocksys}--\eqref{eq:clockS1}: outside an event of probability $\le 4 e^{-c/\varepsilon}$,
\begin{equation}\label{eq:clock}
h \ :=\ \sup_{i \le N} \bigl| \sigma_i - \kappa \tau_i \bigr| \ \le\ C_h\, (1+T)\, \varepsilon^{1/2}.
\end{equation}
By (H1), $\tau_J<\widetilde T+2\varepsilon^2\le T+2\varepsilon^2$. Hence, on the event of \eqref{eq:clock}, with $J\le N$, both $\kappa\tau_i$ and $\sigma_i$ are at most $\kappa T+2\kappa\varepsilon^2+h\le\kappa T+1$ for $i\le J$, after decreasing $\varepsilon_1$.

\emph{Step 5: modulus of continuity and assembly.} Apply the Brownian modulus-of-continuity estimate \eqref{eq:bmodulus} to $\widehat{B}$ on $[0, \kappa T + 2]$ at the scale $\bar{h} := C_h (1+T) \varepsilon^{1/2}$ with $v := \log(1/\varepsilon)$: the exceptional probability is at most
\[
\frac{C (\kappa T + 2)}{\bar h}\, e^{-\log^2(1/\varepsilon)/3} \ \le\ e^{-\log^2(1/\varepsilon)/4}
\qquad (\varepsilon \le \varepsilon_1(\kappa, c_\star, T)),
\]
and on its complement,
\begin{equation}\label{eq:osc}
\sup_{0 \le t \le \kappa T + 2 - \bar h}\ \sup_{0 < u \le \bar h} \bigl| \widehat{B}(t+u) - \widehat{B}(t) \bigr| \ \le\ \log(1/\varepsilon)\, \sqrt{\bar h} \ \le\ C\, \varepsilon^{1/5}
\end{equation}
after decreasing $\varepsilon_1(\kappa,c_\star,T)$, since $(1+T)^{1/2}\varepsilon^{1/20}\log(1/\varepsilon)\to0$. Since $2\kappa\varepsilon^2 \le \bar h \le 1$, \eqref{eq:osc} contains (C3), and it applies to every pair $\kappa\tau_i,\sigma_i \le \kappa T + 1$ with $i \le J$. On the intersection of the events of \eqref{eq:stepcount}, \eqref{eq:clock}, and \eqref{eq:osc}, we get for $i \le J \le N$:
\[
\bigl| W_i - \widehat{B}(\kappa\tau_i) \bigr|
\ \le\ \underbrace{\bigl| W_i - M_i \bigr|}_{\le\, CT\varepsilon \ \text{by } \eqref{eq:driftacc}}
\ +\ \underbrace{\bigl| \widehat{B}(\sigma_i) - \widehat{B}(\kappa\tau_i) \bigr|}_{\le\, \eqref{eq:osc} \ \text{at scale } h \le \bar h}
\ \le\ C_2 (1+T)\, \varepsilon^{1/5},
\]
which is (C2). The total exceptional probability is at most $e^{-cT/\varepsilon^2} + 4 e^{-c/\varepsilon} + e^{-\log^2(1/\varepsilon)/4} \le \exp\bigl( - \tfrac15 \log^2(1/\varepsilon) \bigr)$ for $\varepsilon \le \varepsilon_1(\kappa, c_\star, T)$.
\end{proof}

\subsubsection{Proof of the theorem}

\emph{Warschawski's inverse-map estimate.} We record the form of \cite[Theorem~VII(b)]{Wars} used twice below. Let $S\subset R$ be bounded simply-connected regions containing a basepoint $w_0$, with boundaries $B_S,B_R$, and let $\sigma:=\dist(w_0,B_S)$. For a crosscut $c$ of $S$ not passing through $w_0$, let $T_c$ be the component of $S\setminus c$ not containing $w_0$, and set
$$
\eta_S(r):=\sup_{\diam(c)<r}\diam(T_c),
$$
the \emph{structure modulus} of $S$ at $w_0$. For boundaries $B_1,B_2$ of regions $R_1,R_2$, put
\begin{align*}
d_1&:=\sup_{P\in B_1\cap R_2}\dist(P,B_2), &
d_2&:=\sup_{Q\in B_2\cap R_1}\dist(Q,B_1),\\
D_i(B_1,B_2)&:=\max\{d_1,d_2\}.
\end{align*}
with the supremum over an empty set equal to $0$. If $D_i(B_R,B_S)<\varepsilon$, $0<\varepsilon<1$, and $\varepsilon<\sigma/64$, and if $\phi_R:R\to\ID$ and $\phi_S:S\to\ID$ are normalized by $\phi_R(w_0)=\phi_S(w_0)=0$ and positive derivative there, then
$$
\sup_{w\in S}|\phi_R(w)-\phi_S(w)|
\le L\bigl[\eta_S(\varepsilon^{1/5})\bigr]^{1/2},
$$
where $L$ depends only on $\sigma$ and $\diam(R)$.

We will need the following rough lower bound on conformal radius of discrete approximation.
\begin{lemma}\label{lem:discreteCrad}
Let $(D^n, a^n, b^n)$ be a discrete approximation of $(D, a, b)$, and $\varphi^n$ be the corresponding conformal maps. There exists $N_{\rm cr}=N_{\rm cr}(R,R_1)$ such that, whenever
$\Crad(\varphi)>R$, $\diam(D)\le R_1,$
then, for all $n\ge N_{\rm cr}$,
\[
\Crad(\varphi^n)\ge \frac{R}{54}.
\]
\end{lemma}

\begin{proof}
Let $\mathfrak p$ and  $\mathfrak p^n$ be the corresponding conformal maps to $\ID$. Set $c:=\mathfrak p^{-1}(0)$ and $q_n:=\mathfrak p^n(c)$. Since $\varphi=\Phi\circ\mathfrak p$ with $\Phi(0)=i$ and $|(\Phi^{-1})'(i)|=\frac12$, the chain rule gives
$
\Crad(\varphi)=\frac12|(\mathfrak p^{-1})'(0)|>R,
$
so Koebe's theorem gives
$
\dist(c,\partial D)>\frac R2.
$
By Definition~\ref{def:approx}(4),
\[
\dist(c,\partial D^n)\ge\frac R2-\frac1n.
\]
Put $S_n:=\mathfrak p(D^n)$ and $\delta_n:=C(R,R_1)n^{-1/2}$. By \cite[Theorem III]{Wars}, applied to circular crosscuts enclosing the connections in Definition \ref{def:approx}(4), their images under $\mathfrak p$ have diameter at most $\delta_n/2$, after increasing $C(R,R_1)$ to absorb the disk-cap factor $2$. Thus every point of $\partial S_n$ is joined to $\partial\ID$ outside $S_n$ by a connection of diameter at most $\delta_n$, and $\dist(0,\partial S_n)\ge1-\delta_n$. Let $\eta_n$ be the structure modulus of $S_n$ at $0$ in the notation above. Closing a crosscut of diameter at most $r$ with the two connections, truncated at their first meeting if they intersect, gives a crosscut of $\ID$ or a Jordan loop of diameter at most $r+2\delta_n$. For $r+2\delta_n<1/4$, the component cut off from $0$ lies in the small disk cap or inside the loop. Hence $\eta_n(r)\le C(r+\delta_n)$.

Let $g_n:S_n\to\ID$ satisfy $g_n(0)=0$ and $g_n'(0)>0$. Applying \cite[Theorem VII(b)]{Wars} with $\varepsilon=2\delta_n$ gives, for large $n$,
$$
\sup_{z\in S_n}|g_n(z)-z|
\le C\bigl[\eta_n((2\delta_n)^{1/5})\bigr]^{1/2}
\le C(R,R_1)n^{-1/20}.
$$
Let $\zeta_-$ and $\zeta_+$ be the boundary values of $g_n\circ\mathfrak p$ at the selected prime ends $a^n$ and $b^n$. Condition 5 and the same crosscut estimate give $\mathfrak p(w)=\pm1+O(\delta_n)$ as $w$ approaches these prime ends inside $D^n$. Taking limits in the last inequality gives $\zeta_\pm=\pm1+O(n^{-1/20})$.
The disk automorphism $A_n:=\mathfrak p^n\circ(g_n\circ\mathfrak p)^{-1}$ sends $\zeta_\pm$ to $\pm1$ and $0$ to $q_n$. The minimizing normalization, under postcomposition by $(z+u)/(1+uz)$ with $u\in(-1,1)$, gives $\Re q_n=0$, so the disk automorphism formula yields
$$
\frac{4|q_n|}{1+|q_n|^2}=|\zeta_-+\zeta_+|
\le C(R,R_1)n^{-1/20},
\qquad |q_n|\le C(R,R_1)n^{-1/20}.
$$
Choose $N=N_{\rm cr}(R,R_1)$ large enough for these bounds, $N^{-1}\le R/6$, and $C(R,R_1)N^{-1/20}\le1/4$.

Writing $\psi_n=(\mathfrak p^n)^{-1}$, Schwarz's lemma gives
\[
|\psi_n'(q_n)|(1-|q_n|^2)\ge \dist(c,\partial D^n)\ge R/3,
\qquad
\Crad(\varphi^n)=\tfrac12|\psi_n'(0)|
\ge \tfrac{R}{54},
\]
where the last inequality follows from Koebe distortion and $|q_n|\le1/4$.
\end{proof}

\begin{proof}[Proof of Theorem \ref{drivingcvg}]
Let $\varepsilon := n^{-s/3}$, so that the error $n^{-s}$ of Proposition \ref{keyest} equals $\varepsilon^3$. Let $R' := c_\rho R/54$, with $c_\rho$ from Corollary \ref{cor:radhorizon}, and let
\[
\begin{aligned}
n_\star := \max\Bigl(N_{\rm cr}(R,R_1),\ n_*,\ n_0(R'/16, R_1),\ n_2({\mathcal K}, R'),\ n_3(s, R'), {}\\
\varepsilon_1(\kappa, c_\star, T)^{-3/s},\ n_4(s, T, \kappa, c_\star) \Bigr),
\end{aligned}
\]
where $n_*, n_2, n_3$ are the thresholds of Proposition \ref{keyest} and Lemmas \ref{lem:select}, \ref{lem:onestep} at the radius level $R'$, and $n_4$ absorbs the final numerical comparisons below.

\emph{Mesoscopic scheme.} Define stopping indices $k_0 := 0$ and, recursively, $k_{j+1} := (k_j)^+$ as in \eqref{eq:mdef} whenever $\tau_j := t_{k_j} < T^n_\rho$, and $k_{j+1} := k_j$ (frozen) otherwise. Let $J := \min\{ j : \tau_j \ge T^n_\rho \}$, ${\mathcal G}_j := {\mathcal F}^n_{k_j}$, $\Delta W_j := W^n(\tau_{j+1}) - W^n(\tau_j)$, $\Delta\tau_j := \tau_{j+1} - \tau_j$, and $\widetilde{T} := T^n_\rho \le T$. Since $n\geq N_{\rm cr}(R,R_1)$, Lemma \ref{lem:discreteCrad} guarantees that $\Crad(\varphi^n) \ge R/54$. By Corollary \ref{cor:radhorizon}, on $\{ j < J \}$ we have $\Crad^n_{\tau_j} \ge R'$ almost surely; hence, for $n \ge n_\star$, Lemma \ref{lem:onestep} and Proposition \ref{keyest} apply at every unfrozen step and yield exactly the hypotheses of Lemma \ref{lem:coupling}: (H1) is \eqref{eq:onestepincr}; (H2) holds because every unfrozen step is a full step of \eqref{eq:mdef}, the alternative $k^+ = l(\gamma^n)$ being excluded by Lemma \ref{lem:onestep}; (H3) is \eqref{keyes:1}; (H4) is \eqref{keyes:2}, both with $c_\star$ the constant of Proposition \ref{keyest} at the radius level $R'$. Finally, $J$ is by definition the first index with $\tau_J \ge \widetilde{T}$, so $\tau_i < \widetilde T \le \tau_J$ for $i < J$, and $J$ is a $({\mathcal G}_j)$-stopping index.

\emph{Coupling.} Lemma \ref{lem:coupling} provides, on an enlarged space, a standard Brownian motion $\widehat{B}$ such that outside an event $E_n$ of probability at most $\exp\bigl( -\tfrac{s^2}{45} \log^2 n \bigr)$, the conclusions (C1)--(C3) hold. Set $B(t) := \kappa^{-1/2}\,\widehat{B}(\kappa t)$, a standard Brownian motion by Brownian scaling, and $W(t) := \widehat{B}(\kappa t) = \sqrt\kappa\, B(t)$; this is the coupling asserted in the statement.

\emph{Interpolation to all times.} Outside $E_n$, take $t\in[0,T^n_\rho]$ and $i<J$ with $t\in[\tau_i,\tau_{i+1}]$. By \eqref{eq:onestepincr}, $t-\tau_i\le2\varepsilon^2$, so \eqref{eq:onestepincr} and (C2)--(C3) give
$$
\begin{aligned}
|W^n(t)-W(t)|
&\le |W^n(t)-W^n(\tau_i)|+|W_i-\widehat B(\kappa\tau_i)|
   +|\widehat B(\kappa\tau_i)-\widehat B(\kappa t)|\\
&\le2\varepsilon+C_2(2+T)\varepsilon^{1/5}.
\end{aligned}
$$
Thus
\[
\sup_{t \in [0, T^n_\rho]} \bigl| W^n(t) - W(t) \bigr| \ \le\ C_3\, (1+T)\, \varepsilon^{1/5} \ =\ C_3 (1+T)\, n^{-s/15} \ \le\ n^{-s/16}
\]
for $n \ge n_4(s, T, \kappa, c_\star)$, while $\IP(E_n) \le \exp\bigl( -\tfrac{s^2}{45}\log^2 n \bigr) \le n^{-s/16}$ for $n \ge n_4$. This proves the theorem with $r = s/16$.
\end{proof}

\section{From Convergence of Driving Terms to Convergence of Paths}\label{sec:path}
We pass from convergence of driving terms to convergence of paths using the framework of \cite{V}. Its geometric input is the \emph{tip structure modulus}.

\begin{definition} \label{def:tip}
 For $s,t\in [0,T]$ with $s\leq t$, we let $\gamma_{s,t}$ denote the arc of the curve $\gamma(r)$ for $r\in[s,t]$. Let $H_t$ be the unbounded component of $\IH\setminus\gamma_{0,t}$.
 Let $S_{t,\delta}$ be the collection of crosscuts $\cC$ of $H_t$, of diameter at most $\delta$ that separate $\gamma(t)$ from $\infty$ in $H_t$. For a crosscut $\cC\in S_{t,\delta}$,
 \[
  s_{\cC}: = \inf \{s>0 : \gamma[t-s,t]\cap \overline{\cC}\neq \emptyset \}, \; \; \gamma_{\cC} := (\gamma(r), r\in[t-s_{\cC}, t]).
 \]
 Thus $t-s_\cC$ is the last contact with $\overline\cC$ before $t$; set $s_\cC=t$ if there is no contact.
 For $\delta > 0$, the \emph{tip structure modulus} of $(\gamma(t), t\in[0,T])$ in $H$, is defined as
$$\eta_{\text{tip}}(\delta)=\max\left(\delta,\sup_{t\in[0,T]} \sup_{\cC\in S_{t,\delta}} \diam(\gamma_{\cC})\right).$$
 In some sense, $\eta_{\text{tip}}(\delta)$ is the maximal distance the curve travels into a fjord with an  opening smaller than $\delta$ when viewed from the point toward which the curve is growing; see Figure~\ref{fig:tip-modulus}.
\end{definition}

\emph{The modulus in $\ID$.} For a curve $\gamma$ in $\ID$ stopped when it first reaches distance $\rho$ from $1$, the tip structure modulus $\eta^{\ID}_{\mathrm{tip}}$ is defined as in Definition~\ref{def:tip}, with $H_t$ the component of $\ID\setminus\gamma_{0,t}$ containing $B(1,\rho)\cap\ID$, and $S_{t,\delta}$ the crosscuts of $H_t$ of diameter at most $\delta$ that separate $\gamma(t)$ from the whole of $B(1,\rho)\cap\ID$; it is used only for $\delta<\rho$.

\emph{Cayley comparison.} Since
\[
\Phi^{-1}(z)=\frac{z-i}{z+i},\qquad
|\Phi^{-1}(z)-\Phi^{-1}(w)|
=\frac{2|z-w|}{|z+i|\,|w+i|},
\]
if $0<\sigma\le1$ and $d_*(z,\infty),d_*(w,\infty)\ge\sigma$, then
\[
\frac{\sigma^2}{2}|z-w|
\le d_*(z,w)\le2|z-w|.
\]
Let a curve $\gamma$ in $\overline{\IH}$ be stopped when $\gamma_{\ID}:=\Phi^{-1}\circ\gamma$ first reaches distance $\sigma$ from $1$, that is, when $d_*(\gamma,\infty)\le\sigma$, and let $\delta<\sigma/4$. A crosscut of $H_t$ of diameter at most $\delta$ separating $\gamma(t)$ from $\infty$ lies within $2\delta$ of $\gamma[0,t]$: it has an endpoint on $\gamma[0,t]$, or endpoints on $\IR$ on either side of $0$, as $\gamma[0,t]$ joins $0$ to the tip without meeting it. Its image under $\Phi^{-1}$ therefore has diameter at most $2\delta$ and misses $B\bigl(1,\sigma/(1+\sigma\delta)\bigr)\supseteq B(1,\sigma/2)$, so it separates $\gamma_{\ID}(t)$ from $B(1,\sigma/2)\cap\ID$. Comparing the diameters of the arcs $\gamma_{\cC}$ by the left inequality,
\[
\eta_{\mathrm{tip}}^{\IH}(\delta)
\le \frac{2}{\sigma^2}\,
\eta_{\mathrm{tip}}^{\ID}(2\delta),
\]
where on the right $\gamma_{\ID}$ is stopped when it first reaches distance $\sigma/2$ from $1$.

\begin{figure}[!ht]
\centering
\begin{tikzpicture}[x=0.68cm,y=0.62cm,>=Stealth,
  every node/.style={font=\footnotesize},line width=0.7pt]
\draw (0,0) .. controls (1.0,0.7) and (1.8,1.7) .. (3.05,3.45)
  .. controls (4.3,5.2) and (6.5,7.0) .. (7.9,7.55)
  .. controls (9.1,8.05) and (9.9,7.3) .. (10.0,6.3)
  .. controls (10.1,5.4) and (10.5,4.5) .. (9.5,3.9)
  .. controls (8.8,3.45) and (7.5,3.1) .. (6.9,3.0)
  .. controls (6.4,2.9) and (6.4,3.35) .. (5.9,2.95)
  .. controls (4.5,1.8) and (3.4,1.3) .. (3.0,1.55)
  .. controls (2.5,1.95) and (2.9,2.85) .. (3.38,3.13)
  .. controls (4.6,3.8) and (6.7,4.25) .. (7.8,4.95)
  .. controls (8.4,5.3) and (8.75,5.8) .. (8.45,5.95)
  .. controls (7.95,6.2) and (7.95,5.55) .. (7.55,5.6);
\draw[red,thick] (3.05,3.45) .. controls (2.95,3.15) and (3.2,3.0) .. (3.38,3.13);
\draw[<->] (3.38,3.13) -- (8.58,5.83)
  node[midway,above,sloped,fill=white,inner sep=1.5pt] {$\diam(\gamma_{\mathcal C})$};
\fill (3.38,3.13) circle (1.3pt);
\fill (7.55,5.6) circle (1.3pt);
\node[red,left] at (2.95,3.13) {$\mathcal C$};
\node[below right] at (3.38,3.13) {$\gamma(t-s_{\mathcal C})$};
\node[above left] at (7.55,5.6) {$\gamma(t)$};
\node at (7.65,4.05) {$\gamma_{\mathcal C}$};
\end{tikzpicture}
\caption{Illustration of the definition of the tip structure modulus:
a crosscut $\mathcal{C}$ of diameter $\le \delta$ that separates the
tip $\gamma(t)$ from $\infty$ defines a ``fjord''; the portion
$\gamma_{\mathcal{C}}$ of the curve inside that fjord has diameter
$\diam(\gamma_{\mathcal{C}})$, and $\eta_{\mathrm{tip}}(\delta)$
records the maximal such penetration depth.}\label{fig:tip-modulus}
\end{figure}

We use the following form of \cite[Lemma 3.4]{V}, whose proof uses the tip structure modulus at the scale $\bar\delta$ of condition 3.

\begin{lemma}\label{Vmain}
Consider two $\IH$-Loewner pairs $(f_j, W_j),\ j=1,2$ generated by the curves $\gamma_j$, where $f_j$ satisfies the Loewner equation with continuous driving term $W_j$, see \eqref{rLE} for definition. Fix $T<\infty$ and $0<T'\le T$. Assume that there exist $0\le\beta<1$, $L>1+\beta$, $r\in (0,1)$, $p\in (0,1/L)$ and $\epsilon>0$ such that the following holds.
\begin{enumerate}
 \item The driving terms satisfy:
 \[
  \sup_{t\in[0,T']} |W_1(t) - W_2(t)| \leq \epsilon.
 \]
 \item For some $c'<\infty$,
 \[
 \sup_{t\in[0,T']} d~\left|f_2'(t,W_2(t)+id)\right| \leq c'd^{1-\beta} \;\;\;\;\; \forall~d\leq \epsilon^p.
\]

\item For some $c<\infty$, the tip structure modulus of $\gamma_1$ on $[0,T']$ satisfies
\[
 \eta_{\textnormal{tip}}(\bar\delta) \leq c\, \bar\delta^{\,r},
 \qquad \bar\delta:=C_0\max\bigl\{\epsilon^{p(1-\beta)},\,\epsilon^{1-Lp}\bigr\},
\]
where $C_0<\infty$ depends only on $T$, $\beta$, $L$, $p$ and $c'$.

\end{enumerate}
For $\epsilon\le\epsilon_0(T,\beta,L,p,c')$, there is a constant $c'' = c''(T,\beta,L, r, p, c, c')<\infty$ such that
\[
 \sup_{t\in[0,T']} |\gamma_1(t) - \gamma_2(t)| \leq c'' \max\{\epsilon^{p(1-\beta)r}, \epsilon^{(1-L p)r} \}.
\]

The analogous statement holds for $\ID$-Loewner pairs.
\end{lemma}

\begin{remark}
The comparison scale follows from the perturbation estimate \cite{JVRW2014} and Cauchy's estimate, as in the proof of \cite[Lemma~3.4]{V}. The constants depend on the upper horizon $T$, so the result applies pathwise on any shorter, possibly random interval $[0,T']$. Proposition~\ref{derest1} supplies the derivative bound for SLE.
\end{remark}

\subsection{Main Estimate for the Tip Structure Modulus}
We now establish the tip structure modulus estimate, directly for $\kappa$-admissible curves.

We need four preparatory results: a power bound for unforced crossings observed at stopping times, which is the quantitative consequence of the KS Condition; three elementary facts of planar topology; the observation that a crosscut from the tip cuts off only avoidable regions; and an extraction lemma organizing the topological part of the argument. By the conformal invariance in \cite[Proposition~2.6]{KS}, the image curve in $\ID$ satisfies the same condition with a modulus depending only on the original KS modulus. From now on $M_{\mathrm{KS}}$ denotes this disk modulus.

\begin{lemma}[Power bound for unforced crossings; condition G3]\label{lem:G3}
Let the random curve $\gamma$ in $(\ID,-1,1)$ satisfy the KS Condition (Definition \ref{cond:KS}) with modulus bound $M_{\mathrm{KS}}$, and let $(\mathcal F_t)$ be its exploration filtration. Write $A^D_\tau$ for the corresponding avoidable part of the annulus in the slit domain $D_\tau$. There exist constants $K' < \infty$ and $\Delta_{\mathrm{KS}} > 0$, depending only on $M_{\mathrm{KS}}$, such that for every stopping time $\tau$, every $z_0$, and all $0 < r < R$,
\[
\IP\Bigl( \gamma[\tau, 1] \text{ makes a crossing of } A(z_0, r, R) \text{ contained in } \overline{A^{D}_\tau} \; \Bigm| \; \mathcal F_\tau \Bigr)
\ \le\ K' \Bigl( \frac{r}{R} \Bigr)^{\Delta_{\mathrm{KS}}} .
\]
One may take $K' = 2$ and $\Delta_{\mathrm{KS}} = \log 2 / \log(2 M_{\mathrm{KS}})$.
\end{lemma}

\begin{proof}
This is (G2)$\Rightarrow$(G3) in the proof of \cite[Proposition~2.6]{KS}, applied conditionally on $\mathcal F_\tau$ with $C=M_{\mathrm{KS}}$ in Definition~\ref{cond:KS} and annulus ratio $2M_{\mathrm{KS}}$ to satisfy its strict inequality.
\end{proof}

\begin{lemma}[Three elementary planar facts]\label{lem:enclosure}
Let $\Omega \subset \IC$ be a bounded simply-connected domain.
\begin{enumerate}
\item[(i)] \emph{(Enclosure.)} Let $C$ be a crosscut of $\Omega$ of diameter at most $\delta$, let $V$ be a connected component of $\Omega \setminus C$, and let $\beta_V := \partial V\cap\partial\Omega$ be the boundary arc of $V$. If $\diam \beta_V \le D_0$, then $\diam(V) \le \delta + D_0$.
\item[(ii)] \emph{(Slit attachment.)} Let $L \subset \IC$ be a compact connected set and $W$ a simple arc meeting $L$ only at one of its endpoints. If $\diam W > \diam L$, then every bounded component of $\IC \setminus (L \cup W)$ is contained in a bounded component of $\IC \setminus L$; in particular, its diameter is at most $\diam L$.
\item[(iii)] \emph{(Separation.)} Let $S$ be a crosscut of $\Omega$ with endpoints $z_1, x_1$, let $Y$ be a component of $\Omega \setminus S$, and let $p \in Y$ and $q \in \Omega \setminus (Y \cup S)$. Let $C'$ be a compact connected set containing $z_1$, and $J \subset \IC \setminus \Omega$ a compact connected set containing $x_1$ and meeting $C'$. If $p$ and $q$ lie in the unbounded component of $\IC \setminus (J \cup C' \cup \overline S)$, then they lie in different components of $\Omega \setminus C'$.
\end{enumerate}
\end{lemma}

\begin{proof}
For (i), $\partial V\subseteq\overline C\cup\beta_V$, as $V$ is a component of $\Omega\setminus C$, and $\beta_V$ contains the endpoints of $C$. As $V$ is bounded, $\diam V=\diam\partial V\le\delta+D_0$.
For (ii), the connected set $W\setminus L$ lies in the unbounded component $U$ of $\IC\setminus L$, as a bounded one lies in the convex hull of $L$ and $\diam W>\diam L$. On the Riemann sphere, the arc $W$ and the closed set $\widehat\IC\setminus(U\cup\{\infty\})$ meet in one point, and neither separates two points of $U\setminus W$; by Janiszewski's theorem \cite[Chapter~V]{Newman}, $U\setminus W$ is connected. Hence the bounded components of $\IC\setminus(L\cup W)$ are those of $\IC\setminus L$.
For (iii), work on the Riemann sphere and suppose that $p$ and $q$ lie in one component of $\Omega \setminus C'$. Then neither of the closed sets $A := J \cup C' \cup \overline S$ and $B := (\widehat\IC \setminus \Omega) \cup C'$ separates $p$ from $q$. Since $\overline S \setminus \Omega = \{z_1, x_1\} \subset C' \cup J$, we have $A \cap B = J \cup C'$, which is connected. By Janiszewski's theorem, $A \cup B$ does not separate $p$ from $q$. But $\widehat\IC \setminus (A \cup B) \subseteq \Omega \setminus S$, in which $p$ and $q$ lie in different components.
\end{proof}

\begin{lemma}[Crosscuts from the tip]\label{lem:kstop}
Let $\gamma$ be a simple curve in $\ID$, let $\tau$ be a time, $D_\tau := \ID\setminus\gamma[0,\tau]$, and let $z_1 := \gamma(\tau)$ be the tip. Let $C'$ be the closure of a crosscut of $D_\tau$ joining the tip prime end to another prime end, and let $U$ be the component of $D_\tau\setminus C'$ not containing $1$. Then no connected subset of $U$ disconnects $z_1$ from $1$. In particular, if $C'\subseteq\overline B(z_1,r)$ with $r<1$, every component of $A(z_1,r,R)\cap D_\tau$ that meets $U$ lies in $U$ and in the avoidable set $A^D_\tau$ of Definition \ref{def:unforcros}.
\end{lemma}

\begin{proof}
The tip prime end is an endpoint of the crosscut, so it lies on the boundary of both components of $D_\tau\setminus C'$. The component containing $1$ is connected and misses $U$, so it joins every neighbourhood of $z_1$ to every neighbourhood of $1$ by a path avoiding $U$. A component of $A(z_1,r,R)\cap D_\tau$ is connected and misses $C'\subseteq\overline B(z_1,r)$, so it lies in $U$ as soon as it meets $U$. Finally, $\partial D_\tau$ is connected and contains $z_1$ and $\partial\ID$, so it meets $\partial B(z_1,r)$, and $A^D_\tau$ is given by its defining formula.
\end{proof}

\begin{lemma}[Extraction]\label{lem:extract}
Let $\tau$, $z_1$, $C'$ and $U$ be as in Lemma \ref{lem:kstop}, with $C'\subseteq\overline B(z_1,r)$, and let $r<R$. If $\gamma(t^*)\in U$ for some $t^*>\tau$ with $|\gamma(t^*)-z_1|\ge R$, then $\gamma[\tau,1]$ makes a crossing of $A(z_1,r,R)$ contained in $\overline{A^D_\tau}$; that is, the crossing is unforced as observed at time $\tau$.
\end{lemma}

\begin{proof}
Let $t_{\mathrm{in}}$ be the last time in $[\tau,t^*]$ with $|\gamma(t_{\mathrm{in}})-z_1|\le r$, and $t_{\mathrm{out}}$ the first time after $t_{\mathrm{in}}$ with $|\gamma(t_{\mathrm{out}})-z_1|\ge R$. The arc $\gamma(t_{\mathrm{in}},t^*]$ lies in $D_\tau\setminus\overline B(z_1,r)\subseteq D_\tau\setminus C'$ and contains $\gamma(t^*)\in U$, so it lies in $U$. Hence $\gamma[t_{\mathrm{in}},t_{\mathrm{out}}]$ is a crossing of $A(z_1,r,R)$ whose interior lies in $A(z_1,r,R)\cap U$, hence in $A^D_\tau$ by Lemma \ref{lem:kstop}; so the crossing lies in $\overline{A^D_\tau}$.
\end{proof}


Throughout this subsection the tip structure modulus is the modulus $\eta^{\ID}_{\mathrm{tip}}$ in $\ID$ defined after Definition~\ref{def:tip}.

 \begin{proposition}\label{tip} Let $\gamma$ be a random curve in $(\ID;-1,1)$ satisfying the KS Condition with modulus bound $M_{\mathrm{KS}}$, and assume Hypothesis H1 holds for $\gamma$ with constants $(K_\nu, \Delta_\nu)$ (by Proposition \ref{KStoAB}, this follows from the KS Condition, with all constants depending only on $M_{\mathrm{KS}}$). Let $\rho>0$ and let $\eta_{\mathrm{tip}}(\delta)$ be the tip structure modulus of $\gamma$ stopped when first reaching distance $\rho$ from 1. There exist positive constants $\epsilon=\epsilon(M_{\mathrm{KS}})$, $\alpha=\alpha(M_{\mathrm{KS}})$ and $\eta_0=\eta_0(M_{\mathrm{KS}})$ such that if $\delta\leq\eta^{1+\epsilon}$, $\eta<\eta_0$, and $\eta \le \rho/32$ then
 \[
  \IP(\eta_{\mathrm{tip}}(\delta)>\eta)\leq 2 \eta^{\alpha}.
 \]
One may take $\alpha = 1/2$ and $\epsilon = 9/\Delta_{\mathrm{KS}}$, with $\Delta_{\mathrm{KS}}$ from Lemma \ref{lem:G3}.
 \end{proposition}

As an immediate application of the proposition, we obtain the following:
 \begin{corollary}\label{cor:tip} Suppose the random family of curves $\{\gamma^n\}$ is $\kappa$-admissible for some $\kappa$. Let $\eta_{\mathrm{tip}}^{(n)}(\delta)$ be the tip structure modulus for $\gamma^n_{\ID}$ stopped when first reaching distance $\rho>0$ from 1. Let $s \in (0,1)$ and set
\[
p_\star := \frac{1}{1+\epsilon} \in (0,1),
\]
with $\epsilon$ from Proposition \ref{tip}. Then, with $\alpha = 1/2$ from Proposition \ref{tip}, there exists $n_2 = n_2(M_{\mathrm{KS}}, \rho, s)<\infty$  so that for $n\geq n_2$ we have
 \[
  \IP_n\left(\eta_{\mathrm{tip}}^{(n)}(n^{-s})>n^{-p_\star s}\right)\leq 2n^{-\alpha p_\star s}.
 \]
In particular, $p_\star$ and $\alpha$ depend only on the KS modulus bound, and not on $\rho$ or $n$.
 \end{corollary}

\subsection{Results from Aizenman and Burchard}
   We use the following multiple-crossing hypothesis of Aizenman and Burchard \cite{AB}.

\noindent\textbf{Hypothesis H1.} A collection of probability measures $(\IP_h)_{h>0}$ indexed by the mesh $h$, on the space of curves $\mathcal{S}(\mathbb{C})$, satisfies a \emph{power-law bound on multiple crossings} if there exists $C>0, \; K_\nu>0$ and sequence $\Delta_\nu\ge0$ with $\Delta_\nu\rightarrow\infty$ as $\nu\rightarrow\infty$ such that for all radii $0<r\leq C^{-1} R$ and for any annulus $A = A(z_0,r,R)$ and for each $\IP_h\in(\IP_h)_{h>0}$
\[
 \IP_h\left(\gamma \text{ makes } \nu \text{ crossings of } A\right) \leq K_\nu \left(\frac{r}{R}\right)^{\Delta_\nu}.
\]

Denote
\[
 M(\gamma, l) = \min \left\{\! N\in\IN \; \bigg|\begin{aligned} \; \exists \text{ partition } 0=t_0<t_1<\cdots<t_N=1 \\ \text{ s.t. } \diam\left(\gamma[t_{k-1},t_k]\right)\leq l \text{ for } 1\leq k\leq N\end{aligned} \right\}.
\]

\begin{lemma}[Piece count at one scale]\label{lem:count}
Let $\Omega_0 \subset \IC$ be a bounded region of diameter at most $R_0$, let $\beta > 1$, $\nu \in \IN$, and $\lambda \in (0,1]$. Let $\Lambda_\lambda$ be a $(\lambda/2)^\beta$-net of $\Omega_0$, so that $\#\Lambda_\lambda \le C R_0^2 (\lambda/2)^{-2\beta}$. If a curve $\gamma \subset \Omega_0$ makes at most $\nu - 1$ crossings of the annulus $A\bigl(z, (\lambda/2)^\beta, \lambda/2\bigr)$ for every $z \in \Lambda_\lambda$, then
\[
M(\gamma, \lambda) \ \le\ 1 + C R_0^2\, \nu\, (\lambda/2)^{-2\beta}\ \le\ C_1(R_0)\, \nu\, \lambda^{-2\beta}.
\]
Consequently, if the law of $\gamma$ satisfies Hypothesis H1 with constant $C$, then for every $\lambda \le \lambda_0$, where $\lambda_0 = \lambda_0(C,\beta) \in (0,1]$ is chosen so that $(\lambda_0/2)^{\beta-1} \le C^{-1}$,
\begin{equation}\label{eq:countprob}
\begin{aligned}
\IP\Bigl( M(\gamma, \lambda) > C_1(R_0)\, \nu\, \lambda^{-2\beta} \Bigr)
&\ \le\ C R_0^2 (\lambda/2)^{-2\beta} \cdot K_\nu\, (\lambda/2)^{(\beta - 1)\Delta_\nu} \\
&\ \le\ C_2(R_0, K_\nu)\, \lambda^{(\beta-1)\Delta_\nu - 2\beta} .
\end{aligned}
\end{equation}
\end{lemma}

\begin{proof}
Define times $u_0 := 0$ and, recursively, choose $z_i \in \Lambda_\lambda$ with $|\gamma(u_i) - z_i| \le (\lambda/2)^\beta$ and let $u_{i+1}$ be the first exit of $\gamma$ after $u_i$ from $B(z_i, \lambda/2)$. Each completed piece $\gamma[u_i, u_{i+1}]$ has diameter at most $\lambda$, so this construction exhibits an admissible partition and the number of pieces (including the final, possibly incomplete, one) is an upper bound for the minimum $M(\gamma, \lambda)$. Each completed piece starts in $\overline{B}(z_i, (\lambda/2)^\beta)$ and ends on $\partial B(z_i, \lambda/2)$, hence contains a crossing of $A(z_i, (\lambda/2)^\beta, \lambda/2)$; distinct pieces with the same center are disjoint sub-arcs, hence distinct crossings. Therefore the number of completed pieces with center $z$ is at most $\nu - 1$ for every $z \in \Lambda_\lambda$, and the total number of pieces, including the final incomplete one, is at most $(\nu-1)\, \#\Lambda_\lambda + 1 \le \nu\, \#\Lambda_\lambda$. The probability bound follows from the union bound over $\Lambda_\lambda$ and Hypothesis H1 applied to each annulus, whose radius ratio is $(\lambda/2)^{\beta - 1}$.
\end{proof}

\begin{proposition}[{\cite[Proposition 3.6]{KS}}]\label{KStoAB}
 If the random family of curves $\{\gamma_n\}$ satisfies the KS Condition with modulus bound $M_{\mathrm{KS}}$, then it (transformed to $\ID$) satisfies  Hypothesis H1, with the constants $C$, $K_\nu$, $\Delta_\nu$ depending only on $M_{\mathrm{KS}}$.
\end{proposition}
The reduction of the event $E(\delta,\eta)$ defined below to unforced crossings follows the proof of \cite[Proposition 3.12]{KS}.

\subsection{Proof of Main Estimate for Tip Structure Modulus}

 \begin{proof}[Proof of Proposition \ref{tip}]
  Let $\tau$ be the first time at which $\gamma$ reaches distance $\rho$ from $1$. Take $0<\delta<\eta/32$, which follows from $\delta\le\eta^{1+\epsilon}$ for $\eta\le\eta_0$.

Define $E(\delta,\eta)$ as the event that there exists $(s,t)\in[0,\tau]^2$ with $s<t$ such that
\begin{quote}
\begin{enumerate}
  \item $\diam\!\left(\gamma[s,t]\right)\geq \eta$, and
  \item there exists a crosscut $C$ with $\diam(C)\leq \delta$ that separates $\gamma(s,t]$ from $B(1,\rho)\cap\ID$ in $\ID\setminus\gamma(0,s]$.
\end{enumerate}
\end{quote}
To see that $\eta_{\mathrm{tip}}(\delta)>\eta$ implies $E(\delta,\eta)$, take $t\le\tau$ and $\cC\in S_{t,\delta}$ with $\diam\gamma_\cC>\eta$, and set $s:=t-s_\cC$. Then $\cC$ is a crosscut of $D_s:=\ID\setminus\gamma[0,s]$ missing $\gamma(s,t]$; let $V$ be the component of $D_s\setminus\cC$ containing $\gamma(s,t]$. As $\widehat\IC\setminus V$ is connected and meets the arc $\gamma[s,t]$ only at $\gamma(s)$, Janiszewski's theorem shows that $V\setminus\gamma(s,t]$ is connected. This set lies in $D_t\setminus\cC$ and accumulates at the tip, which $\cC$ separates from $B(1,\rho)\cap\ID$; so $V$ misses $B(1,\rho)\cap\ID$, and $s$, $t$, $\cC$ witness $E(\delta,\eta)$.

\textsc{Step 1.}
\begin{figure}[htbp]
  \centering
  \begin{tikzpicture}[scale=1.1,
    every node/.style={font=\footnotesize},
    point/.style={circle,fill=black,inner sep=1.5pt},
    cover/.style={draw=black,dotted,thick},
    lbl/.style={fill=white,inner sep=1pt}]
  \foreach \x/\y in {0/0,1.2/0.4,2/1.0,3/2.1,4.2/2.0,5/2.4}
    \draw[cover] (\x,\y) circle (0.75);
  \draw[thick,smooth] plot[tension=1.5] coordinates
    {(0,0) (1.2,0.4) (2,1.0) (2.6,1.6) (3,2.1) (3.7,2.2) (4.2,2.0) (5,2.4)};
  \node[point] (p0) at (0,0) {};
  \node[point] (p1) at (1.2,0.4) {};
  \node[point] (p2) at (2,1.0) {};
  \node[point] (p3) at (3,2.1) {};
  \node[point] (p4) at (4.2,2.0) {};
  \node[point] (p5) at (5,2.4) {};
  \node[lbl,below left] at (p0) {$\gamma(\sigma_0)$};
  \node[lbl,below] at (p1) {$\gamma(\sigma_1)$};
  \node[lbl,below right] at (p2) {$\gamma(\sigma_2)$};
  \node[lbl,above left] at (p3) {$\gamma(\sigma_3)$};
  \node[lbl,below] at (p4) {$\gamma(\sigma_4)$};
  \node[lbl,right] at (p5) {$\gamma(\sigma_5)$};
  \node[lbl] at (0.5,0.55) {$J_1$};
  \node[lbl] at (1.45,1.05) {$J_2$};
  \node[lbl] at (2.25,1.65) {$J_3$};
  \node[lbl] at (3.6,1.7) {$J_4$};
  \node[lbl] at (4.55,2.65) {$J_5$};
  \end{tikzpicture}
  \caption{Illustration of Step 1: $N \le \inf_{\varepsilon>0} M(\gamma,\eta/8-\varepsilon) \le M(\gamma,\eta/16)$,
  to which the tail bound of Lemma~\ref{lem:count} applies. Each subarc of diameter less than $\eta/8$ contains at most one division point $\gamma(\sigma_k)$.}
  \label{fig:covering-gamma}
\end{figure}

Divide the curve $\gamma$ into $N$ arcs of diameter less than or equal to $\eta/8$.

 Let $\sigma_k$ be defined by $\sigma_k = 0$ for $k \leq 0$ and then recursively
 \[\sigma_k = \sup\{t\in[\sigma_{k-1}, 1] \; : \; \diam(\gamma[\sigma_{k-1},t])<\eta/8\}.\]
 Let $J_k = \gamma[\sigma_{k-1},\sigma_{k}]$ and $J_0 = \partial\ID$. As in Figure \ref{fig:covering-gamma}, $N \le M(\gamma, \eta/16)$. Fix $\beta := 2$ and let $\nu$ be the smallest index with $\Delta_\nu \ge 5$ (which exists since $\Delta_\nu \to \infty$; it depends only on $M_{\mathrm{KS}}$). Set
\[
m := \bigl\lceil C_1\, \nu\, (\eta/16)^{-4} \bigr\rceil \ \le\ C(\nu)\, \eta^{-4} .
\]
By Lemma \ref{lem:count},
\begin{equation}\label{eq:Ntail}
\IP( N > m ) \ \le\ C_2\, \eta^{(\beta-1)\Delta_\nu - 2\beta} \ =\ C_2\, \eta^{\Delta_\nu - 4} \ \le\ C_2\, \eta .
\end{equation}

\textsc{Step 2.} Reduction to observed unforced crossings, illustrated in Figure~\ref{fig:stepII}. For $0 \le i < k$ and $t \in [\sigma_{k-1},\sigma_k]$, let $f_i(t)$ be the infimum of the diameters of crosscuts of $D_t$ joining the tip prime end of $D_t$ to a prime end supported on $J_i$, and let $\tau_{i,k}$ be the first such $t$ with $f_i(t)\le2\delta$, with $\inf\emptyset := 1$. For $t<u$, joining such a crosscut of $D_u$ to the tip of $D_t$ by the arc $\gamma[t,u]$ gives $f_i(t)\le f_i(u)+\diam\gamma[t,u]$. Hence the infimum defining $\tau_{i,k}$ is attained, so that $D_{\tau_{i,k}}$ has such a crosscut of diameter at most $3\delta$ if $\tau_{i,k}<1$, and $\{\tau_{i,k}\le t\}$ depends only on $\gamma[0,t]$, as $J_i\subseteq\gamma[0,\sigma_i]$. Thus each $\tau_{i,k}$ is a stopping time.

Fix $j < k$ and work on $\{\tau_{j,k} < 1\}$. Set $z_1 := \gamma(\tau_{j,k})$, and let $C'$ be the closure of a crosscut of $D_{\tau_{j,k}}$ of diameter at most $3\delta$ joining the tip prime end to a prime end supported at a point $z_2 \in J_j$; thus $C' \subseteq \overline B(z_1, 3\delta)$. Let $U^{j,k}_{\mathrm{far}}$ be the component of $D_{\tau_{j,k}} \setminus C'$ not containing $1$, and let $E_{j,k}$ be the event that $\tau_{j,k} < 1$ and $\gamma[\tau_{j,k}, 1]$ makes a crossing of $A(z_1, 3\delta, \eta/8)$ contained in $\overline{A^D_{\tau_{j,k}}}$.

\emph{Claim: $E(\delta,\eta) \cap \{N \le m\} \subseteq \bigcup_{0 \le j < k \le m} E_{j,k}$.} Let $C, s, t$ witness $E(\delta,\eta)$, and let $p_1, p_2$ be the endpoints of $C$ on $\partial D_s$, supported at points $x_1$ and $x_2$, each either on $\gamma[0,s]$ or on $\partial\ID$. Let $V$ be the component of $D_s \setminus C$ containing $\gamma(s,t]$, and $\beta_V$ its boundary arc.

\emph{(a) Two preliminary observations.} First, $\gamma(s,t] \cap \overline{C} = \emptyset$: a point of the excursion on $C$ would lie in neither component of $D_s \setminus C$, contradicting the separation of $\gamma(s,t]$ from $B(1,\rho)$ by $C$; hence $\gamma(s,t] \subset V$ exactly. Second, the tip prime end of $D_s$ (at $\gamma(s)$) belongs to $\beta_V$: otherwise $\gamma(s + \varepsilon) \in D_s \setminus \overline{V}$ for small $\varepsilon$, again contradicting the separation. In particular $\diam \beta_V \ge \eta - \delta$: otherwise Lemma \ref{lem:enclosure}(i) would give $\diam V \le \delta + \diam \beta_V < \eta \le \diam\gamma[s,t]$, contradicting $\gamma(s,t]\subset V$. Note also that the ``target bulk'' $B(1,\rho)\cap\ID$ lies in $D_s$ (the stopped curve avoids $B(1,\rho)$) and, by the separation, in $D_s \setminus V$; it is connected, of diameter at least $\rho$.

\emph{(b) Both feet on $\partial\ID$ is impossible.} Suppose $x_1, x_2 \in \partial\ID$. Then $|x_1 - x_2| \le \delta$, and $\partial\ID$ splits into a short arc $\iota$, of diameter at most $2\delta$, and a long arc. Since the tip end lies on $\beta_V$, the arc $\beta_V$ is the one through the starting point $-1$. Two cases arise.
\begin{itemize}
\item $-1$ on the $\iota$-side. Then the accessible boundary of $V$ is contained in $L_0 \cup \gamma[0,s]$, where $L_0 := \overline{C} \cup \iota$ is connected of diameter at most $3\delta$. The curve $\gamma[0,s]$ meets $L_0$ only at its endpoint $-1 \in \iota$, and $\diam \gamma[0,s] \ge \eta - 3\delta > 3\delta \ge \diam L_0$ by (a); so by Lemma \ref{lem:enclosure}(ii) applied with $L := L_0$ and $W := \gamma[0,s]$, every bounded component of $\IC \setminus (L_0 \cup \gamma[0,s])$ has diameter at most $\diam L_0 \le 3\delta$. As $V$ is such a component, $\diam V \le 3\delta < \eta$, contradicting $\gamma(s,t] \subset V$.
\item $-1$ on the long-arc side. Then $D_s\setminus V$ has boundary arc $\iota$, so $\diam(D_s \setminus V) \le 3\delta < \rho$ and the target bulk does not fit.
\end{itemize}
So at least one foot is on the curve; relabel so that $x_2 = \gamma(t_2)$ with $t_2 \le s$ maximal among curve feet, and let $x$ denote the prime end of $D_{t_2}$ at $x_1$ determined by the approach of $C$.

\emph{(c) The feet lie on distinct pieces, and $\tau_{j,k} \le t_2$.} Let $j$ be the index of the piece supporting $x_1$ and $k$ the arc index of $t_2$; on the event considered $k \le N \le m$. Suppose that both feet are supported on $J_k$. Let $I\subset J_k$ be the sub-arc joining $x_1$ and $x_2$, and put $L:=I\cup\overline C$ (if the two supports coincide, take $I=\{x_1\}$). Since $\gamma$ is simple and $C$ is a crosscut of $D_s$, $L$ is a Jordan curve, with the coincident-support case understood in the evident limiting sense, and
\[
\diam L\le \diam J_k+\diam C\le \eta/8+\delta<\eta.
\]
The excursion $\gamma(s,t]$ and the target bulk lie on opposite sides of $L$. Indeed, the two components of $D_s\setminus C$, which contain them, miss $L$, so each lies on one side of $L$; and both sides meet $D_s\setminus C$, since a small disc about a point of $C$ lies in $D_s$, misses $L\setminus C$ and meets both sides of $L$. Hence one of them lies in the bounded component $U$ of $\IC\setminus L$, and $\diam U\le\diam L$. If it is the excursion, this contradicts $\diam\gamma[s,t]\ge\eta$; if it is the target bulk, it contradicts $\diam(\text{target bulk})\ge\rho>\eta/8+\delta$. Thus the feet cannot both lie on $J_k$, that is, $x_1\notin J_k$. Hence $j < k$, and this is the pair $(j,k)$ for which we exhibit $E_{j,k}$. Since $C$ is a crosscut of $D_{t_2}$ of diameter at most $\delta$ joining the tip $\gamma(t_2)$ to $x_1 \in J_j$, and $t_2 \in [\sigma_{k-1},\sigma_k]$, we get $\tau_{j,k} \le t_2 < 1$.

\emph{(d) The excursion lies behind $C'$ outside a small ball.} Put $S := \gamma[\tau_{j,k},t_2]\cup\overline C$. Since
$\tau_{j,k}$ and $t_2$ both lie in $[\sigma_{k-1},\sigma_k]$, we have $\gamma[\tau_{j,k},t_2]\subseteq J_k$, which misses $x_1$ by (c). Hence $S$ is a crosscut of $D_{\tau_{j,k}}$ joining the tip prime end to the prime end at $x_1\in J_j$ (on $\partial\ID$ if $j=0$), and $\diam S\le\eta/8+\delta$.
The curve beyond $t_2$ is irrelevant: $\gamma(t_2,t]$ is connected, contained in $D_{t_2}$, disjoint from $C$ (which lies in $D_s$, hence off the curve), and contains $\gamma(s,t] \subseteq V$. Let $W$ be the component of $D_{t_2}\setminus C$ containing it. As in the reduction to $E(\delta,\eta)$, $W\setminus\gamma(t_2,s]$ is connected. This set lies in $D_s\setminus C$ and contains $\gamma(s,t]$, so it is $V$; hence $W$ misses the target bulk, which lies in $D_s\setminus V$ and off the curve. As $D_{\tau_{j,k}}\setminus S=D_{t_2}\setminus\overline C$, $W$ is the component $Y$ of $D_{\tau_{j,k}} \setminus S$ not containing $1$. Consequently $V\subseteq Y$, and every path from $V$ to $1$ inside $D_{\tau_{j,k}}$ meets $S$.

Write $V' := V \cap U^{j,k}_{\mathrm{far}}$ for the part of the pocket cut off from $1$ by $C'$. Let $\widehat J := J_j$ if $j \ge 1$, and if $j = 0$ let $\widehat J$ be the shorter arc of $\partial\ID$ between $x_1$ and $z_2$. In either case $\widehat J$ is a compact connected subset of $\IC \setminus D_{\tau_{j,k}}$ containing $x_1$ and $z_2$. Measuring from $z_1$: every point of $S$ lies within $\diam S \le \eta/8 + \delta$, as $z_1 \in S$; every point of $C'$ within $3\delta$; and every point of $\widehat J$ within $|z_1 - z_2| + \diam \widehat J \le \eta/8 + 7\delta$, using $\diam J_j \le \eta/8$ for $j \ge 1$ and, for $j = 0$, $\diam\widehat J = |x_1 - z_2| \le |x_1 - \gamma(t_2)| + |\gamma(t_2) - z_1| + |z_1 - z_2| \le \eta/8 + 4\delta$. Hence $L := \widehat J \cup \overline S \cup C' \subseteq \overline B(z_1, r_\delta)$, where $r_\delta := \eta/8 + 7\delta < \eta/2$ as $\delta < \eta/32$. Fix $q \in B(1,\rho/2) \cap \ID$. It lies in $D_{\tau_{j,k}} \setminus (Y \cup S)$ at distance at least $\rho/2 > r_\delta$ from $z_1$, and $B(1,\rho/2) \cap \ID$ is a connected set joining it to $1$ and missing $C'$. If $p \in V \setminus \overline B(z_1, r_\delta)$, then $p$ and $q$ lie in the unbounded component of $\IC \setminus L$, and Lemma \ref{lem:enclosure}(iii), applied in $D_{\tau_{j,k}}$ to the crosscut $S$, the component $Y \ni p$, and $J := \widehat J$, shows that $p$ and $q$ lie in different components of $D_{\tau_{j,k}} \setminus C'$, that is, $p \in V'$. Therefore $V \setminus V' \subseteq \overline B(z_1, r_\delta) \subseteq B(z_1, \eta/2)$.

\begin{figure}[htbp]
\centering
\begin{tikzpicture}[scale=0.62,
  curve/.style={line width=0.8pt},
  exc/.style={line width=1.7pt},
  piece/.style={line width=2.4pt},
  cc/.style={line width=0.9pt},
  ann/.style={densely dotted, line width=0.7pt},
  ball/.style={dashed, line width=0.7pt},
  lbl/.style={font=\footnotesize}]

\begin{scope}
\draw[line width=0.7pt] (0,0) circle (4);
\fill (4,0) circle (2.2pt) node[lbl,right] {$1$};
\fill (-3.9,-0.9) circle (2.2pt) node[lbl,left] {$-1$};

\draw[curve] (-3.9,-0.9) .. controls (-3.6,-1.05) and (-3.3,-1.18) .. (-3.05,-1.26);
\draw[piece] (-3.05,-1.26) .. controls (-2.95,-1.29) .. (-2.78,-1.32) .. controls (-2.66,-1.36) .. (-2.50,-1.40);
\draw[curve] (-2.50,-1.40) .. controls (-1.2,-2.3) and (0.6,-2.5) .. (2.2,-1.6)
  .. controls (3.1,-0.6) and (3.1,0.7) .. (2.6,1.6)
  .. controls (1.6,2.7) and (-0.2,3.0) .. (-1.7,2.4)
  .. controls (-2.7,1.9) and (-3.1,1.0) .. (-3.10,0.15)
  .. controls (-3.12,-0.4) and (-3.00,-0.75) .. (-2.92,-1.05);
\draw[piece] (-2.92,-1.05) .. controls (-2.80,-1.11) .. (-2.64,-1.20);
\draw[exc] (-2.64,-1.20) .. controls (-2.3,-0.7) and (-1.9,-0.1) .. (-1.2,0.4)
  .. controls (-0.6,0.85) and (0.5,0.95) .. (1.3,0.25);

\draw[cc] (-2.92,-1.05) -- (-2.92,-1.29);
\draw[cc] (-2.64,-1.20) -- (-2.62,-1.35);
\fill (-2.92,-1.05) circle (1.7pt);

\draw[ball] (-2.92,-1.05) circle (0.74);

\node[lbl] at (0.2,-1.0) {$V$};
\node[lbl,anchor=west] at (-1.35,1.35) {$\gamma(s,t]$};
\draw[->,line width=0.5pt] (-1.45,1.30) -- (-1.15,0.55);
\node[lbl,anchor=west] at (-2.0,-2.35) {$B\bigl(z_1,\tfrac\eta8+7\delta\bigr)$};
\draw[->,line width=0.5pt] (-2.1,-2.28) -- (-2.70,-1.74);
\node[lbl] at (-3.6,0.9) {$\gamma$};
\node[lbl] at (0,-4.6) {(a)};
\end{scope}

\begin{scope}[xshift=10.4cm,scale=1.12]
\draw[curve] (-0.9,2.6) .. controls (-0.5,1.7) and (-0.15,0.7) .. (0,0);
\draw[piece] (0,0) .. controls (0.5,-0.25) .. (1.12,-0.60);
\draw[piece] (-1.5,-0.60) .. controls (-0.9,-0.80) .. (0.56,-1.08) .. controls (1.1,-1.24) .. (2.3,-1.55);
\draw[curve] (-2.6,-0.25) .. controls (-2.2,-0.38) .. (-1.5,-0.60);
\draw[exc] (1.12,-0.60) .. controls (1.7,0.2) and (2.1,0.9) .. (2.7,1.7);

\draw[cc] (0,0) -- (0,-0.96);
\draw[cc] (1.12,-0.60) -- (1.20,-1.20);
\fill (0,0) circle (2.2pt);
\fill (0,-0.96) circle (2.2pt);
\fill (1.12,-0.60) circle (2.2pt);

\draw[ann] (0,0) circle (0.96);
\draw[ann] (0,0) circle (2.0);

\node[lbl,anchor=east,fill=white,inner sep=1pt] at (-0.35,0.55) {$z_1=\gamma(\tau_{j,k})$};
\draw[->,line width=0.4pt] (-0.30,0.45) -- (-0.05,0.08);
\node[lbl,anchor=east,fill=white,inner sep=1pt] at (-0.15,-1.20) {$z_2$};
\node[lbl,anchor=east,fill=white,inner sep=1pt] at (-0.12,-0.46) {$C'$};
\node[lbl,anchor=west,fill=white,inner sep=1pt] at (2.4,-0.75) {$C$};
\draw[->,line width=0.4pt] (2.28,-0.78) -- (1.25,-0.91);
\node[lbl,anchor=west,fill=white,inner sep=1pt] at (1.62,-0.34) {$\gamma(t_2)$};
\node[lbl,fill=white,inner sep=1pt] at (-1.55,-1.18) {$J_j$};
\node[lbl,anchor=north west,fill=white,inner sep=1pt] at (0.95,-1.65) {$x_1$};
\node[lbl,fill=white,inner sep=1pt] at (0.60,0.12) {$J_k$};
\node[lbl,anchor=west,fill=white,inner sep=1pt] at (1.55,2.65) {$A(z_1,3\delta,\eta/8)$};
\draw[->,line width=0.5pt] (1.60,2.45) -- (1.20,1.35);
\node[lbl] at (0.6,-4.107) {(b)};
\end{scope}

\end{tikzpicture}
\caption{Step 2, with $\delta$ exaggerated. (a) The excursion $\gamma(s,t]$ lies in the component $V$ cut off from $1$ by $C$. (b) At $\tau_{j,k}$ the tip $z_1$ first comes within crosscut diameter $2\delta$ of $J_j$. The crosscut $C'$ has diameter at most $3\delta$ and cuts off the part $V'\subseteq V$. The proof gives $V\setminus V'\subseteq\overline B(z_1,\eta/8+7\delta)$, whereas the excursion reaches distance $\eta/2$ from $z_1$. Lemma~\ref{lem:extract} then gives an unforced crossing of $A(z_1,3\delta,\eta/8)$.}
\label{fig:stepII}
\end{figure}

\emph{Extraction at $\tau_{j,k}$.} Since $\diam\gamma[s,t] \ge \eta$, some point of $\gamma[s,t]$ lies at distance $\ge \eta/2$ from $z_1$, and by continuity there is $t^* \in (s,t]$ with $|\gamma(t^*) - z_1| > r_\delta$. Then $\gamma(t^*) \in V \setminus \overline B(z_1, r_\delta) \subseteq V' \subseteq U^{j,k}_{\mathrm{far}}$ by (d), and Lemma \ref{lem:extract} at $\tau := \tau_{j,k}$, with $r := 3\delta$ and $R := \eta/8$, yields a crossing of $A(z_1, 3\delta, \eta/8)$ contained in $\overline{A^D_{\tau_{j,k}}}$, which is $E_{j,k}$. This proves the Claim.

\textsc{Step 3.} Estimate $\IP\left(E(\delta,\eta)\right)$. As the centre $z_1$ is random, let $w$ be a point of $(\delta/2)\IZ^2$ nearest to it, a function of $\gamma[0,\tau_{j,k}]$ with countably many values. A crossing as in $E_{j,k}$ contains one of $A(w,7\delta/2,\eta/8-\delta/2)\subseteq A(z_1,3\delta,\eta/8)$, also unforced at time $\tau_{j,k}$, as its component in the smaller annulus lies in a non-disconnecting one in the larger and $\partial B(w,7\delta/2)$ meets the connected set $\partial D_{\tau_{j,k}}\supseteq\{z_1\}\cup\partial\ID$. By Lemma \ref{lem:G3}, applied at $\tau_{j,k}$ to each value of $w$, with $r = 7\delta/2$ and $R = \eta/8-\delta/2>7\eta/64$,
\[
\IP(E_{j,k}) \ \le\ K' \Bigl( \frac{32\delta}{\eta} \Bigr)^{\Delta_{\mathrm{KS}}} .
\]
Combining with the Claim, \eqref{eq:Ntail}, and the union bound over the at most $(m+2)^2 \le C \eta^{-8}$ pairs,
\[
\IP\bigl( E(\delta,\eta) \bigr) \ \le\ \IP(N > m) \; + \sum_{0 \le j < k \le m} \IP(E_{j,k})
\ \le\ C_2\, \eta \; + \; C_3\, \eta^{-8} \Bigl( \frac{\delta}{\eta} \Bigr)^{\Delta_{\mathrm{KS}}} .
\]
With $\delta \le \eta^{1+\epsilon}$ and $\epsilon := 9/\Delta_{\mathrm{KS}}$, the second term is at most $C_3\, \eta^{\epsilon \Delta_{\mathrm{KS}} - 8} = C_3\, \eta$. Hence $\IP(E(\delta,\eta)) \le C_4 \eta \le 2 \eta^{1/2}$ for $\eta \le \eta_0$, which is the assertion with $\alpha = 1/2$.

\end{proof}

\subsection{Convergence of Discrete Domains}
We next compare conformal maps of H\"older domains with those of their discrete approximations.
\begin{lemma}\label{lem:discrete}
Let $D^n$ be a $1/n$-discrete approximation of a domain $D$, $w_0$ be a point in $D^n\subset D$ and $\psi$, $\psi^n$ be the conformal maps of the unit disc $\ID$ to $D$ and $D^n$ respectively with $\psi(0)=\psi^n(0)=w_0$, $\psi'(0)>0$, $(\psi^n)'(0)>0$.

Suppose that $\psi$ extends to an $\alpha$-H\"older map on $\overline\ID$, $0<\alpha\le1$, and set $\beta_{\mathrm{map}} := \alpha^2/100$.
Then there is a parameter $N$ depending only on  $\alpha$, the $\alpha$-H\"older norm of $\psi$, diameter of $D$, and on the conformal radius $\psi'(0)$, such that for $n>N$ and any $z\in \ID$ we have
$$|\psi(z)-\psi^n(z)|<n^{-\beta_{\mathrm{map}}}.$$
\end{lemma}

\begin{proof}
Set $\Gamma:=\partial D^n$. Let $w_1,w_2\in\Gamma$ be joined by a crosscut $\chi\subset D^n$ of diameter at most $(2/n)^{1/5}$. By Definition \ref{def:approx}, there are arcs $\chi_1,\chi_2\subset D\setminus D^n$ of diameter at most $1/n$ joining $w_1,w_2$ to $\partial D$. Set $\chi_0:=\chi_1+\chi+\chi_2$, and let $D^n_0$ be the component of $D^n\setminus\chi$ not containing $w_0$. We claim that, for large $n$,
\begin{equation}\label{eq:Cara}
\diam(D^n_0)\leq n^{-\alpha/20}.
\end{equation}
If $\chi_1\cap\chi_2=\emptyset$, then $\chi_0$ is a crosscut of $D$ of diameter at most $2n^{-1/5}$. For the component $D_0$ cut off from $w_0$, Beurling's estimate and Koebe's theorem give
\[
\diam(\psi^{-1}(D_0))\le C_1\sqrt{\diam(\chi_0)}\le n^{-1/12},
\]
where $C_1$ depends only on $\psi'(0)$. Since $D^n_0\subset D_0$ and $\psi$ is $\alpha$-H\"older,
\[
\diam(D^n_0)\le \|\psi\|_\alpha\diam(\psi^{-1}(D_0))^\alpha\le n^{-\alpha/20}.
\]
If $\chi_1\cap\chi_2\ne\emptyset$, then $\chi_0$ contains a closed loop $\Lambda$ of diameter at most $(2/n)^{1/5}+2/n$, and $D^n_0$ lies in a bounded component of $\IC\setminus\Lambda$. Hence $\diam(D^n_0)\le\diam(\Lambda)\le n^{-\alpha/20}$, proving \eqref{eq:Cara}.

Let $\eta_n$ be the structure modulus of $D^n$ at $w_0$ in the notation above.
Thus \eqref{eq:Cara} gives $\eta_n((2/n)^{1/5})\le n^{-\alpha/20}$.
Writing $r=\psi'(0)$, condition 4 of Definition~\ref{def:approx}
and Koebe's theorem give
$D_i(\partial D,\partial D^n)\le1/n$ and
$\dist(w_0,\partial D^n)\ge r/4-1/n\ge r/8$ for large $n$.
The Warschawski estimate recorded above, with $\varepsilon=2/n$, therefore gives
\begin{equation}\label{eq:inverse}
\sup_{w\in D^n}|(\psi^n)^{-1}(w)-\psi^{-1}(w)|
\le C(r,\diam D)n^{-\alpha/40}
\le n^{-\alpha/50}.
\end{equation}
Substituting $w=\psi^n(z)$ gives
\begin{equation}\label{eq:approx}
|z-\psi^{-1}(\psi^n(z))|\leq n^{-\alpha/50}.
\end{equation}
Applying the H\"older bound once more yields
\[
|\psi(z)-\psi^n(z)|\leq\|\psi\|_\alpha n^{-\alpha^2/50}
\leq n^{-\alpha^2/100}=n^{-\beta_{\mathrm{map}}}
\]
for large $n$.
\end{proof}

\begin{remark}\label{rem:discretemaps}
The same estimate holds for the endpoint-normalized maps $\mathfrak p^{-1}$ and $(\mathfrak p^n)^{-1}$. Indeed, set $f:=\mathfrak p^{-1}$, $q_n:=\mathfrak p^n(c)$, and choose $A_n(z)=(q_n+\lambda_nz)/(1+\overline{q_n}\lambda_nz)$, $|\lambda_n|=1$, so that $F_n:=(\mathfrak p^n)^{-1}\circ A_n$ and $f$ have the same derivative argument at $0$. Put $\delta_n:=n^{-\alpha/50}$. The proof of Lemma~\ref{lem:discrete}, after a common rotation of the disk, gives
$$\|F_n^{-1}-\mathfrak p\|_{\infty,D^n}\le\delta_n,\qquad \|F_n-f\|_\infty=O(\delta_n^\alpha).$$
By Beurling's estimate, the caps in Definition~\ref{def:approx}(5) have $\mathfrak p$-images of diameter $O(n^{-1/2})$. Approaching $a^n,b^n$ from their marked prime-end sides in the inverse estimate gives $\zeta_\pm:=A_n^{-1}(\pm1)=\pm1+O(\delta_n)$. The automorphism identity in the proof of Lemma~\ref{lem:discreteCrad} gives $|q_n|=O(\delta_n)$, and either endpoint identity gives $|\lambda_n-1|=O(\delta_n)$. Hence $\|A_n^{-1}-\mathrm{id}\|_\infty=O(\delta_n)$, and
$$
\begin{aligned}
\|(\mathfrak p^n)^{-1}-f\|_\infty
&\le \|F_n-f\|_\infty+\|f\|_\alpha\|A_n^{-1}-\mathrm{id}\|_\infty^\alpha\\
&=O(n^{-\alpha^2/50})<n^{-\beta_{\mathrm{map}}}
\end{aligned}
$$
for large $n$, with the same parameter dependence as in Lemma~\ref{lem:discrete}.
\end{remark}

 \subsection{Proof of Main Theorem}\label{subsec:mainthm}
  We follow the proof of Theorem 4.3 in \cite{V}.

 \begin{proof}[Proof of Theorem \ref{mainthm}]
Fix the parameters first; every choice below depends only on the quantities listed in the statement.

\emph{Step 0: parameters.} Let $r_0 := s/16$ be the driving exponent of Theorem \ref{drivingcvg} and $\epsilon_n := n^{-r_0}$. Fix $\beta := \tfrac12(1 + \beta_+) \in (\beta_+, 1)$ and $q := \tfrac12 q(\beta) > 0$, from Proposition \ref{derest1}; these depend only on $\kappa$. Fix the deterministic-lemma parameter $L := 2>1+\beta$. Let $\epsilon = 9/\Delta_{\mathrm{KS}}$ be the exponent of Proposition \ref{tip} and choose
\[
p := \frac14 \in \Bigl(0, \frac{1}{L}\Bigr),
\qquad d_n := \epsilon_n^{\,p} = n^{-r_0 p},
\qquad \bar\delta_n := C_0\,\epsilon_n^{\,p(1-\beta)},
\]
with $C_0$ from condition 3 of Lemma \ref{Vmain} for these parameters and $c':=c_B$ of Step 1. As $p(1-\beta)<1-Lp$, $\bar\delta_n$ is the scale $\bar\delta$ of that condition for driving terms within $\epsilon_n$ of each other. Set $p_\star := \frac{1}{1+\epsilon}$ and $\alpha_\star := \frac12$, the tip exponents of Proposition \ref{tip} as in Corollary \ref{cor:tip}; both depend only on the KS modulus.

\emph{Step 1: three events.} Couple $\tilde\gamma^n$ with $\sqrt{\kappa}B$ as in Theorem \ref{drivingcvg}, applied with the parameter $\rho/2$, and let $\widehat{T} := T^n_{\rho/2}$ be the horizon \eqref{eq:Tnrho} of the discrete curve. Define:
\begin{enumerate}
\item $A_n$: $\displaystyle \sup_{t \in [0, \widehat{T}]} |W^n(t) - W(t)| \le \epsilon_n$. By Theorem \ref{drivingcvg}, $\IP(A_n^c) \le n^{-r_0}$ for $n \ge n_\star$.
\item $B_n$: the chordal SLE$_\kappa$ chain $(f_t)$ driven by $W$ satisfies $\sup_{t\in[0,T]} d\, |f'(t, W(t)+id)| \le c_B\, d^{1-\beta}$ for all $d \le d_n$. By Proposition \ref{derest1}, $\IP(B_n^c) \le c_B'\, d_n^{\,q} = c_B'\, n^{-r_0 p q}$.
\item $C_n$: the tip structure modulus of $\gamma^n_\ID:=\mathfrak p^n(\gamma^n)$, stopped when it first reaches distance $\rho/4$ from $1$, satisfies $\eta^{(n)}_{\mathrm{tip}}(2\bar\delta_n) \le (2\bar\delta_n)^{p_\star}$. Proposition \ref{tip}, with stopping radius $\rho/4$, $\delta=2\bar\delta_n$, and $\eta=(2\bar\delta_n)^{p_\star}$, gives $\IP(C_n^c) \le 2(2C_0)^{\alpha_\star p_\star} n^{-\alpha_\star p_\star r_0 p(1-\beta)}$ for $n \ge n_2=n_2(M_{\mathrm{KS}},\rho,r_0p(1-\beta),C_0)$, chosen so that $\eta<\eta_0$ and $\eta\le\rho/128$.
\end{enumerate}

\emph{Step 2: deterministic comparison.} On $A_n\cap B_n\cap C_n$, the Cayley comparison with $\sigma=\rho/2$ applies up to $\widehat T=T^n_{\rho/2}$ for $n$ large enough that $\bar\delta_n<\rho/8$, and gives
$$
\eta_{\mathrm{tip}}^{\IH}(\bar\delta_n)
\le \frac{8}{\rho^2}(2\bar\delta_n)^{p_\star}
=C_\rho\bar\delta_n^{p_\star}.
$$
Thus Lemma \ref{Vmain} applies pathwise with $T'=\widehat T\le T$, driving error $\epsilon_n$, the parameters of Step 0 and $r=p_\star$, yielding
\begin{align*}
\sup_{t \in [0, \widehat{T}]} \bigl| \tilde\gamma^n(t) - \tilde\gamma(t) \bigr|
\ &\le\ c'' \max\Bigl\{ \epsilon_n^{\,p (1-\beta) p_\star},\ \epsilon_n^{\,(1 - 2p)\, p_\star} \Bigr\}
\ =\ c''\, n^{-u_0}, \\
&\qquad u_0 := r_0\, p_\star \min\bigl\{ p(1-\beta),\ 1 - 2p \bigr\}.
\end{align*}

\emph{Step 3: exponent and probability.} Set
\begin{equation}\label{eq:uformula}
u \ :=\ \tfrac12 \min\Bigl\{\ r_0\, p_\star\, p\, (1-\beta),\ \ r_0\, p_\star (1 - L p),\ \ r_0,\ \ r_0\, p\, q,\ \ r_0\, p\,(1-\beta)\, p_\star\, \alpha_\star\ \Bigr\},
\end{equation}
where the factor $\tfrac12$ absorbs the constants, including the Lipschitz factor $d_*(z,w)\le2|z-w|$. Then, for $n \ge N = N(T, \rho, R, R_1, s, \kappa, \text{admissibility and KS data})$,
\[
\IP\Bigl( \sup_{t\in[0,\widehat{T}]} d_*\bigl(\tilde\gamma^n(t), \tilde\gamma(t)\bigr) > n^{-u} \Bigr) \ \le\ \IP(A_n^c) + \IP(B_n^c) + \IP(C_n^c) \ <\ n^{-u} .
\]

\emph{Step 4: horizon comparison.} We next replace $\widehat{T} = T^n_{\rho/2}$ by $T_\rho$. If $\widehat T<T_\rho$, then $\widehat T<T$ and the definition of $\widehat T$ gives $d_*(\tilde\gamma^n(\widehat T),\infty)\le\rho/2$. On the event of Step 3, however,
\[
d_*(\tilde\gamma^n(\widehat T),\infty)
\ge d_*(\tilde\gamma(\widehat T),\infty)-n^{-u}
\ge\rho-n^{-u}>\rho/2
\]
for $n$ large, a contradiction. Thus $\widehat T\ge T_\rho$, and the supremum over $[0,T_\rho]$ is covered by Step 3. This proves the first assertion.

\emph{Step 5: transfer to the physical domain.} Assume $D$ is $\alpha$-H\"older. Set $\psi:=\mathfrak p^{-1}$, $\psi^n:=(\mathfrak p^n)^{-1}$, $w^n(t):=\Phi^{-1}(\tilde\gamma^n(t))$ and $w(t):=\Phi^{-1}(\tilde\gamma(t))$. Then $\gamma^n(t)=\psi^n(w^n(t))$ and
$$|w^n(t)-w(t)|=d_*(\tilde\gamma^n(t),\tilde\gamma(t)).$$
On the event of Steps 3--4, for $t \le T_\rho$ and $n \ge N$:
\[
\bigl| \gamma^n(t) - \psi(w(t)) \bigr|
\ \le\ \underbrace{\bigl| \psi^n(w^n(t)) - \psi(w^n(t)) \bigr|}_{\le\ C n^{-\beta_{\mathrm{map}}}}
\ +\ \underbrace{\bigl| \psi(w^n(t)) - \psi(w(t)) \bigr|}_{\le\ \|\psi\|_{\alpha} \, n^{-\alpha u}}
\ \le\ n^{-\lambda},
\]
by Lemma \ref{lem:discrete} and Remark \ref{rem:discretemaps}, with $\lambda = \tfrac12\min\{\beta_{\mathrm{map}}, \alpha u\}$ as in the statement. Since $\lambda \le u$, the probability bound of Step 3 applies verbatim. This proves the last assertion.
\end{proof}

\begin{remark}
The exponent $u$ can be improved by maximizing the right-hand side of \eqref{eq:uformula} over the admissible choices of $p$, $L$, $\beta$, and $q$; we do not pursue this optimization.
\end{remark}

\section{Applications}
\label{applications}
We review the modified bond percolation model of \cite[\S 2.2]{CL06}, which includes triangular-lattice site percolation \cite{SMIRNOV2001239}, and apply the framework to these models.

\subsection{Percolation}
\label{percolation}

Let $D\subset\IC$ be simply connected, and let the distinct boundary prime ends $a,c$ divide $\partial D$ into arcs $c_1,c_2$, coloured blue and yellow, respectively.

\smallskip
\subsubsection{Review of the Model}
\smallskip

Consider the hexagonal tiling dual to the triangular lattice.

\begin{definition}
A \emph{flower} consists of an \emph{iris} and its six neighboring \emph{petals}; all other hexagons are \emph{fillers}. A \emph{floral arrangement} $D_{\cF}$ designates irises so that none lies on the boundary, any two are separated by at least two non-iris hexagons, and the infinite-volume arrangement is periodic with $60^\circ$ symmetry.
\end{definition}

\begin{definition}
Let $D$ carry a floral arrangement $D_{\cF}$.
\begin{itemize}
    \item Fillers and petals are independently blue or yellow with probability $1/2$.
    \item Except in a trigger configuration, an iris is blue or yellow with probability $\mathfrak{a}$ each, or is in one of the horizontal, $120^\circ$ and $60^\circ$ mixed states with probability $\varsigma$ each, where $2\mathfrak{a}+3\varsigma=1$ and $\mathfrak{a}^2\ge2\varsigma^2$.
    \item A \emph{trigger} has three petals of each colour and exactly one contiguous pair of each; its iris is blue or yellow with probability $1/2$.
\end{itemize}
\end{definition}

\begin{remark}
Triggering is the only local dependence; distinct flowers and the background fillers are independent. At $\varsigma=0$ the model is triangular-lattice site percolation. The critical crossing estimates and Cardy's formula are \cite[Theorems 3.10 and 2.4]{CL06}.
\end{remark}

\subsubsection{Exploration Process}
\smallskip

Given $(D,a,c)$ as above, take a hexagon tiling of side $1/(2n)$, with the locations of all irises, flowers and fillers fixed. After discarding partial flowers, let $\Gamma^n$ be the connected component of its intersection with $D$ containing the conformal basepoint $\mathfrak p^{-1}(0)$, and choose the vertices $a^n,c^n$ closest to $a,c$ in the Carath\'eodory metric. This component exists for all sufficiently large $n$. The associated $(D^n,a^n,c^n)$ is a $C/n$-discrete approximation in the sense of Definition~\ref{def:approx}: its edges and faces have diameter at most $1/(2n)$ and $1/n$, respectively, each boundary vertex is connected to $\partial D$ through the discarded partial flowers, and by the paragraph after Definition~\ref{def:approx} the chosen vertices satisfy condition 5. Here $C$ is absolute because a flower contains finitely many hexagons; absorbing it into the mesh parameter yields a $1/n$-discrete approximation and changes no exponent below.
\smallskip

\noindent Write $\partial_{\mathrm{lat}}D^n$ for the collection of boundary hexagons and boundary half-irises, as distinct from the topological boundary of the planar domain.  \emph{Admissible domains} $(D^n, \partial_{\mathrm{lat}}D^n, a^n, c^n)$ satisfy the following properties:
\begin{itemize}
    \item $D^n$ contains no partial flowers.
    \item $\partial_{\mathrm{lat}}D^n$ is the union of lattice-connected blue and yellow sets $\gamma^n_1,\gamma^n_2$, made of hexagons or half-irises, which meet at $a^n,c^n$ and form a positive-probability flower configuration.
    \item At each of $a^n,c^n$, the three incident hexagons consist of one blue, one yellow and one interior hexagon.
\end{itemize}

Set $\IX_0^n=a^n$ and explore toward $c^n$, keeping blue on the right. An encountered filler is sampled fairly; petals and irises are sampled from their conditional flower law given the revealed states. Away from an iris, the path takes the next interface edge. At an iris, reveal the iris and, if needed, the next petal, then follow the blue--yellow interface through the flower to the next hexagon vertex. The resulting process has the following domain Markov property.

\begin{proposition}[Proposition 4.3, \cite{SLEI}]\label{admiss}
Let $\gamma^n([0,t])$ be the line segments traced by the process up until a stopping time $t$ at which $\gamma^n(t)$ is a vertex, and let $\Gamma([0,t])$ be the collection of hexagons whose states have been revealed by time $t$ (also denoting their planar union). Let ${\mathcal F}^n_t := \sigma\bigl( \gamma^n([0,t]), \text{ states of all hexagons in } \Gamma([0,t]) \bigr)$ be the full exploration filtration. Define the updated lattice interior and discrete boundary by
\[
\Omega_t^n:=\Omega_0^n\setminus\Gamma([0,t]),
\qquad
\partial_{\mathrm{lat}}\Omega_t^n
:=\partial_{\mathrm{lat}}D^n\cup\Gamma([0,t]).
\]
Thus the revealed hexagons cease to be interior hexagons and become part of the discrete boundary data; $\partial_{\mathrm{lat}}\Omega_t^n$ is not the topological boundary of an open subset of $\IC$. The unrevealed region $\Omega_t^n$ may have several connected components. Then the decorated lattice domain
\[
\mathfrak D_t^n:=\bigl(\Omega_t^n,\partial_{\mathrm{lat}}\Omega_t^n,\IX_t^n,c^n\bigr)
\]
carries the induced boundary data, including any untraced blue--yellow contacts and forced edges. Conditionally on ${\mathcal F}^n_t$, the remaining Exploration Process has the law of the exploration in this state from $\IX_t^n$ to $c^n$.
\end{proposition}

\smallskip
\noindent\emph{Exploration-slit approximation.}
Write $D_t^n=D^n\setminus\gamma^n[0,t]$ and $u_t^n=(\varphi_t^n-W_t^n)^{-1}(i)$. The Loewner maps, $\Xi_t^n$ and $\Crad_t^n$ refer to $D_t^n$; the domain Markov property refers to the full decorated state. Let $U_t^n$ be the component of $\Omega_t^n$ containing $u_t^n$. Every revealed hexagon is within $C/n$ of the trace. Thus, when $\Crad_t^n\ge r>0$, \eqref{eq:basepoint} gives $\dist(u_t^n,\partial D_t^n)\ge r/2$, ensuring that $U_t^n$ exists for $n\ge N(r,R_1)$. It is simply connected because every revealed component meets the trace, which is attached to the original boundary. Each new boundary edge has a connection of diameter $C/n$ through its revealed flower to the corresponding side of the trace; original boundary edges have the analogous connection to $\partial D_t^n$. The argument after Definition~\ref{def:approx} supplies the marked representatives on their prime-end sides and in cyclic order. They are distinct under \eqref{eq:sep}, so $U_t^n$ is a $C/n$-discrete approximation of $D_t^n$. Here $C$ depends only on the fixed flower geometry. Lemma~\ref{lem:localexploration} identifies the effective endpoints and preserves the representatives for site percolation.

\smallskip
\subsubsection{Percolation satisfies the KS Condition}

 Let $\IP_{D^n}$ be the law on random curves induced by the Exploration Process on $D^n$. As always, we endow the space of curves with the sup-norm metric $\mathrm{dist}(\gamma_1,\gamma_2) = \inf\limits_{\varphi_1,\varphi_2}\sup\limits_t |\gamma_1(\varphi_1(t)) - \gamma_2(\varphi_2(t)) | $ over all possible parameterizations $\varphi_1,\varphi_2$.

\begin{remark}\label{rem:harris}
As long as $\mathfrak{a}^2\geq 2\varsigma^2$, a restricted form of the Harris--FKG property holds for path events; see \cite[Lemma~A.2]{CL06}.
\end{remark}

\begin{proposition}[adaptation of Proposition 4.13, \cite{KS}]\label{percks}
The collection of the laws of the interface of the modified bond percolation model described above on the hexagonal lattice,
\begin{equation}
    \Sigma_{\mathrm{Percolation}} = \{ (D^n, \mathfrak{p}^n, \IP_{D^n}): D^n \text{ an admissible domain} \} \label{perc}
\end{equation}
satisfies the KS Condition, with a modulus depending only on the model parameters $(\mathfrak{a},\varsigma)$ (subject to $\mathfrak{a}^2 \ge 2\varsigma^2$).
\end{proposition}

The crossing estimate that replaces the independent-percolation input of \cite[Proposition~4.13]{KS} is the RSW estimate for the flower
model, \cite[Theorem~3.10]{CL06}. Its constants depend only on $(\mathfrak{a},\varsigma)$. By Proposition~\ref{admiss}, conditional on the
exploration the unexplored region carries the same flower model with the induced boundary data. The conditioning of boundary flowers is
controlled by \cite[Lemma~3.19]{CL06}, while the restricted Harris--FKG estimate \cite[Lemma~A.2]{CL06} applies to the path events
used in the RSW gluing argument. Thus the annulus-crossing estimate of \cite[Proposition~4.13]{KS} holds with a modulus depending only on
$(\mathfrak{a},\varsigma)$, uniformly over the admissible slit domains.

\smallskip
\subsubsection{The Martingale Observable}
\smallskip

 Consider two additional marked points (or prime ends) $b$, $d$ so that $a$, $b$, $c$, $d$ are in cyclic counterclockwise order. Let $D^n$ be the \emph{admissible domain} described above at hexagon side $\frac1{2n}$, hence mesh $1/n$, for the domain $D$. More details of the construction can be found in \cite[\S 3 and \S 4]{SLEI} and \cite[\S 4.2]{SLEII}. The boundary arcs can be appropriately coloured, and the lattice points $a^n,b^n,c^n,d^n$ can be selected. The principal observable is the probability $\cC^n$ of a blue crossing of the conformal rectangle $D^n$ from $(a^n,b^n)$ to $(c^n,d^n)$; its limit $\cC^\infty$ is Cardy's formula in $D$. Each percolation configuration with these boundary conditions determines a unique interface connecting $a^n$ to $c^n$, namely the curve separating the boundary-connected blue and yellow clusters. Let us temporarily omit the marked point $a^n$ and consider the conformal triangle $(D^n;b^n,c^n,d^n)$.

Following \cite{SMIRNOV2001239}, for a lattice point $z\in D^n$ let $S^n_d(z)$ be the probability of a yellow crossing from $(c^n,d^n)$ to $(d^n,b^n)$ separating $z$ from $(b^n,c^n)$, and define $S^n_b,S^n_c$ by cyclic permutation of $b,c,d$. Thus $S^n_d$ has boundary values $0$ on $(b^n,c^n)$ and $1$ at $d^n$, with the analogous identities for the other two functions. The \emph{Carleson--Cardy--Smirnov} function is
$S^n:=S^n_b+\tau S^n_c+\tau^2S^n_d$, where $\tau=e^{2\pi i/3}$. Extend each function affinely along lattice edges and over an arbitrary triangulation of every face, retaining the same notation.

For a fresh site-percolation configuration, colour symmetry gives
$S^n(a^n)=1+(\tau^2-1)\cC^n$.

\begin{lemma}\label{lem:mgo}
At a bounded stopping time $\sigma$, select a blue crossing event $E$ with
${\mathcal F}^n_\sigma$-measurable domain and boundary data, and keep this
event fixed thereafter, including its boundary incidences and slit sides. Then
$$M_t:=1+(\tau^2-1)\IP(E\mid{\mathcal F}^n_t),\qquad t\ge\sigma,$$
is a bounded martingale for the full exploration filtration. Its
martingale identity holds at bounded stopping times.
\end{lemma}
\begin{proof}
Apply the tower property and optional sampling. Choices made only on
$G\in{\mathcal F}^n_\sigma$ are extended arbitrarily off $G$.
\end{proof}

 The CCS functions $S^n$ are not \emph{discrete analytic} but are ``almost'' discrete analytic in the following sense, see \cite[\S 4]{BCL}:
 \begin{definition}[$(\sigma,\vartheta)$-Holomorphic]
\label{nhf}
Let $D \subseteq \mathbb C$ be a simply-connected domain and $D^n$ be a $\frac1n$-approximation and let $(Q^{n}: D^{n} \to \mathbb C)_{n\ge1}$ be a
sequence of functions defined on the vertices of $D^n$.
We say that the sequence $(Q^n)$ is \emph{$(\sigma,\vartheta)$--holomorphic} if there exist constants $0 < \sigma, \vartheta \leq 1$ such that for all $n$ sufficiently large:
\begin{enumerate}
\item $Q^{n}$ is H\"older continuous up to $\partial D^n$: there exists $\epsilon > 0$ and constants $c,C\in (0,\infty)$ (independent of $n$) such that
\begin{enumerate}
\item if $z_n, w_n \in D^n $, $\dist(z_n,\partial D^n)>\epsilon$ with $|z_n - w_n| < \epsilon$, then $|Q^n(z_n) - Q^n(w_n)| \leq c \left( \frac{|z_n - w_n|}{\epsilon} \right)^\sigma$; and
\item if $z_n \in D^n$, $\dist(z_n, \partial D^n)\leq\epsilon$ then there exists some $w_n^\star \in \partial D^n$ such that $ |Q^n(z_n) - Q^n(w_n^\star)| \leq C \left(\frac{|z_n - w_n^\star|}{\epsilon} \right)^\sigma$.
\end{enumerate}

\item For any simple closed lattice contour $\Lambda$,
\begin{equation}
\label{sint}
\left| \oint_{\Lambda} Q^n~dz
\right| = \left|\sum_{h_n \subseteq D^n_\Lambda} \oint_{\partial h_n} Q^n~dz \right|
\leq c\cdot |\Lambda| \cdot n^{-\vartheta},
\end{equation}
where $D^n_\Lambda, |\Lambda|$ denote the region enclosed by $\Lambda$ and the Euclidean length of $\Lambda$, respectively.
\end{enumerate}
\end{definition}

\begin{proposition}[Uniform form of Proposition 4.3 of \cite{BCL}] \label{sfp}
There are $\sigma\in(0,1]$ and $C<\infty$, depending only on the model, with the following property. Let $\eta>0$ and $n\ge C/\eta$, and let $\Omega$ be a simply-connected union of hexagons of mesh $1/n$, with marked boundary points $b,c,d$ and a point $w_0\in\Omega$ such that every crosscut of $\Omega$ separating two of $b,c,d$ from $w_0$ has diameter at least $\eta$. Then the CCS function $S$ of $(\Omega;b,c,d)$ satisfies
\[
|S(z)-S(w)|\le C\Bigl(\frac{r}{\eta}\Bigr)^{\sigma}
\]
whenever $z,w\in\overline\Omega$ are joined by a curve of diameter $r$ that lies in $\Omega$ except possibly for its endpoints, and $\bigl|\oint_\Lambda S\,dz\bigr|\le C\eta^{-\sigma}|\Lambda|\,n^{-\sigma}$ for every simple closed lattice contour $\Lambda\subseteq\overline\Omega$. In particular, $S$ satisfies items~1 and~2 of Definition~\ref{nhf} with $\vartheta:=\sigma$, $\epsilon:=\eta/8$, and constants depending only on the model and $\eta$.
\end{proposition}

\begin{proof}
Write $X_b,X_c,X_d$ for the arcs of $\partial\Omega$ opposite to $b,c,d$. By \cite[Theorem~3.10]{CL06}, as in Proposition~\ref{percks}, choose $C_{\mathrm{RSW}}\ge1$ and $c_{\mathrm{RSW}}>0$, depending only on the model, so that, for the configuration extended independently outside $\Omega$, every annulus $A(z,s,2s)$ with $s\ge C_{\mathrm{RSW}}/n$ contains a circuit of a prescribed colour with probability at least $c_{\mathrm{RSW}}$, independently for annuli at mutual distance at least $C_{\mathrm{RSW}}/n$, and set $\sigma:=\min\{1,\frac12\log_2\frac1{1-c_{\mathrm{RSW}}}\}$.

For a lattice point $v$, $S_x(v)$ is the probability of the event $E_x(v)$ that a yellow crossing between the two arcs adjacent to $x$ separates $v$ from $X_x$. Let $z,w$ be distinct lattice points joined by a curve $\pi$ of diameter $r$ that lies in $\Omega$ except possibly for its endpoints; then $r\ge1/(2n)$. The bound is trivial unless $r<\eta/20$, so assume this. If $E_x(z)$ and $E_x(w)$ differ, say $E_x(z)$ holds, let $\gamma$ be the lowest yellow crossing separating $z$ from $X_x$, that is, the one closest to $X_x$. It does not separate $w$, so $\pi$ meets it at a hexagon $h$, and, as $\gamma$ is lowest, duality gives a blue path from a neighbour of $h$ to $X_x$. With the two halves of $\gamma$, this gives three paths from $\overline B(z,5r)$ to the three arcs, as $1/n\le2r$. One of them leaves $B(z,\eta/4)$. Otherwise, together with $h$ and its neighbour, they would contain a tripod in $\overline B(z,\eta/4)$ with ends on the three arcs; moved off $w_0$ if necessary, it cuts $\Omega$ into three regions, one for each marked point, and its two legs bounding the region of $w_0$ form a crosscut of diameter less than $\eta$ separating the other two marked points from $w_0$, contradicting the hypothesis. Subdividing $A(z,\bar r,\eta/4)$, where $\bar r:=\max\{5r,C_{\mathrm{RSW}}/n\}\le5C_{\mathrm{RSW}}r$, into $K:=\lfloor\log_2\frac{\eta}{4\bar r}\rfloor$ dyadic annuli, each crossed inside $\Omega$ by that path and hence free of circuits of the other colour, and using every other one, gives $\IP(E_x(z)\triangle E_x(w))\le2(1-c_{\mathrm{RSW}})^{K/2}\le C(r/\eta)^{\sigma}$. Summing over $x\in\{b,c,d\}$ yields the H\"older bound for lattice points; it extends to all $z,w$ with a larger $C$, since $S$ is affine on edges and faces, of diameter at most $1/n$, and $\sigma\le1$.

For a simple closed lattice contour $\Lambda\subseteq\overline\Omega$, $\oint_\Lambda S\,dz$ is the sum of $\oint_{\partial h}S\,dz$ over the hexagons $h$ enclosed by $\Lambda$. The at most $Cn|\Lambda|$ of them adjacent to $\Lambda$ contribute at most $Cn^{-1}(n\eta)^{-\sigma}$ each, as $\oint_{\partial h}dz=0$ and $S$ varies by at most $C(n\eta)^{-\sigma}$ on $h$, by the previous step. On the others, all of whose vertices are interior, the colour-switching identities of \cite{SMIRNOV2001239} (see also \cite[\S4.3]{CL06}) reduce the sum to at most $Cn|\Lambda|$ terms, each bounded by $n^{-1}\IP(E_x(v)\triangle E_x(v'))\le Cn^{-1}(n\eta)^{-\sigma}$ for some $x$ and adjacent lattice points $v,v'$ within $2/n$ of $\Lambda$. Item~1 of Definition~\ref{nhf} follows by taking straight segments, with $w_n^\star$ a nearest boundary point in (b).
\end{proof}

\smallskip
\subsubsection{Power-law convergence of the Martingale Observable Function to its Continuous Counterpart}
\smallskip

The quantity $\cC^n$ is realized from $S^n(a^n)$ as $\cC^n = \frac{-2}{\sqrt{3}}\cdot \Im[S^n(a^n)]$. Since it is already known that $S^n$ converges to $H:D\to T$, a conformal map to equilateral triangle $T$ which sends $(b,c,d)$ to $(1,\tau,\tau^2)$, we can see that $\cC^\infty = \frac{-2}{\sqrt{3}} \Im[H(a)]$ (see \cite{SMIRNOV2001239, Beffara, SLEII}).

\begin{theorem}
\label{Obscvg}
Let $ D$ be a domain with four marked boundary points (or prime ends) $a$, $b$, $c$, and $d$. Let $( D^n, a^n, b^n, c^n, d^n)$ be its admissible $\frac1n$-discrete approximation. Consider the site percolation model or the generalized model introduced at the beginning of this section on the domain $ D^n$. In the latter case, we also assume that the boundary Minkowski dimension is less than 2 (in the former, this is not necessary).  Let $\varphi$ be a conformal map of $(D, a,c)$ to $(\IH,0,\infty)$. Assume that $\varphi(b)\in (1,2)$, $\varphi(d)\in(-2,-1)$, $\Crad(\varphi) > R$ and $\diam(D) \le R_1 < \infty$. Then there exists $s > 0$, depending only on the model --- and in particular neither on the domain $ D$, nor on $R$, nor on $R_1$ --- such that $S^n(a^n)$ converges to its limit $H(a)$ with the estimate
$$
|S^n(a^n) - H(a)| \leq  n^{-s},
$$
 provided $n \geq n_{0}(R, R_1)$ is sufficiently large.
\end{theorem}

This strengthens the Main Theorem of \cite{BCL}: the exponent $s$ is uniform over the normalized class, the domain dependence of constants and thresholds being confined to $R$ and $R_1$. Stated for $S^n(a^n)$, it is equivalent to the statement for $\cC^n$, since $S^n_c(a^n)=0$, and $S^n_b(a^n)+S^n_d(a^n)=1$ by duality and colour symmetry, so that $S^n(a^n)=1+(\tau^2-1)\,\cC^n$. In the seven comparisons below, exponents are independent of $D$, $R$ and $R_1$; constants and thresholds may depend on $R$ and $R_1$.

\emph{Separation of the marked points.} Under the above normalization, the four points $\varphi(a) = 0$, $\varphi(b) \in (1,2)$, $\varphi(c) = \infty$, $\varphi(d) \in (-2,-1)$  are separated by an absolute positive distance in the compact metric $d_*$ on $\overline{\IH}$. Let $w_0:=\varphi^{-1}(i)$. By Koebe's one-quarter theorem, \[ \dist(w_0,\partial D) \ge \frac14\bigl|(\varphi^{-1})'(i)\bigr| \ge \frac R4. \] Consequently, conformal invariance of harmonic measure and the Beurling estimate imply that any crosscut or closed curve in $D$ separating two of the marked prime ends together from $w_0$ has diameter at least $\eta_0 R$, where $\eta_0>0$ is absolute. Therefore
\begin{equation}\label{eq:sep}
\min_{x\neq y\in\{a,b,c,d\}}\bigl( {\textstyle\dist_C}(x,y)\bigr) \ \ge \eta_0R =:\eta_R,
\end{equation}
where $\dist_C$ is the Carath\'eodory distance in $D$ with base point $w_0$ (Definition \ref{def:carametric}).
The same separation holds, with $\eta_R/2$ in place of $\eta_R$ once $n\ge n_1(R,R_1)$, for $D^n$ and for the domain $D^n_\square$ below with the marked points of Lemma~\ref{s1}; for the latter, see the paragraph after that lemma. For $D^n$, whose four marked points are each $1/n$-close to those of $D$ as in Definition~\ref{def:approx}(5), extend a separating crosscut of $D^n$ at its ends by the curves $\chi$ of Definition~\ref{def:approx}(4); together with the crosscuts of Definition~\ref{def:approx}(5), this costs at most $4/n<\eta_R/2$ in diameter. These are the hypotheses of Proposition~\ref{sfp} with base point $w_0$ and $\eta:=\eta_R/2$.

Below, \emph{absolute} exponents depend only on the model data of Proposition~\ref{sfp} and the fixed normalization. Constants and thresholds may depend on $R,R_1$. We absorb prefactors into thresholds by replacing $C(R,R_1)n^{-a}$ with $n^{-a'}$, where $0<a'<a$ is absolute.

\emph{Boundary regularization.} Let $a_1\in(0,1)$ and introduce an auxiliary square grid of mesh $n^{a_1-1}$. Let $\mathcal Q^n$ be the union of the grid squares whose closures are contained in $D^n$ and which are joined to a square containing $w_0$ by a chain of such squares, consecutive squares sharing an edge. Define $D^n_{\square}$ to be the union of the hexagons of the original lattice contained in $\mathcal Q^n$. The square grid selects the hexagons, so $\partial D^n_{\square}$ remains a hexagonal-lattice contour. Every square outside $\mathcal Q^n$ is connected through such squares to $\mathbb C\setminus D^n$, and every excluded hexagon meets this connected complement; hence $D^n_{\square}$ has no bounded complementary component and is simply connected. Since $D^n\subseteq D$ and $\diam D\le R_1$, the auxiliary grid has $O(R_1^2n^{2(1-a_1)})$ squares in a bounding square of side $O(R_1)$. Each grid edge corresponds to a portion of $\partial D^n_{\square}$ of length $O(n^{a_1-1})$. Consequently,
\begin{equation}\label{eq:length1} \Length(\partial D^n_{\square}) = O\!\left( R_1^2n^{2(1-a_1)}n^{a_1-1} \right) = O\!\left(R_1^2n^{1-a_1}\right). \end{equation}
Choose $a_1$ close to $1$ so that $1-a_1$ satisfies link~(3).

\emph{The seven-link chain.} With $S^n_\square$, $F^n_\square$, $H^n_\square$, $H^n_\vardiamond$ as below, Theorem \ref{Obscvg} follows from
\[
S^n(a^n) \overset{(1)}{\sim} S^n_\square(a^n_\square) \overset{(2)}{\sim} S^n_\square(a^n_\diamondsuit) \overset{(3)}{\sim} F^n_\square(a^n_\diamondsuit) \overset{(4)}{\sim} H^n_\vardiamond(a^n_\diamondsuit) \overset{(5)}{\sim} H^n_\square(a^n_\diamondsuit) \overset{(6)}{\sim} H^n_\square(a^n_\square) \overset{(7)}{\sim} H(a),
\]
where $A \sim B$ means $|A - B| \le C(R,R_1)\, n^{-c}$ with $c > 0$ absolute and $C(R,R_1)$ depending only on $R$ and $R_1$.

\begin{lemma}[Link (1) and (7); cf. Theorem 3.3 of \cite{BCL}]\label{s1}
Let $D^n_\square \subseteq D^n$ with marked points $(a^n, \dots, d^n)$ be as above. There exist an absolute $a_2 > 0$, a constant $C_2 = C_2(R,R_1)$, a threshold $n_1 = n_1(R,R_1)$, and marked points $a^n_\square, b^n_\square, c^n_\square, d^n_\square$ such that for all $n > n_1$,
\[
|S^n(a^n) - S^n_\square(a^n_\square)| \le C_2(R,R_1)\, n^{-a_2}
\qquad\text{and}\qquad
|H(a) - H^n_\square(a^n_\square)| \le C_2(R,R_1)\, n^{-a_2},
\]
$H^n_\square$ being the conformal map of $(D^n_\square; b^n_\square, c^n_\square, d^n_\square)$ onto $(T; 1, \tau, \tau^2)$.
\end{lemma}

\emph{Uniformity of (1), (7).} Put $\ell=n^{a_1-1}$. Each boundary edge of $D^n_\square$ has a connection to $\partial D^n$, outside $D^n_\square$, of diameter $C\ell$: an adjacent excluded square meets $\partial D^n$, otherwise it would belong to the edge-connected collection of full squares. Use these connections in the boundary-walk construction after Definition~\ref{def:approx}, at scale $C\ell$. For each $x\in\{a^n,b^n,c^n,d^n\}$ it gives a representative $x_\square^n$ and a crosscut $\chi_x$ cutting $x,x_\square^n$ off from $w_0$ in a cap with mouth diameter $C\ell$: $\chi_x$ and the connecting boundary-walk segment lie in $B(x_\square^n,C\ell)$. Retain the prime-end sides. Overlapping or nested caps would, by the triangle inequality for $\dist_C$, contradict \eqref{eq:sep}; their incidences retain cyclic order. Extending a separating crosscut of $D^n_\square$ through its endpoint connections and these marked caps proves separation by $\eta_R/2$, for $n$ sufficiently large.

For the comparison in \cite[Theorem~3.3]{BCL}, take a full-plane circuit surrounding $B(x_\square^n,C\ell)$ in $A(x_\square^n,s,2s)$, where $s>2C\ell$ and $4s<\eta_R/2$. The basepoint is outside the circuit. The cap mouth is connected and lies inside it, so a circuit component in $D^n$ is a crosscut separating that cap, and hence $x$, from $w_0$, by cutting the circuit components in boundary order. Its diameter is at most $4s$, so it cuts off no other marked point. Its endpoints therefore lie on the two boundary arcs adjacent to $x$. In $D^n_\square$ the same argument gives the incidences at $x_\square^n$. These arcs give the crossing-indicator comparison of \cite[Theorem~3.3]{BCL}, with separate prime-end incidences for repeated slit visits. RSW on disjoint annuli between $C\ell$ and $n^{-\theta}$, with fixed $0<\theta<1-a_1$, gives link (1) with a model-dependent exponent. The construction at finer meshes in the two fixed domains gives link (7) by Cardy's limit. Requiring $n^{-\theta}<\eta_R/16$ changes only the threshold. Hence $a_2$ is absolute and the prefactor depends only on $R,R_1$.

We consider the Cauchy integral
\begin{equation}\label{eq:FS}
F^n_\square(z) = \frac{1}{2\pi i} \oint_{\partial D^n_\square} \frac{S^n_\square(\zeta)}{\zeta - z}\, d\zeta .
\end{equation}

\begin{lemma}[Link (3); Lemma 3.5 of \cite{BCL}]\label{ci}
For $a_1$ sufficiently close to $1$ there exist absolute $\beta > 0$ and $a_5 > 0$, and a constant $C_3 = C_3(R,R_1)$, such that for all $z \in D^n_\square$ with $\dist(z, \partial D^n_\square) > n^{-a_5}$,
\[
|S^n_\square(z) - F^n_\square(z)| \le C_3(R,R_1)\, n^{-\beta}.
\]
\end{lemma}

\emph{Uniformity of (3).} We use the exponents $\sigma,\vartheta$ of Proposition~\ref{sfp} and the Cauchy estimate of \cite[Corollary~4.5]{BCL}. Put $h:=n^{-1}$, $d_1:=h^{a_5}$, $d_2:=h^r$, and let the auxiliary block scale be $\ell_n:=h^t$. That estimate gives
\[
|S^n_\square(z)-F^n_\square(z)|
\lesssim
\left(
 \frac{h^\vartheta}{d_2\ell_n}
 +\frac{\ell_n^\sigma}{d_2^2}
\right)
\left(
 \operatorname{Area}(D^n_\square)
 +\Length(\partial D^n_\square)\ell_n
\right)
+\left(\frac{d_2}{d_1}\right)^\sigma .
\]
By $\diam D\le R_1$ and \eqref{eq:length1}, the middle factor is bounded by $C(R_1)\bigl(1+h^{t-(1-a_1)}\bigr)$. Fix the parameters absolutely so that
\[
0<a_5<r,\qquad 1-a_1<t,\qquad r+t<\vartheta,
\qquad 2r<t\sigma .
\]
This is possible by first taking $t,r,a_5$ sufficiently small and then $a_1$ sufficiently close to $1$. Hence the three error terms are bounded by $C(R,R_1)n^{-\beta}$, where
\[
\beta:=\min\{\vartheta-r-t,\ t\sigma-2r,\ \sigma(r-a_5)\}>0 .
\]

\emph{Definition of the shrunken triangle and its preimage.} Fix absolute exponents $a_3, a_4$ with $0 < a_3 < \min\{\beta, a_5\}$ and $0 < a_4 < \sigma a_3$, and set $T_\vardiamond := (1 - n^{-a_4})\, T$, the concentric shrinking of $T$.
Let $w_{T}$ be the barycenter of $T$. By the argument-principle construction in the proof of \cite[Lemma~3.9]{BCL}, there exists $n_*(R,R_1)<\infty$ such that, for $n\ge n_*(R,R_1)$, one can choose $w_0^n\in D^n_\square$ satisfying $F^n_\square(w_0^n)=w_{T}$. Since $w_T \in T_\vardiamond$, we define $D^n_\vardiamond$ to be the connected component of $(F^n_\square)^{-1}(T_\vardiamond)$ containing $w_0^n$.

\begin{lemma}[Link (4); Lemma 3.9 of \cite{BCL}]\label{cdmd}
$F^n_\square$ is conformal (univalent) on $D^n_\vardiamond$.

Let $(b^n_\vardiamond, c^n_\vardiamond, d^n_\vardiamond)$ be the preimages of $(1 - n^{-a_4})(1, \tau, \tau^2)$ under $F^n_\square$. Let $H^n_\vardiamond : D^n_\vardiamond \to T$ be the conformal map sending $(b^n_\vardiamond, c^n_\vardiamond, d^n_\vardiamond)$ to $(1, \tau, \tau^2)$. Then for all $z \in \overline{D^n_\vardiamond}$,
\[
|F^n_\square(z) - H^n_\vardiamond(z)| \lesssim n^{-a_4}.
\]
\end{lemma}

\emph{Uniformity of (4).} Put $d=n^{-a_3}$. Choose a regular parallel boundary of the component of $\{z:\dist(z,\partial D^n_\square)>d'\}$ containing $w_0$, with $d<d'<3d/2$, and perturb it by $O(1/n)$ to a simple lattice contour $\Gamma^n$. Thus its distance from $\partial D^n_\square$ lies between $d$ and $2d$ and Lemma~\ref{ci} applies.

Join each point of the parallel boundary to a nearest outer boundary point. The open segment has distance less than $d'$ from the outer boundary and stays outside the inner component. By the nearest-point property, these segments have disjoint interiors or coincide, and land in cyclic order. At a jump, the limiting segments bound a boundary interval cut off from $w_0$ by a crosscut of diameter at most $2d'$. This interval contains at most one of $b^n_\square,c^n_\square,d^n_\square$, by \eqref{eq:sep}. Its endpoint $S^n_\square$ values differ by $O((d/\eta_R)^\sigma)$ by Proposition~\ref{sfp}; since $S^n_\square$ runs monotonically along the triangle sides, the skipped interval has image of that diameter up to an absolute factor. Fill each jump by that interval, allowing pauses in the inner parameterization. This gives one positive traversal of $\partial T$, uniformly within $C(R,R_1)d^\sigma$ of $S^n_\square(\Gamma^n)$; the lattice perturbation has the same bound.

Lemma~\ref{ci} bounds the parameterized distance by $C(R,R_1)(n^{-\beta}+n^{-\sigma a_3})$ between $F^n_\square(\Gamma^n)$ and $\partial T$. Since $a_4<\sigma a_3$, the straight homotopy between these curves misses $\overline{T_\vardiamond}$ for large $n$. Thus the winding number is one about each point of $T_\vardiamond$. The argument principle gives one preimage inside $\Gamma^n$, counted with multiplicity. Taking $w_0^n$ to be the preimage of $w_T$ proves that its component $D^n_\vardiamond$ maps conformally onto $T_\vardiamond$.

The normalizations give $H^n_\vardiamond=(1-n^{-a_4})^{-1}F^n_\square$, so
$$
|F^n_\square(z)-H^n_\vardiamond(z)|
=\frac{n^{-a_4}}{1-n^{-a_4}}|F^n_\square(z)|\le Cn^{-a_4},
$$
with $C$ absolute, since $F^n_\square(D^n_\vardiamond)\subset T$.

\begin{lemma}[Uniform macroscopic connections]\label{lem:macroRSW}
Assume the radius and marked-point separation bounds above. Each single connection from the central disk to a prescribed boundary arc, and each pair of disjoint connections with the colours and boundary order prescribed in \cite[Lemma~3.12, Proposition~3.14, Lemma~3.18]{BCL}, has probability at least $c_{\rm mac}(R,R_1)>0$, uniformly over the admissible slit domains and all sufficiently large $n$.
\end{lemma}
\begin{proof}
Put $\rho=c\min\{\eta_R,R\}$, with $c>0$ a sufficiently small absolute constant, and let $W$ be the component of $\{\dist(z,\partial D^n_\square)>\rho\}$ containing $w_0$. In the parallel-boundary correspondence of link (4), each skipped interval is behind a crosscut of diameter $2\rho$. It contains at most one marked point. Applying this at scales $4\rho$ and $8\rho$ gives a landing on each boundary arc whose terminal neighbourhood meets that arc on its prescribed side, away from its endpoints in crosscut distance. Otherwise the successive skipped intervals would cut off both endpoints of that arc by a crosscut of diameter at most $16\rho<\eta_R/2$.

Join the interior landing neighbourhoods to $w_0$ in $W$. Cover these paths by squares of side $c'\rho$, erase loops, and retain a planar tree joining the landings. There are at most $C(R_1/\rho)^2$ squares, each with a fixed-size neighbourhood inside the domain. For a single connection, use overlapping bounded-aspect-ratio rectangle crossings along the tree and a terminal separating circuit arc. For two disjoint connections, join parallel strips in boundary order at each branch point, using the double-crossing blocks of \cite[Lemma~3.12 and Figure~4]{BCL}: two longitudinal crossings and two transverse crossings in each overlap square preserve two disjoint paths through a turn. The terminal circuit arcs attach the paths to the prescribed boundary arcs. With a bounded number of rectangles per square, RSW and gluing give the lower bound $c_0^{C(R_1/\rho)^2}$ for each of these single or double connection patterns. Taking the minimum over the finitely many patterns gives $c_{\rm mac}$.
\end{proof}

\begin{lemma}[Links (2), (5), and (6); Prop. 3.14 and Lemma 3.18 of \cite{BCL}] \label{dia2}
Assume that \[ \min_{x\in\{b^n_\square,c^n_\square,d^n_\square\}} \textstyle{\dist_{\mathrm C}^{D^n_\square}}(a^n_\square,x) \ge \frac{\eta_R}{2}, \] as holds for all sufficiently large $n$ by \eqref{eq:sep} and the construction of the marked prime ends. Then one may choose $a^n_\diamondsuit\in D^n_\vardiamond$ in the Harris system stationed at $a^n_\square$, separated from its central disk by $k_n := \lfloor \eta_H \log n \rfloor$ Harris rings with $\eta_H > 0$ absolute, and at Carath\'eodory distance at least $\eta_R/4$ in $D^n_\square$ from the other three marked points, such that $\dist(a^n_\diamondsuit, \partial D^n_\square) > n^{-a_5}$ and, for some absolute $\gamma,a_6>0$ and some $C_4=C_4(R,R_1)<\infty$,
\begin{align*}
\bigl|S^n_\square(a^n_\square) -S^n_\square(a^n_\diamondsuit)\bigr| &\le C_4(R,R_1)n^{-\gamma}, &&\text{\rm(link (2))},\\
\bigl|H^n_\vardiamond(a^n_\diamondsuit) -H^n_\square(a^n_\diamondsuit)\bigr| &\le C_4(R,R_1)n^{-a_6}, &&\text{\rm(link (5))},\\
\bigl|H^n_\square(a^n_\diamondsuit) -H^n_\square(a^n_\square)\bigr| &\le C_4(R,R_1)n^{-\gamma}, &&\text{\rm(link (6))}.
\end{align*}
\end{lemma}

\emph{Uniformity of (2), (5), (6).} Lemma~\ref{lem:macroRSW} supplies uniformly the order-one macroscopic connection costs in the cited BCL arguments. The rings of the Harris system \cite[Theorem 3.10]{BCL} adapt to the boundary, but two of their properties are uniform. First, each ring carries a circuit separating its station from the central disk $D_\Delta$ with probability at least $\vartheta_H$, independently over rings. Second, the width of a ring and the diameters of its two segments agree up to a fixed factor \cite[Theorem 3.10(2)]{BCL}; as consecutive rings share a segment, their widths differ by at most a fixed factor $M$. Both $\vartheta_H$ and $M$ depend only on the RSW constants of the model \cite[Theorem~3.10]{CL06}. We take $D_\Delta:=B(w_0,\Delta)$ with $2\Delta:=\dist(w_0,\partial D^n_\square)$, as the construction of \cite[\S 5]{BCL} uses only that $B(w_0,2\Delta)\subseteq D^n_\square$; then $\Delta \ge R/16$ for $n \ge n(R)$, by Koebe's theorem. In the notation of \cite[\S 5.4]{BCL}, the first segment is the component of $\partial S_0\cap D^n_\square$ through the point where $\partial S_0$ touches $D_\Delta$; it joins $\partial D_\Delta$ to $\partial B(w_0,2\Delta)$, so the first ring has width at least $c_0\Delta$, with $c_0>0$ depending only on the RSW constants. The construction stops after $\Gamma\log n$ rings or when the last segment connects to the station $\omega$ with probability at least $\vartheta_H$. To control the latter alternative, write $T_j$ for the modified segment, $Y_j$ for its unmodified separator, and $b_j=c_bJ_j$ for its box size. The full box layer of \cite[Remark~5.8 and \S5.4, induction hypothesis~0(ii)]{BCL} separates $T_j$ from $Y_j$, which separates it from $\omega$. Every lattice path from $\omega$ to $T_j$ therefore contains a subpath spanning that layer, of diameter at least $b_j-Ch$. Here $h=1/n$; the boundary incidences cannot bypass the layer because paths otherwise stay in the domain. For $b_j\ge 2Ch$, such a path reaches distance $b_j/4$ from $\omega$. RSW circuits of the opposite colour in disjoint annuli about $\omega$ give
\[
\IP(\omega\longleftrightarrow T_j)\le C(h/b_j)^{\zeta},\qquad \zeta>0,
\]
uniformly in the domain. Thus the probability stopping rule requires $J_j\le C_Hh$. The same one-arm bound starts the construction uniformly, since the initial segment is at distance at least $\Delta$ from $\omega$. Set $c_H := 1/\log M$ and take $\eta_H<\min\{\Gamma,\ c_H a_3,\ a_4/V\}$, where $e^{-V}$ is the uniform cost of one ring in the constructions of the proofs of \cite[Lemma 3.12, Proposition 3.14]{BCL}. Then, for $n\ge n(R)$, the rings at depths $j\le k_n$ exist and have widths at least $c_0\Delta M^{-j}\ge c_0\Delta\, n^{-\eta_H/c_H}$. Choose $a^n_\diamondsuit$ as in \cite[Definition 3.13]{BCL}, at the center of a full box of the percolating component of the ring at depth $k_n$. This ring has width $J\ge c_0\Delta\, n^{-\eta_H/c_H}$, and its full boxes have side $c_1J$, with $c_1>0$ depending only on the RSW constants, and do not meet $\partial D^n_\square$ \cite[Theorem 3.10(1)]{BCL}. Hence $\dist(a^n_\diamondsuit,\partial D^n_\square)\ge c_1J/2\ge (c_0c_1/32)\,R\,n^{-\eta_H/c_H}>n^{-a_5}$ for $n\ge n(R)$, since $\eta_H/c_H<a_3<a_5$, and Lemma \ref{ci} applies at $a^n_\diamondsuit$.

The $k_n$ rings between $a^n_\diamondsuit$ and $D_\Delta$ surround both $a^n_\square$ and $a^n_\diamondsuit$. A circuit in one of them makes the events defining $S^n_\square$ at the two points coincide if the ring is a conduit, that is, if its segments run from $(d^n_\square,a^n_\square)$ to $(a^n_\square,b^n_\square)$. By the hypothesis of the lemma, every segment of diameter less than $\eta_R/2$ is a conduit, as it separates $a^n_\square$ from $D_\Delta\ni w_0$. A ring with a segment of diameter at least $\eta_R/4$ has width at least $c\,\eta_R$ \cite[Theorem 3.10(2)]{BCL}, and so contains a full box of side at least $c_1c\,\eta_R$; as the rings are disjoint and $\diam D^n_\square\le R_1$, there are at most $C(R,R_1)$ such rings. Hence, outside an event of probability at most $(1-\vartheta_H)^{k_n-C(R,R_1)} \le C(R,R_1)\, n^{-\gamma}$, $\gamma := \eta_H \log\frac{1}{1-\vartheta_H}$, some conduit among the $k_n$ rings carries a separating circuit; this is link (2). As the crossing estimates hold uniformly at finer meshes, the same count gives link (6), as in the proof of \cite[Proposition 3.14]{BCL}. Some ring among the first $k_n-1$ also has segments of diameter less than $\eta_R/4$, separating $a^n_\diamondsuit$ and $a^n_\square$ from $w_0$; the triangle inequality for $\dist_{\mathrm C}$ gives the separation $\eta_R/4$ in the statement. Link (5) is \cite[Lemma 3.18]{BCL}, with exponent $a_6$ determined by $a_4$, $\beta$ and the ring counts of \cite[Proposition 3.11, Lemma 3.12]{BCL}. Its constant $\Upsilon$ counts Harris segments that cut off two marked points from $D_\Delta$; these have diameter at least $\eta_R/2$ by the separation stated after \eqref{eq:sep}, so the same count gives $\Upsilon\le C(R,R_1)$. Finally, as $\eta_H V<a_4$, the last assertion of \cite[Proposition 3.14]{BCL} gives $F^n_\square(a^n_\diamondsuit)\in T_\vardiamond$; its proof applies Lemma \ref{ci} at $a^n_\diamondsuit$, which the depth bound permits. The full boxes of the rings at depths $1,\dots,k_n$ have sides at least $s_n:=c_0c_1\Delta\,n^{-\eta_H/c_H}$. The percolating ones form a chain from the box centered at $a^n_\diamondsuit$ to a box meeting $D_\Delta$, in which consecutive boxes share an edge \cite[Theorem 3.10(1), Remarks 5.6, 5.12]{BCL}. Two closed squares in $D^n_\square$ sharing an edge of the smaller one contain the path joining their centers through the midpoint of that edge, together with its neighbourhood of radius a third of the smaller side. Follow these paths, and then segments through a point of $\overline{D_\Delta}$ in the last box to $w_0$. By convexity of the last box and the bound $\dist(\overline{D_\Delta},\partial D^n_\square)\ge\Delta$, the resulting path has boundary clearance at least $\min\{s_n,\Delta\}/3$. This exceeds $2n^{-a_3}$ for $n\ge n(R)$, since $\eta_H<c_Ha_3$. So $a^n_\diamondsuit$ lies inside the contour $\Gamma^n$ of link (4), where the preimage of $T_\vardiamond$ is $D^n_\vardiamond$.

\emph{Assembly.} Concatenating (1)--(7), each an inverse power of $n$ with absolute exponent and constant depending on $R$ and $R_1$,
\[
|S^n(a^n) - H(a)| \le C(R,R_1)\, n^{-c}, \qquad c := \min\{a_2, \beta, a_4, \gamma, a_6\} > 0 \ \text{absolute},
\]
with $C(R,R_1)$ collecting $C_2, C_3, C_4$. Fix any absolute $s \in (0,c)$ and absorb $C(R,R_1)$ into a threshold.

\subsection{Verification of Definition \ref{def:family}}\label{subsec:percverify}

We now take $\varsigma=0$ and use the usual hexagon-by-hexagon exploration of critical site percolation. We verify that it is $6$-admissible in the sense of Definition~\ref{def:family}. Throughout, $a,c$ are the boundary prime ends at which the interface starts and terminates, while $b,d$ are the two auxiliary boundary prime ends, listed in counterclockwise order as $a,b,c,d$. We uniformize by $ \varphi:(D,a,c)\longrightarrow(\IH,0,\infty). $ In the notation of Definition~\ref{def:family}, the configuration data are $m=0$ and $l=2$, corresponding to the images of $b$ and $d$; the target $\infty$ is not a coordinate of ${\mathcal X}^{0,2}$.

\emph{Domain Markov property and uniform KS Condition.} Proposition~\ref{admiss} gives the domain Markov property in the full decorated state at vertex stopping times. Proposition~\ref{percks} gives the uniform KS Condition, with a modulus depending only on the model parameters. It remains to verify weak $6$-admissibility.

\emph{The continuum observable.} Write $b^*:=\varphi(b)$ and $d^*:=\varphi(d)$. Let $H_{(\IH;b^*,\infty,d^*)}: \IH\longrightarrow\overline T $ be the conformal map sending $(b^*,\infty,d^*)$ to $(1,\tau,\tau^2)$, where $\tau=e^{2\pi i/3}$, and define \begin{equation}\label{eq:hperc} \mathfrak{h}(b^*,d^*) := H_{(\IH;b^*,\infty,d^*)}(0). \end{equation} Since $\varphi(c) = \infty$, the map $H_{(\IH;b^*,\infty,d^*)} \circ \varphi$ sends $(b,c,d)$ to $(1,\tau,\tau^2)$, so it is the map $H$ of Theorem~\ref{Obscvg}, and $\mathfrak h(b^*,d^*) = H(a)$ is the limit of the discrete observable $S^n(a^n)$ with the normalization of Section~\ref{percolation}.
We take \begin{equation}\label{eq:Vperc} {\mathcal V} := \left\{ (b^*,d^*)\in\IR^2: 1<b^*<2,\quad -2<d^*<-1 \right\} \subset{\mathcal X}^{0,2}. \end{equation} Its closure is compact in ${\mathcal X}^{0,2}$, and the three points $0,b^*,d^*$ remain uniformly separated on $\overline{\mathcal V}$. In particular, $\mathfrak{h}\in C^\infty({\mathcal V})$.

\begin{lemma}[Local exploration in the principal component]\label{lem:localexploration}
Consider critical site percolation. Let $\sigma\le\sigma'$ be bounded stopping times, $G\in\mathcal F^n_\sigma$ an event on which both endpoints are vertex times, and $b,d$ be $\mathcal F^n_\sigma$-measurable auxiliary prime ends on $G$. Suppose, almost surely on $G$, that
$$
\Crad_t^n\ge r>0,\qquad \Xi_t^n(b,d)\in\mathcal V
\quad(\sigma\le t\le\sigma').
$$
For $n\ge N(r,R_1)$, the marked representatives on $U_\sigma^n$ can be chosen measurably so that, on $G$, at every vertex time in this interval the same $b^n_{U,\sigma},c^n_{U,\sigma},d^n_{U,\sigma}$ and the effective entrance $a^n_{U,t}$ make $U_t^n$ a $C/n$-discrete approximation of $D_t^n$. Moreover, $U_t^n$ is the unexamined target component of the exploration restarted in $U_\sigma^n$ from $a^n_{U,\sigma}$ to $c^n_{U,\sigma}$.
\end{lemma}

\begin{proof}
Put $h=n^{-1}$, allowing $C$ to increase. By \eqref{eq:sep}, no $Ch$ crosscut cuts two of $\gamma^n(t),b,c,d$ off from $u_t^n$. The exploration-slit argument gives the unmarked approximation conditions. All incidences retain their prime-end sides. By \eqref{eq:basepoint}, $\dist(u_t^n,\partial D_t^n)\ge r/2$. Since $\partial U_\sigma^n$ lies within $Ch$ of $\partial D_\sigma^n$, continuity and the decrease of $\Omega_t^n$ give $u_t^n\in U_\sigma^n$ and $U_t^n\subseteq U_\sigma^n$.

Use boundary walks with separate incidences for the two slit sides. Both known colour clusters are connected to their original boundary arcs. A component meeting both clusters has exactly two colour changes: four alternating contacts, joined within the two known clusters, would give intersecting planar connections. Contract each such component to an oriented edge, keeping the untraced blue--yellow edges between them. Each intermediate junction has one incoming and one outgoing interface edge, with blue on the right; only the global tip and target have degree one. A cycle would separate a known colour cluster from its original arc, so the edges form the tip-to-target chain. Components meeting one colour lie off the chain. The component containing $u_t^n$ cannot be one of them: the boundary-walk connections would then give a $Ch$ crosscut cutting both global endpoints off from $u_t^n$, contrary to \eqref{eq:sep}. Its two colour changes are the entrance and exit. The same separation argument places $b,d$ on their prescribed arcs: a correspondence to the opposite arc cuts off the mark together with one endpoint.

At $\sigma$, fix these representatives, taking $c^n_{U,\sigma}$ at the actual exit. Choose $Ch$ crosscuts $\chi_x$, $x=b,c,d$, cutting $x$, its representative and the incident unexamined sites (the unique exit site for $c$) off from $u_\sigma^n$. Call the components cut off from $u_\sigma^n$ caps. A fixed enumeration makes the choices measurable.

The same clearance keeps the basepoint outside these caps. At first trace contact with $\chi_x$, its half facing the basepoint, with endpoint moved behind the tip on its slit side, cuts off $x$ and the tip, a contradiction. Suppose revealing a hexagon $H$ first removes a protected incident site or disconnects it from the basepoint. In the preceding unexamined component, the boundary of a small regular neighbourhood of $H$ has an arc separating the lost part from the basepoint. Follow the boundary of the complementary component containing the lost part, whether $H$ is deleted or closes a passage.

Extend the arc through adjacent old boundary hexagons to the corresponding geometric-boundary sides. Each extension has diameter $Ch$ and lies outside the preceding unexamined component. A path in that component from the basepoint to $q$ in the lost part retains odd intersection number with the extended arc. Adjoin a segment through the trace-free interior of $H$ to the current tip. Take a small regular neighbourhood, retaining its boundary incidences. On the frontier facing the basepoint, discard closed components; odd intersection selects an arc separating $q$; simple connectivity prevents an interior endpoint. Move the tip incidence behind the tip on its slit side, choosing the frontier with the tip on the side of $q$. The resulting crosscut has diameter $Ch$, lies in $D_t^n$ apart from its endpoints, and cuts off both $q$ and the tip. Choose $q$ inside the lost incident site and off the trace, including when the site is $H$. The original cap persists until this first loss and contains $x$ and $q$.

The two caps, in $D_t^n$ with basepoint $u_t^n$, contain a common interior point $q$ of a lost incident site. The triangle inequality gives
$$
\dist_C^{D_t^n}(x,\gamma^n(t))\le\dist_C^{D_t^n}(x,q)+\dist_C^{D_t^n}(q,\gamma^n(t))\le Ch,
$$
contradicting \eqref{eq:sep}. Thus the incidences and correspondences persist.

The exit site retains its oppositely coloured neighbours, so the exit is unchanged. By the chain structure, the global explorer examines no site of $U_\sigma^n$ before entering it, then follows its ordinary exploration toward this exit. Nesting identifies its target component as $U_t^n$; the retained correspondences and effective entrance complete the approximation.
\end{proof}

\emph{The observable and almost-martingale estimate.}
Let $H_t^n(b,d)$ be the value at $\gamma^n(t)$ of the conformal map
from $D_t^n$ to $T$ sending $(b,c,d)$ to $(1,\tau,\tau^2)$.
This defines a bounded observable on the discrete marked configurations;
set it to zero when the marks are not in the required cyclic order.
$$H_t^n(b,d)=\mathfrak h\bigl(\Xi_t^n(b,d)\bigr)$$
whenever $\Xi_t^n(b,d)\in\mathcal V$. Thus condition~\ref{cond:2b}
holds with zero error, and measurability follows from that of $\Xi_t^n$.

Let $s_0>0$ be the exponent in Theorem~\ref{Obscvg}, and let $\sigma,\sigma',G,b,d$ satisfy Lemma~\ref{lem:localexploration}. Select its marked
representatives at time $\sigma$ measurably, and fix the blue crossing
$E$ in $U_\sigma^n$ from $(a^n_{U,\sigma},b^n_{U,\sigma})$ to
$(c^n_{U,\sigma},d^n_{U,\sigma})$. Lemma~\ref{lem:localexploration} and the domain
Markov identity \cite[\S2.4, equations~(3)--(5) and Remark~2.7]{SLEI}
identify the martingale $M$ of Lemma~\ref{lem:mgo} with the fresh CCS
observable on $U_t^n$ using the retained representatives. Theorem
\ref{Obscvg}, with geometric domain $D_t^n$, discrete domain $U_t^n$
and radius bound $r/2$, gives, almost surely on $G$,
$$|M_\sigma-H_\sigma^n|+|M_{\sigma'}-H_{\sigma'}^n|
\le2C(r,R_1)n^{-s_0}.$$
Optional sampling therefore yields, almost surely on $G$,
$$\left|H_\sigma^n-
\IE[\mathbf1_GH_{\sigma'}^n\mid{\mathcal F}^n_\sigma]\right|
\le2C(r,R_1)n^{-s_0}.$$
Choose $s\in(0,\min\{s_0,1\})$ and absorb the constant and fixed mesh
factor into a common threshold $n_0=n_1$, nonincreasing in its radius
argument. This proves condition~\ref{cond:1}.

\emph{The BPZ equation and non-degeneracy.} By the Schwarz--Christoffel representation, \[ \mathfrak{h}=A_0+B_0\,G(u), \qquad u:=\frac{b^*}{d^*}, \] where $G$ is Cardy's function, so that $G$ satisfies the hypergeometric equation $u(1-u)G'' + \tfrac23(1-2u)G' = 0$. Substituting $\mathfrak{h} = A_0 + B_0\,G(u)$ into \eqref{eq:BPZ} and using this equation to eliminate $G''$ gives $A(\mathfrak{h}) + 6\,C(\mathfrak{h}) = 0$, which is the BPZ equation with $\kappa = 6$. Moreover, \[ B(\mathfrak{h}) = B_0\,G'(u)\,\Dop u \text{ and } C(\mathfrak{h}) = \frac{B_0}{2} \left[ G''(u)(\Dop u)^2 +G'(u)\Dop^2u \right]. \] Consequently, the constant $B_0$ cancels and \begin{equation}\label{eq:CBratio} \frac{C(\mathfrak{h})}{B(\mathfrak{h})} = \frac12 \left( \frac{G''(u)}{G'(u)}\,\Dop u +\frac{\Dop^2u}{\Dop u} \right). \end{equation} On ${\mathcal V}$ one has $b^*>0>d^*$ and hence $u<0$. Since, after absorbing a constant into $B_0$, $$ G'(u) = u^{-2/3}(1-u)^{-2/3} $$ is finite and nonzero there, with the branch inherited from the Schwarz--Christoffel map, and $\Dop u = (d^*-b^*)/(d^*)^2 \ne 0$ on ${\mathcal V}$, we have $B(\mathfrak{h})\ne0$ on ${\mathcal V}$. The ratio $G''/G'$ is branch-independent, so \eqref{eq:CBratio} is unaffected by this choice. Evaluating \eqref{eq:CBratio} at \[ \mathfrak v^{(1)} := \left(\frac65,-\frac95\right), \qquad \mathfrak v^{(2)} := \left(\frac95,-\frac{11}{10}\right), \] for which $u=-2/3$ and $u=-18/11$, respectively, gives $-\frac5{54}$ and $\frac{35}{297}$. The two values are distinct. Hence Remark~\ref{rem:nondeg}(a) implies that $M(\mathfrak{h})(\mathfrak v^{(1)},\mathfrak v^{(2)})$ has rank $2$, verifying condition~\ref{cond:2a}.

For the flower model, choose the fixed crossing event on an inner approximation with protected target arcs; RSW and uniformly positive compatible completions of the boundary flowers give the same polynomial comparison with Cardy's formula, uniformly over admissible exploration histories. The optional-sampling argument therefore verifies weak $6$-admissibility for this model as well; the uniform KS Condition then gives $6$-admissibility for both models, and Theorem~\ref{mainthm} applies with $\kappa=6$.

\begin{theorem}\label{percmain}
In the setting of Theorem \ref{mainthm}, let $\gamma^n$ be the Exploration Process of critical site percolation on the triangular lattice (or of the generalized bond-triangular/flower model with boundary Minkowski dimension less than $2$) on the admissible domain $D^n$, and $\tilde\gamma^n = \varphi^n(\gamma^n)$ its image in $(\IH, 0, \infty)$, parameterized by capacity. Let $T \ge 1$, $\rho \in (0,1)$, and $T_\rho := T \wedge \inf\{ t : d_*(\tilde\gamma(t), \infty) \le \rho \}$, where $\tilde\gamma$ is the chordal SLE$_6$ driven by $\sqrt{6}B(t)$. Then there exist $u > 0$ and $N < \infty$, with the dependence stated in Theorem \ref{mainthm}, and a coupling of $\tilde\gamma^n$ with $\sqrt{6}B$, such that for $n \ge N$,
\[
\IP\Bigl( \sup_{t \in [0, T_\rho]} d_*\bigl(\tilde\gamma^n(t), \tilde\gamma(t)\bigr) > n^{-u} \Bigr) < n^{-u},
\]
and, if $D$ is an $\alpha$-H\"older domain, the path estimate of Theorem \ref{mainthm} with exponent $\lambda = \tfrac12\min\{\beta_{\mathrm{map}}, \alpha u\}$, $\beta_{\mathrm{map}} = \alpha^2/100$, and the threshold there.
\end{theorem}

\begin{remark}\label{rem:percscope}
For the triangular-lattice site percolation, all clauses of Definition \ref{def:family} are verified above with no dimension hypothesis. For the generalized bond-triangular/flower model, the verification of Theorem \ref{Obscvg} (hence of condition \ref{cond:1} of Definition \ref{def:family}) additionally requires the boundary Minkowski dimension to be less than $2$, which is therefore included in the hypotheses of Theorem \ref{percmain} for that model.
\end{remark}

\section{Appendix: The Radial Setting}\label{app:radial}

This appendix gives the Radial analogue of the Main Theorem, Theorem \ref{mainthmradial}, with the radial BPZ equation and the stopping required by the SLE derivative estimate.

\emph{Curves and domains.} For a boundary prime end $a$ and an interior point $b \in D$ (Radial Loewner evolution), we take $\mathfrak p: (D;a,b) \to (\ID;-1,0)$; this map is uniquely determined by $\mathfrak p(a)=-1$, $\mathfrak p(b)=0$. The discrete approximation of $(D,a,b)$ is as in Definition \ref{def:approx}, conditions 2--4, together with $b \in D^n$, $a^n$ within $1/n$ of $a$ in the Carath\'eodory distance, and an interior vertex $b^n\in D^n$ with $|b^n-b|\le1/n$; the corresponding map $\mathfrak p^n: D^n \to \ID$ satisfies $\mathfrak p^n(a^n) = -1$, $\mathfrak p^n(b^n) = 0$, which determine it uniquely. Set $\varphi=\mathfrak p$, $\varphi^n=\mathfrak p^n$ and $\tilde\gamma^n=\mathfrak p^n(\gamma^n)$, parameterized by radial capacity. The KS Condition (Definition \ref{cond:KS}) is defined identically in this setting, and the diameter bound of Lemma \ref{DiamKt} holds with a constant $C$: for $K_t \subset \overline\ID$ the Radial hull of a continuous driving function $W:[0,\infty)\to\partial\ID$ and $k(t)$ as in that lemma,
\[
C^{-1}\min\{k(t),1\} \;\le\; \diam(K_t) \;\le\; C\, k(t), \qquad t \ge 0
\]
(\cite[Lemma 2.1]{LSW}, which covers both cases).

\emph{Loewner evolution and SLE.} A \emph{Radial Loewner pair} or \emph{$\ID$-Loewner pair} $(f,W)$ consists of $f_t(z)=f(t,z)$, $t>0$, $z\in\ID$, and a continuous $W:\IR_+\to\IT$ satisfying
\begin{equation}\label{rrLE}
\partial_t f_t(z) \;=\; z\, f_t'(z)\, \frac{z+W(t)}{z-W(t)}, \qquad f_0(z)\equiv z;
\end{equation}
$f_t$ maps $\ID$ onto $D_t := \ID\setminus K_t$ for a growing family of compact hulls $(K_t)$, normalized by $f_t(0)=0$, $f_t'(0)=e^{-t}$ (so $t$ is minus the log conformal radius of $D_t$ at $0$). \emph{Radial Schramm Loewner Evolution}, Radial SLE$_\kappa$, is the corresponding process with $W(t) = -e^{i\sqrt\kappa B_t}$, so that $W(0) = -1$ agrees with the normalization $\mathfrak p(a) = -1$ fixed above; in a domain $D$ with boundary prime end $a$ and interior point $b$, it is the image of SLE$_\kappa$ in $(\ID,-1,0)$ under the conformal map $\psi:(\ID,-1,0)\to(D,a,b)$. The trace trichotomy of Theorem~\ref{Rohde-Schramm} holds in this setting as well \cite{MR2883396,LSW}. The derivative estimate of Proposition~\ref{derest1} has a Radial counterpart, but only on a stopped interval. Write $g_t := f_t^{-1}$ and, for $\varepsilon \in (0,1)$ and $T<\infty$, set
\[
\sigma_\varepsilon \ :=\ \inf\bigl\{\, t \ge 0 \,:\, \bigl| g_t(1) - W(t) \bigr| \le \varepsilon \,\bigr\} \wedge T ,
\]
the first time the hull comes within $\varepsilon$ of disconnecting the antipode $1$ of the starting prime end from the target. By \cite[Proposition~A.5]{V}, the bound of Proposition~\ref{derest1} holds in radial form, with $d\,|f_t'((1-d)W(t))|$ in place of $y\,|f'(t,W(t)+iy)|$ and $\sup_{t\in[0,\sigma_\varepsilon]}$ in place of $\sup_{t\in[0,T]}$, the same exponents $\beta_+$ and $q(\beta)$, and a constant depending also on $\varepsilon$. Writing $\theta_t:=\arg(g_t(1)\overline{W(t)})$, one has $d\theta_t=\cot(\theta_t/2)\,dt-\sqrt\kappa\,dB_t$ and $\theta_0=\pi$. For $\kappa<4$, \cite{V} gives $\sigma_\varepsilon\uparrow T$ almost surely as $\varepsilon\downarrow0$; for $\kappa>4$, disconnection can occur in finite time. We make no unstopped assertion for $\kappa=4$. Thus Theorem~\ref{mainthmradial} is stated on $[0,\sigma_\varepsilon]$; no rate for $\IP(\sigma_\varepsilon<T)$ is used. The tip structure modulus (Definition \ref{def:tip}) is defined identically, with $\ID$ in place of $\IH$ and the interior target $0$ in place of $\infty$.

\emph{Admissibility.} Weak $\kappa$-admissibility and $\kappa$-admissibility (Definition \ref{def:family}) extend to this setting with the endpoint $b$ interior: $\varphi: D \to \ID$ satisfies $\varphi(a)=-1$, $\varphi(b)=0$, as above, and $\Crad(\varphi) := |(\varphi^{-1})'(0)|$. The centering \eqref{eq:centered} by subtraction is replaced by the rotation $z \mapsto z/W^n_t$, which sends the driving value to $1$; accordingly the configuration space is
\[
{\mathcal X}^{m,l}_{\mathrm{rad}} := \bigl\{ (z^1,\dots,z^m; c^1,\dots,c^l) \in \ID^m\times(\IT\setminus\{1\})^l : z^j\neq z^k \text{ and } c^j\neq c^k \text{ for } j\neq k \bigr\}.
\]
Since $W^n_0 = -1$, the centered configuration at $t = 0$ is $-\hat\varphi^n$. Conditions \ref{cond:1} and \ref{cond:2b} are otherwise as in the Chordal case; condition \ref{cond:2a} requires $\mathfrak{h}$ to solve the radial BPZ equation \eqref{eq:radBPZ}, derived below. Apply the KS implications of Section~\ref{sec:path} componentwise: only a component whose crossing is necessary to reach the interior target is forced; other components of the same annulus remain subject to the unforced-crossing bound. The successive unforced-crossing argument of \cite[Proposition~2.6]{KS} gives the same G3 exponent. For H1 on $[0,T]$, the trace avoids $B(0,\rho_T)$, where $\rho_T=e^{-T}/4$ as shown below. Consider $A(z,r,R)$ with $R\ge2r$ that is crossed by the trace. Since the trace lies in $\overline\ID$, $R\le r+2\le R/2+2$, so $R\le4$. If $r\ge\rho_T/4$, the ratio $R/r$ is bounded in terms of $T$, and H1 follows by increasing $K_\nu$. Otherwise $|z|\ge3\rho_T/4$. Replace $R$ by $R':=\min\{R,|z|/2\}>r$. The smaller annulus excludes $0$, every original crossing contains a crossing of it, and $R/R'\le C(T)$. The forced/unforced crossing argument of \cite[Proposition~3.6]{KS} applies to this annulus, with the interior target outside it; bounded ratios $R'/r$ are again absorbed in $K_\nu$. This gives the same exponents, with constants and thresholds depending on $T$.

For radial loop-erased random walk, the radial extension of the KS verification is stated at the beginning of \cite[\S4.5]{KS}, referring to \cite{LSW}. The discrete Poisson-kernel ratio of \cite[Theorem~4.1 and \S5]{BJVK} is an exact martingale and, after rescaling the mesh, is polynomially close on the conformal-radius sublevels used above to
\[
\mathfrak h_{\rm LERW}(\eta)=\frac{1-|\eta|^2}{|1-\eta|^2}.
\]
This function satisfies the radial BPZ equation with $\kappa=2$. Moreover $B_{\rm rad}(\mathfrak h_{\rm LERW})\ne0$ at $\eta=i/2,i/3$, and direct substitution gives $C_{\rm rad}/B_{\rm rad}=-4/5$ and $-3/5$, respectively. Thus the radial admissibility conditions are satisfied for loop-erased random walk, justifying the claim in the introduction.

\emph{Transfer of the chordal estimates.}
The path estimate transfers as stated. For the tip estimate, let $\ID_t=\ID\setminus K_t$ be the normalized radial slit domain. Since $f_t(0)=0$ and $|f_t'(0)|=e^{-t}$, Koebe's theorem gives
\[
B(0,\rho_T)\subset\ID_t,\qquad \rho_T:=e^{-T}/4,\qquad t\le T.
\]
A crosscut of $\ID_t$ of diameter at most $\delta$ that separates the tip from $0$ has its endpoints outside $B(0,\rho_T)$, so, as $\delta<\rho_T/2$, it separates the tip from $B(0,\rho_T/2)$. The proof of Proposition \ref{tip} therefore applies with the target $1$ and its target bulk $B(1,\rho)\cap\ID$ replaced by $0$ and $B(0,\rho_T/2)$, respectively, and with ``not containing $1$'' replaced by ``not containing $0$''. It gives the same exponents $\alpha=1/2$ and $\epsilon=9/\Delta_{\mathrm{KS}}$; constants and thresholds may depend on $T$, in particular through $\eta\le\rho_T/64$.

For driving convergence, set $\Crad^n_t:=\bigl|\bigl((\varphi^n_t)^{-1}\bigr)'(0)\bigr|$, which is unchanged by the centering rotation. The counterpart of Lemma~\ref{lem:discreteCrad} is simpler because $\mathfrak p^n(b^n)=0$.

\begin{lemma}\label{lem:discreteCradRad}
Let $(D^n,a^n,b^n)$ be a discrete approximation of $(D,a,b)$ in the Radial sense above, and $\varphi^n=\mathfrak p^n$. There exists $N_{\rm rad}=N_{\rm rad}(R)$ such that, whenever $\Crad(\varphi)>R$, then, for all $n\ge N_{\rm rad}$,
\[
\Crad(\varphi^n)\ge\frac R6 .
\]
\end{lemma}

\begin{proof}
Since $\mathfrak p^{-1}(0)=b$ and $\Crad(\varphi)=|(\mathfrak p^{-1})'(0)|>R$, Koebe's theorem gives $\dist(b,\partial D)>R/4$. By $|b^n-b|\le1/n$ and Definition~\ref{def:approx}(4),
\[
\dist(b^n,\partial D^n)\ge\frac R4-\frac2n\ge\frac R6
\]
once $n\ge N_{\rm rad}:=24/R$. Writing $\psi_n=(\mathfrak p^n)^{-1}$, we have $\psi_n(0)=b^n$, so no normalization error arises and Schwarz's lemma gives $\Crad(\varphi^n)=|\psi_n'(0)|\ge\dist(b^n,\partial D^n)\ge R/6$.
\end{proof}

For the H\"older-domain transfer, the endpoint-normalization argument in Remark~\ref{rem:discretemaps} is replaced by the following radial estimate. Let $d:=\dist(b^n,\partial D^n)\ge R/6$ and $q_n:=\mathfrak p^n(b)$. Cauchy's estimate for $(\mathfrak p^n)'$ on $B(b^n,d)$ and $|b-b^n|\le1/n$ give $|q_n|\le 2|b-b^n|/d\le 12/(Rn)$. The remaining rotation is fixed by $a^n$ and differs from the identity by the same $O(n^{-\alpha/50})$ boundary error as in Remark~\ref{rem:discretemaps}. Hence the proof of Lemma~\ref{lem:discrete} gives the same exponent $\beta_{\mathrm{map}}=\alpha^2/100$ in the radial normalization.

By Lemma~\ref{lem:discreteCradRad}, $\Crad^n_0=\Crad(\varphi^n)\ge R/6$ for $n\ge N_{\rm rad}$, and since $(f^n_t)'(0)=e^{-t}$,
\[
\Crad^n_t=e^{-t}\Crad^n_0\ge e^{-T}R/6,\qquad t\le T,
\]
so the Base-Point Estimate and horizon comparison are unnecessary. The stopping at $\sigma_\varepsilon$ enters through the derivative event $B_n$ below.

\emph{The radial BPZ equation.} We first compute the evolution of one centered marked point.

Let $g_t: \ID\setminus K_t \to \ID$ solve the inverse Radial Loewner equation, obtained from \eqref{rrLE} by differentiating $f_t(g_t(z)) = z$:
\[
\partial_t g_t(z) = g_t(z)\, \frac{g_t(z) + W_t}{W_t - g_t(z)},
\]
and let $\eta_t := g_t(z)/W_t$ be the centered (rotated) coordinate, $W_t = -e^{i\sqrt\kappa B_t}$. Since $W_t$ is an explicit function of the real Brownian motion $B_t$, It\^o's formula gives
\[
dW_t = i\sqrt\kappa\, W_t\, dB_t - \tfrac{\kappa}{2}\, W_t\, dt, \qquad d(W_t^{-1}) = -i\sqrt\kappa\, W_t^{-1}\, dB_t - \tfrac{\kappa}{2}\, W_t^{-1}\, dt.
\]
Since $g_t(z)$ has finite variation, $d\langle g_t(z),W_t^{-1}\rangle=0$, and the product rule gives
\begin{equation}\label{eq:radSDE}
d\eta_t \;=\; \eta_t \left( \frac{1+\eta_t}{1-\eta_t} - \frac{\kappa}{2} \right) dt \;-\; i\sqrt\kappa\, \eta_t\, dB_t .
\end{equation}
For several marked points driven by the same $B_t$, the quadratic covariations are
\[
d\langle \eta^{(i)}, \eta^{(j)}\rangle_t \;=\; -\kappa\, \eta^{(i)}_t\, \eta^{(j)}_t\, dt,
\]
and $d\langle \eta^{(i)},\overline{\eta^{(j)}}\rangle_t=\kappa\eta^{(i)}_t\overline{\eta^{(j)}_t}\,dt$. Writing $\eta_j=r_je^{i\theta_j}$ and
\[
\Theta \ :=\ \sum_j \partial_{\theta_j}
\]
for the \emph{common-rotation} operator --- the Radial counterpart of the Chordal common-shift $\Dop$ of \eqref{eq:Dop} --- It\^o's formula applied to $F(\eta^{(1)}_t, \dots, \eta^{(k)}_t)$ and the martingale condition give the \emph{radial BPZ equation}
\begin{equation}\label{eq:radBPZ}
\sum_j \left( \eta_j\, \frac{1+\eta_j}{1-\eta_j}\, \partial_{\eta_j} F \;+\; \overline{\eta_j}\, \frac{1+\overline{\eta_j}}{1-\overline{\eta_j}}\, \partial_{\overline{\eta_j}} F \right) \;+\; \frac{\kappa}{2}\, \Theta^2 F \;=\; 0,
\end{equation}
valid for any $C^2$ observable; the $-\kappa/2$ drift of $\eta_t$ and the It\^o correction combine into $(\kappa/2)\Theta^2F$. For holomorphic $F$ one has $\Theta^2 F = -\mathcal D_{\mathrm{rad}}^2 F$, where $\mathcal D_{\mathrm{rad}} := \sum_j \eta_j \partial_{\eta_j}$ is the scaling (Euler) operator and $\mathcal D_{\mathrm{rad}}^2$ means it applied twice, so \eqref{eq:radBPZ} reduces to $\sum_j \eta_j \tfrac{1+\eta_j}{1-\eta_j}\, \partial_{\eta_j}F = \tfrac{\kappa}{2}\, \mathcal D_{\mathrm{rad}}^2 F$.
In the notation of \eqref{eq:BCdef}, set $B_{\mathrm{rad}}(F):=\Theta F$ and $C_{\mathrm{rad}}(F):=\tfrac12\Theta^2F$, and let $A_{\mathrm{rad}}(F)$ be the first-order part. Then \eqref{eq:radBPZ} reads $A_{\mathrm{rad}}(F)+\kappa C_{\mathrm{rad}}(F)=0$. Definition~\ref{def:nondeg} and Lemma~\ref{lem:nondegquant} apply to the pair $(B_{\mathrm{rad}}(F),C_{\mathrm{rad}}(F))$ as before.

\emph{The Radial Main Theorem.}

\begin{theorem}[Main Theorem, Radial setting]\label{mainthmradial}
Let $\{\gamma_{(D,a,b)}^n\}$ be a family of $\kappa$-admissible probability measures in the Radial sense above, for some $\kappa\in(0,8)$, with admissibility data verified using the radial BPZ equation \eqref{eq:radBPZ} in place of \eqref{eq:BPZ}. Fix $R>0$ and the chosen map $\varphi:(D,a,b)\to(\ID,-1,0)$ with $\Crad(\varphi)>R$ and $\diam(D)\le R_1<\infty$, let $T\ge1$ and $\varepsilon\in(0,1)$, and let $\sigma_\varepsilon$ be the stopping time defined above. Then the conclusion of Theorem \ref{mainthm} holds with $\tilde\gamma$ the Radial SLE$_\kappa$ driven by $-e^{i\sqrt{\kappa}B(t)}$ and $[0,\sigma_\varepsilon]$ in place of $[0,T_\rho]$:
\[
\IP\Bigl( \sup_{t\in[0,\sigma_\varepsilon]} |\tilde\gamma^n(t)-\tilde\gamma(t)| > n^{-u}\Bigr) < n^{-u},
\]
with $N$ depending on $R$ and on the diameter bound $R_1$ of the admissible class as in Theorem \ref{mainthm}, and in addition on $\varepsilon$; the exponent $u$ does not depend on $\varepsilon$. If $D$ is $\alpha$-H\"older, the H\"older-domain conclusion of Theorem \ref{mainthm} holds with the same exponent $\lambda = \frac12\min\{\beta_{\mathrm{map}}, \alpha u\}$, $\beta_{\mathrm{map}} = \alpha^2/100$.
\end{theorem}


\begin{proof}
The argument is that of Theorem \ref{mainthm}; we record the changes in each step.

Let $\xi_t$ be the unique continuous real lift with $W_t=-e^{i\xi_t}$ and $\xi_0=0$. Apply Section~\ref{sec:convdrive} with $\Delta\xi$ in place of $\Delta W$: the stopping rule \eqref{eq:mdef} measures angular increments, and the Taylor expansion in Proposition~\ref{keyest} is
$$e^{-i\Delta\xi}=1-i\Delta\xi-\tfrac12(\Delta\xi)^2+O(|\Delta\xi|^3),$$
with $B_{\mathrm{rad}}$ and $C_{\mathrm{rad}}$ as above. Lemma~\ref{lem:coupling} then couples $\xi$ to $\sqrt\kappa B$, giving the radial form of Theorem~\ref{drivingcvg}.

\emph{Steps~0, 3, and 5} are unchanged, with horizon $\sigma_\varepsilon$; the transfer to the physical domain uses Lemma~\ref{lem:discrete} and the radial normalization estimate above.

\emph{Step 1.} The event $A_n$ bounds $\sup|\xi^n(t)-\xi(t)|$. For $C_n$, use $\eta^{(n)}_{\mathrm{tip}}(\bar\delta_n)\le\bar\delta_n^{p_\star}$ for the radial modulus on $[0,T]$. The radial form of Corollary~\ref{cor:tip} above gives the same exponents and threshold $n_2(M_{\mathrm{KS}},T,s)$. For $B_n$, apply \cite[Proposition~A.5]{V} on $[0,\sigma_\varepsilon]$, giving $\IP(B_n^c)\le c_B'(\varepsilon)d_n^q$ with the same $\beta_+$ and $q(\beta)$. Here $c_B$, $c_B'$ and hence $C_0$ may also depend on $\varepsilon$; use these constants in Step~0 and the threshold for $C_n$.

For the radial version of Lemma~\ref{lem:onestep}, consider one edge. While $\Crad_t^n\ge r$, Koebe gives target clearance at least $r/4$. Beurling therefore bounds the harmonic measure from the target of the next edge by $C(r)n^{-1/2}$. In the disk, its hull has diameter $O_r(n^{-1/2})$ and contains its starting driving point. In a Cayley chart sending that point to $0$ and the target to $i$, Lemma~\ref{lem:hulldiam} and the half-plane capacity bound give angular oscillation $O_r(n^{-1/2})$ and radial time increment $O_r(n^{-1})$; normalization at the target changes the angle by $O_r(n^{-1})$. The estimates hold for every partial edge. The driving values stay in one proper circle arc, so the continuous real lift satisfies the same oscillation bound. Induction as in Lemma~\ref{lem:onestep}, using $\Crad_t^n=e^{-t}\Crad_0^n$, controls the stopping-index overshoot. The radial Taylor expansion gives Proposition~\ref{keyest} with the operators above.

Thus, if $k^+$ is the first index at which
$\Delta t\ge n^{-2s/3}$ or $|\Delta\xi|\ge n^{-s/3}$, then
\[
\sup_{t_k\le t\le t_{k^+}}|\xi(t)-\xi(t_k)|\le 2n^{-s/3},
\qquad
t_{k^+}-t_k\le 2n^{-2s/3},
\]
the conformal radius remains bounded below by a fixed multiple of its value at $t_k$. Every block starting before $T$ ends before $T+2$, so the radial radius bound gives $\Crad_t^n\ge e^{-(T+2)}R/6$ throughout the block. Condition~\ref{cond:1} therefore applies also on the last block, including its overshoot, and
\[
\left|\IE[\Delta\xi\mid\mathcal F_k]\right|\le Cn^{-s},
\qquad
\left|\IE[(\Delta\xi)^2-\kappa\Delta t\mid\mathcal F_k]\right|
\le Cn^{-s}.
\]
Hence the hypotheses of Lemma~\ref{lem:coupling} hold.

\emph{Step 2} (deterministic comparison) needs no Cayley comparison, as the radial modulus is defined in $\ID$. The $\ID$-Loewner form recorded in the final sentence of Lemma \ref{Vmain} is \cite[Lemma~3.4]{V}, with condition 3 at the scale $\bar\delta$. It applies to the radial chains driven by $W^n=-e^{i\xi^n}$ and $W=-e^{i\xi}$, since $|W^n-W|\le|\xi^n-\xi|$, pathwise on $[0,T']$ with $T'=\sigma_\varepsilon\le T$ and parameters $\epsilon:=\epsilon_n$, $L:=2$, $p$, $\beta$, $r:=p_\star$.

\emph{Step 4} is unnecessary, since $\Crad_t^n=e^{-t}\Crad_0^n$ stays uniformly positive on $[0,T]$.

\end{proof}


\end{document}